\documentclass[oneside, reqno]{amsart}

\usepackage[utf8]{inputenc}
\usepackage[T1]{fontenc}
\usepackage[dvipsnames]{xcolor}
\usepackage{amssymb}
\usepackage{amsmath}
\usepackage{amsthm}
\usepackage{caption}
\usepackage{microtype}
\usepackage{mathtools}
\usepackage{thmtools}
\usepackage{bbm}
\usepackage[bookmarks, bookmarksdepth=subsection, colorlinks, allcolors=Blue, final]{hyperref}
\usepackage{lipsum}
\usepackage[colorinlistoftodos, obeyDraft, textsize=tiny]{todonotes}
\usepackage{nicefrac}
\usepackage[mathscr]{euscript}

\usepackage{ifdraft}
\ifdraft{
    \usepackage[left]{showlabels}
}{}

\usepackage[shortlabels]{enumitem}
\setlist[itemize]{leftmargin=2\parindent}
\setlist{topsep=0ex, itemsep=0.4ex, parsep=0ex}
\setenumerate[0]{label=(\roman*), font=\normalfont}

\usepackage[capitalize, noabbrev, nameinlink]{cleveref}
\Crefname{equation}{}{}
\Crefname{enumi}{}{}
\Crefname{diagram}{Diagram}{Diagrams}
\crefformat{footnote}{#2\footnotemark[#1]#3}

\theoremstyle{plain}
\newtheorem{theorem}                {Theorem}     [section]

\newtheorem{lemma}        [theorem] {Lemma}
\newtheorem{proposition}  [theorem] {Proposition}
\newtheorem{corollary}    [theorem] {Corollary}
\newtheorem*{theorem*}              {Theorem}

\theoremstyle{definition}
\newtheorem{example}      [theorem] {Example}

\newtheorem{definition}   [theorem] {Definition}

\newtheorem*{definition*} {Definition}

\theoremstyle{remark}

\newtheorem*{remark*}               {Remark}

\usepackage[
    backend=biber,
    style=alphabetic,
    giveninits=true,
    maxbibnames=10,
    minbibnames=10,
    mincrossrefs=10,
    bibencoding=utf8]{biblatex}

\AtBeginBibliography{
    \raggedright
}

\AtEndDocument{%
    \printbibliography
}

\makeatletter
\renewcommand{\@mkboth}[2]{}
\makeatother

\usepackage{tikz}
\usetikzlibrary{cd}
\usetikzlibrary{decorations.markings}
\usetikzlibrary{calc}

\tikzcdset{
    diagrams={ampersand replacement=\&},
    display style/.style={
        cells={font=\everymath\expandafter{\the\everymath\displaystyle}},
    },
    postmark/.style={
        postaction={
            decorate,
            decoration={
                markings,
                mark={#1}
            }
        }
    },
    dagger arrow/.style={
        postmark=at position 0.5 with {
            \draw[-] (#1,1.5pt) -- (#1,-1.5pt);
        },
        labels={
            inner sep=0.75ex,
        }
    },
    mono/.style={tail},
    epi/.style={two heads},
    dagger mono/.style={dagger arrow={0pt}, mono},
    dagger epi/.style={dagger arrow={-1pt}, epi},
    independent/.style={phantom, "\perp"{font=\scriptsize}},
    coindependent/.style={phantom, "\perp"{font=\scriptsize}},
    left adjoint/.style={phantom, "\dashv"{font=\scriptsize, rotate=-90}},
    transformation/.style={
        Rightarrow,
        to path={
            let
                \p1 = ($(\tikztostart)!0.5!(\tikztotarget)$)
            in
                ($(\p1)!0.6em!(\tikztostart)$) -- ($(\p1)!0.6em!(\tikztotarget)$) \tikztonodes
        },
        every label/.append style={
            inner sep=0.8ex,
            pos=0.4
        }
    },
    transformation/.default={0.5}{90},
}

\tikzset{
    pullback/.style={
        label={[inner sep=0pt, anchor=center]south east:{\lrcorner}}
    }
}

\RequirePackage{tikz-cd}
\RequirePackage{amssymb}
\usetikzlibrary{calc}
\usetikzlibrary{decorations.pathmorphing}
\tikzset{curve/.style={settings={#1},to path={(\tikztostart)
    .. controls ($(\tikztostart)!\pv{pos}!(\tikztotarget)!\pv{height}!270:(\tikztotarget)$)
    and ($(\tikztostart)!1-\pv{pos}!(\tikztotarget)!\pv{height}!270:(\tikztotarget)$)
    .. (\tikztotarget)\tikztonodes}},
    settings/.code={\tikzset{quiver/.cd,#1}
        \def\pv##1{\pgfkeysvalueof{/tikz/quiver/##1}}},
    quiver/.cd,pos/.initial=0.35,height/.initial=0}
\tikzset{tail reversed/.code={\pgfsetarrowsstart{tikzcd to}}}
\tikzset{2tail/.code={\pgfsetarrowsstart{Implies[reversed]}}}
\tikzset{2tail reversed/.code={\pgfsetarrowsstart{Implies}}}
\tikzset{no body/.style={/tikz/dash pattern=on 0 off 1mm}}

\makeatletter
\newcommand{\fps@diagram}{tbp}
\newcounter{diagram}
\def\ftype@diagram{1}
\def\ext@diagram{lof}
\def\fnum@diagram{\diagramname\ \thediagram}
\def\diagramname{Diagram}
\newenvironment{diagram}{%
\@float{diagram}%
}{%
\end@float
}
\newenvironment{diagram*}{%
\@dblfloat{diagram}%
}{%
\end@dblfloat
}
\makeatother

\renewcommand{\leq}{\leqslant}

\DeclarePairedDelimiter{\paren}{(}{)}

\DeclarePairedDelimiter{\norm}{\|}{\|}
\DeclarePairedDelimiter{\ang}{\langle}{\rangle}
\DeclarePairedDelimiterXPP{\innerProd}[2]{\bgroup}{\langle}{\rangle}{\egroup}{#1 \delimsize\vert\mathopen{} #2}
\DeclarePairedDelimiterX{\setb}[2]{\lbrace}{\rbrace}{#1 \,\delimsize\vert\,\mathopen{} #2}
\DeclarePairedDelimiter{\set}{\lbrace}{\rbrace}

\newcommand{\cat}{\mathscr}

\newcommand{\C}{\cat{C}}
\newcommand{\D}{\cat{D}}

\newcommand{\pob}{\phantom{\cdot}}

\DeclareMathOperator{\Isom}{\cat{I}som}
\DeclareMathOperator{\Coisom}{\cat{C}oisom}
\DeclareMathOperator{\Map}{\cat{M}ap}
\DeclareMathOperator{\Rel}{\cat{R}el}
\DeclareMathOperator{\Set}{\cat{S}et}
\DeclareMathOperator{\Cat}{\cat{C}at}
\DeclareMathOperator{\Hilb}{\cat{H}ilb}
\DeclareMathOperator{\Mat}{\cat{M}at}

\newcommand{\FinProb}{\mathrm{\cat{F}in\cat{P}rob}_{\mathrm{sto}}}
\newcommand{\SBProb}{\mathrm{\cat{S}td\cat{B}or\cat{P}rob}_{\mathrm{sto}}}
\newcommand{\Prob}{\mathop{\mathrm{\cat{P}rob}_{\mathrm{sto}}}}

\newcommand{\FinProbDet}{\mathop{\mathrm{\cat{F}in\cat{P}rob}_{\mathrm{det}}}}
\newcommand{\SBProbDet}{\mathop{\mathrm{\cat{S}td\cat{B}or\cat{P}rob}_{\mathrm{det}}}}

\newcommand{\SBMeas}{\mathrm{\cat{S}td\cat{B}or\cat{M}eas}_{\mathrm{sto}}}

\DeclareMathOperator{\Inj}{\cat{I}nj}
\DeclareMathOperator{\Surj}{\cat{S}urj}

\DeclareMathOperator{\FinSet}{\cat{F}in\cat{S}et}

\DeclareMathOperator{\RegEpi}{\cat{R}eg\cat{E}pi}
\DeclareMathOperator{\RegMono}{\cat{R}eg\cat{M}ono}
\DeclareMathOperator{\MSurj}{\cat{MS}urj}

\DeclareMathOperator{\RegInd}{\cat{E}pi\cat{R}eg}

\DeclareMathOperator{\LexReg}{\cat{L}ex\cat{R}eg}
\DeclareMathOperator{\DilDag}{\cat{D}il}

\DeclareMathOperator{\TabAll}{\cat{T}ab\cat{A}ll}
\DeclareMathOperator{\Dag}{\cat{D}ag}
\DeclareMathOperator{\Ind}{\cat{I}nd}
\DeclareMathOperator{\Lex}{\cat{L}ex}

\DeclareMathOperator{\Ran}{Ran}

\DeclareMathOperator{\IndSq}{\mathbb{I}nd\mathbb{S}q}
\DeclareMathOperator{\Sq}{\mathbb{S}q}

\DeclareMathOperator{\ResIso}{\cat{R}es\cat{I}so}
\DeclareMathOperator{\Par}{\cat{P}ar}
\DeclareMathOperator{\PInj}{\cat{PI}nj}
\DeclareMathOperator{\FinPInj}{\cat{F}in\cat{PI}nj}
\DeclareMathOperator{\CRing}{\cat{CR}ing}

\DeclareMathOperator{\Heap}{\cat{H}eap}
\DeclareMathOperator{\Nom}{\cat{N}om}

\DeclareMathOperator{\Perm}{\mathrm{Sym}}

\DeclareMathOperator{\supp}{\mathrm{supp}}
\newcommand{\A}{\mathbb A}

\newcommand{\Reals}{\mathbb{R}}

\renewcommand{\Pr}{P}
\newcommand{\ind}{\mathbbm{1}}
\newcommand{\CPr}{\mathbb{P}}

\newcommand{\op}{\mathrm{op}}

\DeclareMathOperator{\Col}{Col}

\newcommand{\overbar}[1]{\,\overline{\!{#1}}}
\begin{document}

\title[Dagger categories of relations]{Dagger categories of relations\\[2ex]\scriptsize
The equivalence of dilatory dagger categories and\\epi-regular independence categories}
\date{\today}


\author{Matthew Di Meglio}
\address{University of Edinburgh}
\email{m.dimeglio@ed.ac.uk}

\author{Chris Heunen}
\email{chris.heunen@ed.ac.uk}

\author{Jean-Simon Pacaud Lemay}
\address{Macquarie University}
\email{js.lemay@mq.edu.au}

\author{Paolo Perrone}
\address{University of Oxford}
\email{paolo.perrone@cs.ox.ac.uk}

\author{Dario Stein}
\address{Radboud University Nijmegen}
\email{dario.stein@ru.nl}

\subjclass[2020]{18E08, 18M40, 60A05, 46M15, 18B10}

\begin{abstract}
    Several categories look like categories of relations, but do not fit the established theory of relations in regular categories. They include the category of surjective multivalued functions, the category of injective partial functions, the category of finite probability spaces and stochastic matrices, and the category of Hilbert spaces and linear contractions. To explain these anomalous examples, we develop a parallel theory of relations in \emph{epi-regular independence categories}. Just as regular categories correspond to tabular allegories, epi-regular independence categories correspond to \emph{dilatory dagger categories}. The equivalence maps epi-regular independence categories to their associated dagger category of relations, and dilatory dagger categories to their wide subcategory of coisometries.
\end{abstract}

\maketitle

\setcounter{tocdepth}{1}
\tableofcontents

\section{Introduction}

Many concepts in category theory are abstracted from categories of structure-preserving \textit{functions}, and so are not useful for describing categories of structure-preserving \textit{relations}. Several interrelated~\cite{knijnenburg:two-categories-relations} frameworks purportedly capture the essence of categories of relations: \textit{allegories}~\cite{freyd:allegories} (see also \cite[\S~A3]{johnstone:elephant}), \textit{bicategories of relations}~\cite{carboni:cartesian-bicategories-I}, and \textit{double categories of relations}~\cite{lambert:relations}. For all three kinds of categories, there is a simple characterisation of the \textit{maps} (the function-like morphisms), and an instance \(\D\) is canonically isomorphic to the category \(\Rel(\Map(\D))\) of \textit{relations} (isomorphism classes of jointly monic spans) in its wide subcategory \(\Map(\D)\) of maps exactly when every morphism in \(\D\) has a \textit{tabulator}\footnote{In the literature, the name \textit{tabulation} is also used for this concept. As this concept is analogous to our notion of \textit{dilator} (rather than our notion of \textit{dilation}), we will only use the term \textit{tabulator}.} (a suitably well-behaved representation as a span of maps). The tabulators in \(\D\) play two roles: they ensure that every morphism in \(\D\) is represented by a relation in \(\Map(\D)\), and they ensure that \(\Map(\D)\) is \textit{regular} so that composition of relations in \(\Map(\D)\) is defined and is associative~\cite{klein:relations-categories,kawahara:relations-categories-pullbacks,meisen:bicategories-relations-pullback,carboni:relazioni,cruciani:teoria-relazioni}. In fact, \(\Map\) and \(\Rel\) extend to an adjoint equivalence
\begin{equation}
    \label{e:equivalence-old}
    \begin{tikzcd}[column sep=large]
    \LexReg
        \ar[r, "\Rel", shift left=2]
        \ar[r, left adjoint]
        \&
    \TabAll
        \ar[l, "\Map", shift left=2]
    \end{tikzcd}
\end{equation}
between the categories \(\LexReg\)\footnote{\textit{Lex} is short for \textit{left exact}, which is synonymous with \textit{finitely complete}. The category of finitely complete categories is usually denoted \(\Lex\).} of finitely complete regular categories\footnote{\label{f:carboni}We adopt \citeauthor{carboni:bicategories-spans-relations}'s definition of regular category (See \cref{s:epi-regular}). In this definition, the only limits that are assumed to exist are pullbacks.} and \(\TabAll\) of \textit{unitary tabular} allegories~\cite[2.154]{freyd:allegories} (see also \cite[Thm.~A3.2.10]{johnstone:elephant}). There are similar adjoint equivalences for bicategories of relations~\cite[Cor.~3.6]{carboni:cartesian-bicategories-I} and double categories of relations~\cite[Thm.~10.2]{lambert:relations}.

There are, however, several categories that look like categories of relations (or corelations), but do not fit any of the axiomatic frameworks listed above.
\begin{itemize}
    \item Surjective multivalued functions (see \cref{x:bitotal:dagger}) can be represented as jointly monic spans of surjective functions. Their generalisation to regular categories (see \cref{s:bitotal-relations}) are the relations that are represented by spans of regular epimorphisms.
   
    \item Injective partial functions (see \cref{x:partial-injections:dagger}) can be represented as jointly epic cospans of injective functions.
    More generally, in rm-adhesive (or quasiadhesive) categories~\cite{garner:adhesive,lack:adhesive} (see~\cref{s:adhesive}), relations that are represented by spans of regular monomorphisms can also be represented as jointly epic cospans of regular monomorphisms. Even more generally, partial isomorphisms in well-behaved restriction categories~\cite{cockett2002restriction} (see \cref{s:partial-isomorphisms}) can be represented as jointly epic cospans of restriction monomorphisms.
    
    \item Stochastic maps between finite probability spaces (see \cref{x:finprob:dagger}) can be represented by jointly monic spans of surjective functions. More generally, measure-preserving Markov kernels between standard Borel probability spaces (see \cref{s:standard-borel}) can be represented by jointly monic spans of measure-preserving measurable functions, i.e., \textit{couplings} of random variables. Even more generally, in a Markov category with conditionals~\cite{fritz:synthetic} (see also~\cite{stein:RV}), measure-preserving morphisms of probability spaces (see \cref{s:general-probability-spaces}) can be represented by jointly monic spans of deterministic morphisms~\cite[Sec.~13]{fritz:synthetic}.

    \item Contractive real matrices (see \cref{x:mat:dagger}) can be represented as jointly epic cospans of matrices with orthonormal columns. More generally, contractions between Hilbert spaces (see \cref{s:contractions}) can be represented as jointly epic cospans of isometries~\cite[Prop.~4.24]{hadzihasanovic}. This is related to the theory of \emph{minimal unitary dilations} of contractions~\cite[Sec.~I.4]{nagyfoias}.
\end{itemize}
In all four examples, there is no obvious way to realise the legs of the representing spans (or cospans) as maps (or comaps) in any of the listed axiomatic frameworks. Our goal is to capture this different kind of category of relations axiomatically.

Each of these categories has a canonical \textit{dagger} (see~\cite{heunenvicary}), that is, a functorial and involutive choice of morphism \(r^\dagger \colon B \to A\) for each morphism \(r \colon A \to B\). Also the legs of the representing spans (or cospans) of their morphisms are \emph{coisometries} (or \textit{isometries}) with respect to these daggers. For each morphism in these dagger categories, not only is there a representing span of coisometries, but there is one that is terminal among all such representing spans. Such representing spans are called \textit{dilations}, and such terminal ones are called \textit{dilators}~\cite[Def.~7.12]{dimeglio:rcategories}. Dagger categories in which every morphism has a dilator will be called \textit{dilatory dagger categories}; they are the topic of \cref{s:dilatorydaggercats}.

The universal property of dilators is very similar to the universal property of tabulators in allegories~\cite[2.143]{freyd:allegories} (see also \cite[Lem.~A3.2.4]{johnstone:elephant}) and in bicategories of relations~\cite[Cor.~3.4]{carboni:cartesian-bicategories-I}. It is natural, then, to expect that a dilatory dagger category \(\D\) is canonically isomorphic to the category \(\Rel(\Coisom(\D))\) of relations in its wide subcategory \(\Coisom(\D)\) of coisometries. The complication is that \(\Coisom(\D)\) is \textit{not} a regular category—it does not have all pullbacks. This raises the question: what structure and properties of \(\Coisom(\D)\) ensure that composition of relations in \(\Coisom(\D)\) is defined and associative? The answer is that \(\Coisom(\D)\) is an \textit{epi-regular independence category}. Let us first discuss independence categories.

Although a cospan $(u,v)$ in \(\Coisom(\D)\) need not have a pullback, the dilator in \(\D\) of $v^\dag u$ looks like a pullback. While it is not terminal among \textit{all} spans that form commutative squares with the cospan, it \textit{is} terminal among certain distinguished such spans. For example, products of probability spaces are not categorical products, but they are terminal among the spans whose legs are conditionally independent. An \textit{independence category} is a category equipped with a well-behaved distinguished class of commutative squares called \textit{independent squares}. An \textit{independent pullback} of a cospan in an independence category is a span that is terminal among all spans that form an independent square with the cospan. These concepts, which are introduced in \cref{s:independence}, are an evolution of Simpson's \textit{systems of independent pullbacks}~\cite{simpson:equivalence-conditional-independence,simpson:category-theoretic-structure-independence} (see also \cite{bohm:crossedmodules,kenneypronk:span}).

Each span $(f,g)$ in $\Coisom(\D)$ factors through the dilator of $g f^\dag$ in \(\D\). Dilators are jointly monic (see~\cref{p:jointly-monic-is-dilator}). Also, all morphisms in $\Coisom(\D)$ are strong epimorphisms (see~\cref{p:coisometry-strong-epic}).
Hence $\Coisom(\D)$ satisfies the definition\cref{f:carboni} of regular category, but with pullbacks replaced by independent pullbacks. An independence category will be called \emph{epi-regular} if it satisfies these properties (see \cref{s:epi-regular}). The \textit{regular} part of the name \textit{epi-regular} alludes to the similarity with regular categories, while the \textit{epi} part of the name \textit{epi-regular} refers to the assumption that \textit{all} morphisms are strong epic, which is not necessarily the case in regular categories. Composition of relations in epi-regular independence categories is defined similarly to composition of relations in regular categories. In fact, if $\C$ is an epi-regular independence category, then $\Rel(\C)$ is a dilatory dagger category (see~\Cref{p:rel-dilatory}).

Our main result is that \(\Rel\) and \(\Coisom\) extend to a strict 2-adjoint 2-equivalence 
\begin{equation}\label{e:equivalence-new}
    \begin{tikzcd}[column sep=large]
    \RegInd
        \ar[r, "\Rel", shift left=2]
        \ar[r, left adjoint]
        \&
    \DilDag
        \ar[l, "\Coisom", shift left=2]
    \end{tikzcd}
\end{equation}
between the 2-category \(\RegInd\) of epi-regular independence categories and the 2\nobreakdash-category \(\DilDag\) of dilatory dagger categories. This mirrors the equivalence~\cref{e:equivalence-old} for allegories, and the analogous ones for bicategories and double categories of relations. We establish the adjoint equivalence of the underlying (1-)categories first, in \cref{s:one-equivalence}, deferring 2-categorical details to \cref{s:two-equivalence}.

The theory of dilatory dagger categories and epi-regular independence categories parallels the theory of allegories and regular categories so well that it is natural to wonder whether there is a more formal connection. 

\subsection*{Acknowledgements}
We thank Richard Garner for discussions about examples of dilators and the connection with tabulators. We also thank Marino Gran, Steve Lack, Tom Leinster, and Tim Van der Linden for their helpful feedback.

\section{Dilatory dagger categories}
\label{s:dilatorydaggercats}

This section recalls the definitions of dagger category, coisometry, and dilator. It also introduces all of the running examples, and uses them to illustrate the definitions.

\subsection{Dagger categories}
\label{s:dag-cats}

A \textit{dagger category} is a category equipped with a chosen morphism \(r^\dagger \colon B \to A\) for each morphism \(r \colon A \to B\) such that
\[1^\dagger = 1, \qquad (sr)^\dagger = r^\dagger s^\dagger, \qquad\text{and}\qquad r^{\dagger\dagger} = r\]
whenever these equations make sense. The operation $r \mapsto r^\dag$ is called the \emph{dagger}.

\begin{example}\label{x:bitotal:dagger}
    A \textit{multivalued function} \(r \colon A \to B\) between sets \(A\) and \(B\) is a function \(r\) from \(A\) to the power set of \(B\) such that, for each \(a \in A\), the set \(ra\) has at least one element. A multivalued function \(r \colon A \to B\) is \textit{surjective} if, for all \(b \in B\), there exists \(a \in A\) such that \(b \in ra\). Sets and surjective multivalued functions form a dagger category \(\MSurj\). The composite of morphisms \(r \colon A \to B\) and \(s \colon B \to C\) is their composite as relations:
    \[
        (sr)a = \bigcup_{b \in ra} sb.
    \]
    The dagger of a morphism \(r \colon A \to B\) is its converse as a relation:
    \[
        r^\dagger b = \setb[\big]{a \in A}{b \in ra}.
    \]
\end{example}

\begin{example}\label{x:partial-injections:dagger}
    A \textit{partial function} \(r \colon A \to B\) between sets \(A\) and \(B\) is a function \(r\) to~\(B\) from a subset \(\supp r\) of \(A\), called the \textit{support} of \(r\). We say that \(r\) is \textit{defined} at an element \(a\) of \(A\) if \(a \in \supp r\), and that \(r\) is \textit{undefined} at \(a\) otherwise. Sets and partial functions form a category \(\Par\). The composite of partial functions \(r \colon A \to B\) and \(s \colon B \to C\) is the composite of the function \(s \colon \supp s \to C\) with the restriction of the function \(r \colon \supp r \to B\) to the subset \(r^{-1} (\supp s)\) of its domain. In other words, \((sr) a = s(ra)\) whenever \(r\) is defined at \(a\) and \(s\) is defined at \(ra\).

    A partial function \(r \colon A \to B\) is \textit{injective} if \(ra = ra'\) implies \(a = a'\). Injective partial functions form a wide subcategory \(\PInj\) of \(\Par\). Each injective partial function \(r \colon A \to B\) corestricts to a bijection between its support and its range
    \[
        \Ran r = \,\bigcup_{\mathclap{a \in \supp r}}\, \set{ra}.
    \]
    Let \(\supp r^\dagger = \Ran r\), and let \(r^\dagger\) be the composite of the inverse of the bijection with the subset inclusion \(\supp r \subseteq A\). In other words, for all \(a \in A\) and \(b \in B\), the partial function \(r\) is defined at \(a\) and \(b = ra\) if and only if the partial function \(r^\dagger\) is defined at \(b\) and \(a = r^\dagger b\). This definition makes \(\PInj\) into a dagger category. Let \(\FinPInj\) denote the full (dagger) subcategory of \(\PInj\) spanned by the sets \([n]= \set{1, 2, \dots, n}\) for each natural number \(n\), where \([0] = \emptyset\).
\end{example}

\begin{example}\label{x:finprob:dagger}
    A (\textit{fully supported}) \textit{probability measure} on a finite set~\(A\) is 
    a function \(\Pr_A \colon A \to (0, 1]\) such that
	\[\sum_{a \in A} \Pr_A(a) \,=\, 1.\]
    For each \(U \subseteq A\), let
    \[\Pr_A(U) = \sum_{a \in U} \Pr_A(a).\]
    A (\textit{fully supported}) \textit{finite probability space} is a pairing \((A, \Pr_A)\) of a finite set~\(A\) with a probability measure \(\Pr_A\) on \(A\). A (\textit{measure-preserving}) \textit{stochastic matrix} \(r \colon (A, \Pr_A) \to (B, \Pr_B)\) is a function \(r(-|-) \colon B \times A \to [0, 1]\) such that
	\[\sum_{b \in B} r(b|a) \,=\, 1\]
	for all \(a \in A\), and
	\[\sum_{a \in A} r(b|a)\,\Pr_A(a) \,=\, \Pr_B(b)\]
	for all \(b \in B\).
	Probability spaces and stochastic matrices from a category \(\FinProb\). The composite of morphisms \(r \colon (A, \Pr_A) \to (B, \Pr_B)\) and \(s \colon (B, \Pr_B) \to (C, \Pr_C)\) is defined, for all \(a \in A\) and \(c \in C\), by
	\[(sr)(c|a) \,=\, \sum_{b \in B} s(c|b)\,r(b|a).\]
	The \textit{Bayesian inverse} of \(r\) is the stochastic matrix \(r^\dagger \colon (B, \Pr_B) \to (A, \Pr_A)\) defined by
	\[r^\dagger(a|b) \,=\, \frac{r(b|a)\,\Pr_A(a)}{\Pr_{B}(b)}.\]
	Bayesian inverses make \(\FinProb\) into a dagger category.
\end{example}

\begin{example}\label{x:mat:dagger}
	Natural numbers and real matrices form a dagger category \(\Mat\). The morphisms \(M \colon m \to n\) are the real matrices \(M\) with \(m\) columns and \(n\) rows. We write \(M_{ji}\) for the entry of \(M\) in column \(i \in [m]\) and row \(j \in [n]\). Composition is matrix multiplication, and the dagger is matrix transposition. A matrix \(M \colon m \to n\) is \textit{contractive} if $\norm{Mx} \leq \norm{x}$ for all $x \in \Reals^m$, where \(\norm{x}\) denotes the Euclidean norm of~\(x\). Equivalently, \(M\) is contractive if its operator norm \(\norm{M}_{\op}\) is at most \(1\). Contractive matrices form a wide subcategory $\Mat_{\leq 1}$ of $\Mat$. To prove that it is actually a \textit{dagger} subcategory, recall~\cite[eq.~(1.18)]{bhatia:positivedefinitematrices} that \(\norm{M^\dagger}_{\op} = \norm{M}_{\op}\).
\end{example}

A \emph{dagger functor} is a functor between dagger categories that preserves daggers. The next example is a finite analogue of the dagger functor \(\ell^2\)~\cite{heunen:ltwo}.

\begin{example}
\label{x:l2-functor}
The dagger functor \(\ell^2 \colon \FinPInj \to \Mat_{\leq 1}\) is defined as follows:
\begin{itemize}
    \item for each natural number \(n\),
    \[\ell^2 [n] = n;\]
    \item for each injective partial function \(r \colon [m] \to [n]\), each \(i \in [m]\), and each \(j \in [n]\),
    \[
        (\ell^2 r)_{ji} =
        \begin{cases}
            1 &\text{if \(i \in \supp r\) and \(j = ri\), and}\\
            0 &\text{otherwise.}
        \end{cases}
    \]
\end{itemize}
\end{example}

\noindent Dagger categories and dagger functors form a category \(\Dag\).

\subsection{Coisometries}
\label{s:coisom}

A morphism \(f\) in a dagger category is \textit{isometric} if \(f^\dagger f=1\), \textit{coisometric} if \(ff^\dagger = 1\), and \textit{unitary} if it is isometric and coisometric (in which case \(f^{-1} = f^\dagger\)). Isometric morphisms are also called \textit{isometries}, coisometric morphisms are also called \textit{coisometries}, and unitary morphisms are isometric isomorphisms. The isometries and coisometries in a dagger category~\(\D\) form wide subcategories \(\Isom(\D)\) and \(\Coisom(\D)\), respectively. Every dagger functor $G \colon \D \to \D'$ restricts to functors $\Isom(G) \colon \Isom(\D) \to \Isom(\D')$ and $\Coisom(G) \colon\Coisom(\D) \to \Coisom(\D')$. 

\begin{example}\label{x:bitotal:coisom}
	Continuing \cref{x:bitotal:dagger}, a morphism $f \colon A \to B$ in $\MSurj$ is coisometric exactly when it is \textit{single valued}, that is, when, for each \(a \in A\), the set \(fa\) has at most one element (see~\cite[Ex.~3.2.2]{heunen:thesis}). Hence the category \(\Coisom(\MSurj)\) is canonically isomorphic to the category \(\Surj\) of sets and surjective functions.
\end{example}

\begin{example}\label{x:partial-injections:isom}
	Continuing \cref{x:partial-injections:dagger}, a morphism \(f \colon A \to B\) in \(\PInj\) is isometric exactly when \(\supp f = A\), that is, when \(f\) is \textit{total}. Hence \(\Isom(\PInj)\) is isomorphic to the category \(\Inj\) of sets and injective functions.
\end{example}

\begin{example}\label{x:finprob:coisom}
	Continuing \cref{x:finprob:dagger}, a morphism \(f \colon (\Omega, \Pr) \to (A, \Pr_A)\) in \(\FinProb\) is coisometric if and only if it is \textit{deterministic}, that is, if and only if \(f(a|\omega) \in \set{0, 1}\) for all \(\omega \in \Omega\) and \(a \in A\). Let \(\FinProbDet = \Coisom(\FinProb)\).

    A \textit{random variable} on a finite probability space \((\Omega, \Pr)\) valued in a finite set \(A\) is (without loss of generality) a surjective function \(X \colon \Omega \to A\). The \textit{pushforward} of \(\Pr\) along \(X\) is the probability measure \(X_*\Pr\) on \(A\) defined by
    \[(X_*\Pr)(a) = \Pr(X^{-1}\set{a}).\]
    The \textit{delta matrix} of \(X\) is the morphism \(\delta_X \colon (\Omega, \Pr) \to (A, X_*\Pr)\) in \(\FinProbDet\) defined, for all \(\omega \in \Omega\) and \(a \in A\), by
    \[
        \delta_X(a|\omega) = \begin{cases}
            1 &\text{if \(a = X(\omega)\), and}\\
            0 &\text{otherwise.}
        \end{cases}
    \]
    For each morphism \(f \colon (\Omega, \Pr) \to (A, \Pr_A)\) in \(\FinProbDet\), there is a random variable \(X\) on \((\Omega, \Pr)\) valued in \(A\) such that \(\Pr_A = X_*\Pr\) and \(f = \delta_X\).

    Given random variables \(X\) and \(Y\) on \((\Omega, \Pr)\) valued, respectively, in finite sets \(A\) and \(B\), the stochastic matrix \(\delta_Y {\delta_X}^\dagger\) encodes the conditional distribution of \(Y\) given~\(X\). Indeed, for all \(a \in A\) and \(b \in B\),
    \begin{multline*}
        (\delta_Y {\delta_X}^\dagger)(b|a)
        \,=\, \sum_{\omega \in \Omega} \delta_Y(b|\omega) \, {\delta_{X}}^\dagger(\omega|a)
        \,=\, \sum_{\omega \in \Omega} \delta_Y(b|\omega) \, \frac{\delta_X(a|\omega)\, \Pr(\omega)}{(X_*\Pr)(a)}
        \\=\, \frac{\Pr(X^{-1}\set{a} \mathrel{\cap} Y^{-1} \set{b})}{\Pr(X^{-1}\set{a})}
        \,=\, \CPr[Y \mathbin{=} b|X \mathbin{=} a].
    \end{multline*}
\end{example}

\begin{example}\label{x:mat:isom}
	Continuing \cref{x:mat:dagger}, a matrix $A \colon m\to n$ is isometric exactly when one of the following equivalent conditions holds:
    \begin{itemize}
        \item \(\norm{Ax} = \norm{x}\) for all \(x \in \Reals^m\), where \(\norm{x}\) is the Euclidean norm of \(x\);
        \item \(Ax \cdot Ay = x \cdot y\) for all \(x, y \in \Reals^m\), where \(x \cdot y\) is the dot product of \(x\) and \(y\); and,
        \item the columns of \(A\) are \textit{orthonormal}, that is, they are pairwise orthogonal and have Euclidean norm \(1\).
    \end{itemize}
    The matrices
	\[
	\begin{bmatrix}
		1 & 0 \\
		0 & 1 \\
		0 & 0
	\end{bmatrix},
	\qquad
	\begin{bmatrix}
		1/\sqrt{2} & -1/\sqrt{2} \\
		1/\sqrt{2} & \phantom{-}1/\sqrt{2}
	\end{bmatrix},
	\qquad\text{and}\qquad
	\begin{bmatrix}
		1/2 & -1/2 & -1/2\\
		1/2 & -1/2 & \phantom{-}1/2 \\
		1/2 & \phantom{-}1/2 & -1/2 \\
		1/2 & \phantom{-}1/2 & \phantom{-}1/2
	\end{bmatrix}
	\]
    are isometric. The number of columns of an isometric matrix is no more than the number of rows of the matrix. As coisometric matrices are transposes of isometric matrices, they are the matrices whose row vectors are orthonormal. Geometrically, they correspond to the surjective linear maps that restrict to surjective maps between the corresponding Euclidean unit balls. Every isometric matrix is contractive, and the same is true of every coisometric matrix. Let \(\Mat_1 = \Isom(\Mat) = \Isom(\Mat_{\leq 1})\).
\end{example}

\subsection{Dilators}
\label{s:dilators}

Before recalling the definition of dilators~\cite[Def.~7.12]{dimeglio:rcategories}, let us establish notation for spans and cospans. A \textit{span} \((X, f, g)\) in a category from an object \(A\) to an object \(B\) consists of an object \(X\) and morphisms \(f \colon X \to A\) and \(g \colon X \to B\). Dually a \textit{cospan} \((X, f, g)\) from \(A\) to \(B\) consists of an object \(X\) and morphisms \(f \colon A \to X\) and \(g \colon B \to X\). When the object \(X\) is obvious from context, or is irrelevant to our argument, we will omit it and merely write \((f, g)\). Similarly, we will often omit objects from displayed diagrams.

\begin{definition}\label{d:dilation}
    Let \(r \colon A \to B\) be a morphism in a dagger category. A \textit{dilation} of \(r\) is a span \((f, g)\) from \(A\) to \(B\) such that \(f\) and \(g\) are coisometries and \(g {f}^\dagger = r\). A \textit{dilator} of \(r\) is a terminal dilation of \(r\), that is, it is a dilation \((R, r_1, r_2)\) of \(r\), such that for all dilations \((X, f, g)\) of \(r\), there is a unique coisometry \(e\) that makes the diagram
	\[
	    \begin{tikzcd}
	        \&
	    X
	        \arrow[dl, "f" swap]
	        \arrow[dr, "g"]
	        \arrow[d, "e"]
	        \&
	    \\
	    A
	        \&
	    R
	        \arrow[l, "r_1"]
	        \arrow[r, "r_2" swap]
	        \&
	    B
	    \end{tikzcd}
	\]
    commutative. Dually, a \textit{codilation} of \(r\) is a cospan \((f, g)\) from \(A\) to \(B\) such that \(f\) and \(g\) are isometries and \({g}^\dagger f = r\), and a \textit{codilator} of \(r\) is an initial codilation of \(r\). Dilators and codilators are unique up to isometric isomorphism.

    A span \((r_1, r_2)\) is a dilator of \(r\) if and only if the cospan \(({r_1}^\dagger, {r_2}^\dagger)\) is a codilator of \(r\). Write \((A \boxplus_r B, p_1, p_2)\) for a chosen dilator of \(r\), and let \(i_1 = {p_1}^\dagger\) and \(i_2 = {p_2}^\dagger\), so that \((A \boxplus_r B, i_1, i_2)\) is a chosen codilator of \(r\).

    A dagger category is \textit{dilatory} if every morphism has a dilator, or, equivalently, if every morphism has a codilator. A dagger functor between dilatory dagger categories is \emph{dilatory} is if it preserves dilators, or, equivalently, if it preserves codilators. Write $\DilDag$ for the category of dilatory dagger categories and dilatory dagger functors.
\end{definition}

\begin{example}\label{x:bitotal:dilator}
	Continuing \cref{x:bitotal:coisom}, the dagger category $\MSurj$ is dilatory. The canonical dilator of a morphism \(r \colon A \to B\) is its \textit{graph}, which is defined by
    \begin{gather*}
        A \boxplus_r B = \setb[\big]{(a, b) \in A \times B}{b \in ra}, \\
        p_1(a, b) = a, \qquad\text{and}\qquad
        p_2(a, b) = b.
    \end{gather*}
\end{example}

\begin{example}\label{x:partial-injections:dilator}
	Continuing \cref{x:partial-injections:isom}, the dagger category $\PInj$ is dilatory. The canonical codilator of a morphism \(r \colon A \to B\) is defined by the equations
    \begin{gather*}
        A \boxplus_r B = (A \backslash \supp r) \sqcup B,\\
        i_1 a = \begin{cases}
            ra  &\text{if \(a \in \supp r\),}\\
            a   &\text{otherwise,}
        \end{cases}
        \qquad\text{and}\qquad
        i_2 b = b.
    \end{gather*}
    This cospan is the pushout in \(\Set\) of the function \(r \colon \supp r \to B\) along the inclusion \(\supp r \subseteq A\). It is also the pushout in \(\Par\) of the partial function \(r \colon A \to B\) along the partial identity function \(A \to A\) that is only defined on \(\supp r\).
\end{example}

\begin{example}\label{x:finprob:dilator}
	Continuing \cref{x:finprob:coisom}, the dagger category \(\FinProb\) is dilatory.
    
    Consider a morphism \(r \colon (A, \Pr_A) \to (B, \Pr_B)\). Let
    \[
        \supp r = \setb[\big]{(a, b) \in A \times B}{r(b|a) \neq 0},
    \]
    and define \(p_1 \colon \supp r \to A\) and \(p_2 \colon \supp r \to B\) by
    \[
        p_1(a, b) = a \qquad\text{and}\qquad p_2(a, b) = b
    \]
    for all \((a, b) \in \supp r\). Let \(\Pr_A \otimes r\) be the probability measure on \(\supp r\) defined by
    \[
        (\Pr_A \otimes r)(a,b) \,=\, r(b|a)\,\Pr_A(a)
    \]
    for all \((a, b) \in \supp r\). Then
    \[
        (p_1)_*(\Pr_A \otimes r) = \Pr_A
        \qquad\text{and}\qquad
        (p_2)_*(\Pr_A \otimes r) = \Pr_B.
    \]
    
    The span
    \[
        \begin{tikzcd}[cramped]
            (A, \Pr_A)
                \&
            (\supp r, \Pr_A \otimes r)
                \arrow[l, "\delta_{p_1}" swap]
                \arrow[r, "\delta_{p_2}"]
                \&
            (B, \Pr_B)
        \end{tikzcd}
    \]
    is a dilator of \(r\). Indeed, it is a dilation of \(r\) because
    \begin{align*}
        \delta_{p_2} {\delta_{p_1}}^\dagger (b|a)
        \,&=\, \,\sum_{(a',b')\in \supp r} \,\delta_{p_2}(b|a',b') \, {\delta_{p_1}}^\dagger(a',b'|a)
        \\&=\,\, \sum_{(a',b')\in \supp r}\, \delta_{p_2}(b|a',b')\,\frac{{\delta_{p_1}}(a|a',b') \, (\Pr_A \otimes r)(a', b')}{\Pr_A(a)}
        \\&=\, \frac{(\Pr_A \otimes r)(a, b)}{\Pr_A(a)}
        \\&=\, r(b|a)
    \end{align*}
    for all \(a \in A\) and \(b \in B\).

    For universality, consider another dilation of \(r\). It is necessarily of the form
    \[
        \begin{tikzcd}[cramped]
            (A, \Pr_A)
                \&
            (\Omega, \Pr)
                \arrow[l, "\delta_{X}" swap]
                \arrow[r, "\delta_{Y}"]
                \&
            (B, \Pr_B)
        \end{tikzcd}
    \]
    for some random variables \(X\) and \(Y\) on \((\Omega, \Pr)\) valued, respectively, in \(A\) and \(B\), such that \(X_*\Pr = \Pr_A\) and \(Y_*\Pr = \Pr_B\). Define \(Z \colon \Omega \to \supp r\) by
    \[
        Z(\omega) = \paren[\big]{X(\omega), Y(\omega)}.
    \]
    Then \(\delta_Z\) is the unique coisometry \(h \colon (\Omega, \Pr) \to (\supp r, \Pr_A \otimes r)\) such that
    \[
        \begin{tikzcd}[row sep=large]
            \&
            (\Omega, \Pr)
                \arrow[dl, "\delta_{X}" swap]
                \arrow[dr, "\delta_{Y}"]
                \arrow[d, "h"]
            \&
        \\
            (A, \Pr_A) \&
            (\supp r, \Pr_A \otimes r)
                \arrow[l, "\delta_{p_1}"]
                \arrow[r, "\delta_{p_2}" swap]
                \&
            (B, \Pr_B)
        \end{tikzcd}
    \]
    is commutative.
\end{example}

\begin{example}\label{x:mat:dilator}
    Continuing \cref{x:mat:isom}, the category $\Mat_{\leq 1}$ is dilatory. Let us construct the codilator of a morphism \(R \colon m \to n\). The matrix \(1 - R^\dagger R\) is positive semidefinite~\cite[Prop.~1.3.2 and Thm.~1.3.3]{bhatia:positivedefinitematrices}, so it has a full-rank Cholesky factorisation~\cite{canto:cholesky-factorization}. Letting \(d\) be the rank of \(1 - R^\dagger R\), this means that there is an upper row-echelon matrix \(E \colon m \to d\) whose rows all have strictly positive leading entries such that \(E^\dagger E = 1 - R^\dagger R\). By back substitution, there is a matrix $M \colon d \to m$ such that $EM=1$.
    The cospan
	\[
	\begin{tikzcd}[cramped, sep=large]
		m
		\&
		d + n
		\arrow[from=l, "\begin{bsmallmatrix}
            E \\
            R
        \end{bsmallmatrix}"]
		\arrow[from=r, "\begin{bsmallmatrix}
            0\vphantom{E}\\
            1\vphantom{R}
        \end{bsmallmatrix}" swap]
		\&
		n
	\end{tikzcd}
	\]
    is a codilator of \(R\). Indeed, it is a codilation of \(R\), and, if
    \[
    \begin{tikzcd}[cramped, sep=large]
        m
        \&
        p
        \arrow[from=l, "A"]
        \arrow[from=r, "B" swap]
        \&
        n
    \end{tikzcd}
    \]
    is another codilation of \(R\), then
    \[
        \begin{bmatrix} (A-BR)M & B\end{bmatrix}
    \]
    is the unique matrix \(C \colon d + n \to p\) such that
    \[
        C \begin{bmatrix} E \\ R \end{bmatrix} = A,\qquad
        C \begin{bmatrix} 0 \\ 1 \end{bmatrix} = B,\qquad\text{and}\qquad
        C^\dag C=1.
    \]
    To prove the first equality, observe that
    \[
        \begin{bmatrix} (A-BR)M & B \end{bmatrix}\begin{bmatrix} E \\ R \end{bmatrix} - A = (BR - A)(1 - ME),
    \]
    and show instead (see~\cite[eq.~(1.19)]{bhatia:positivedefinitematrices}) that
    \[
        (1 - ME)^\dagger (BR - A)^\dagger (BR - A)(1 - ME) = (1 - ME)^\dagger E^\dagger E (1 - ME) = 0.
    \]
\end{example}

\begin{example}
    Continuing \Cref{x:l2-functor}, the dagger functor \(\ell^2 \colon \FinPInj \to \Mat_{\leq 1}\) is dilatory. Consider an injective partial function \(r \colon [m] \to [n]\). Let \(k\) be the cardinality of \(\supp r\). By reindexing the domain and codomain of \(r\) if necessary, we may assume, without loss of generality, that \(\supp r = [k]\) for some \(k \leq m\) and \(ri = i\) for all \(i \in [k]\). From here, simply compare the dilator of \(r\) described in \cref{x:partial-injections:dilator} with the dilator of \(R = \ell^2(r)\) described in \cref{x:mat:dilator}. Observe that
    \[1 - R^\dagger R = \begin{bmatrix}0 & 0 \\ 0 & 1\end{bmatrix} \colon k + (m - k) \to k + (m - k),\]
    so its rank \(d\) is \(m - k\), and the matrices
    \[
        E = \begin{bmatrix}0 & 1\end{bmatrix} \colon k + d \to d
    \]
    and \(M = E^\dagger\) satisfy \(E^\dagger E = 1 - R^\dagger R\) and \(EM = 1\).
\end{example}

We will repeatedly use the following characterisation of dilators of coisometries.

\begin{lemma}
\label{p:dilator-coisometry}
For each coisometry \(f\) in a dagger category, if \(f\) has a dilator then the span \((1, f)\) is a dilator of \(f\).
\end{lemma}

\begin{proof}
Let \((f_1, f_2)\) be a dilator of \(f\). As \((1, f)\) is a dilation of \(f\), there is a unique coisometry \(e\) such that the diagram
\[
    \begin{tikzcd}[cramped, sep=large]
            \&
        \pob
            \arrow[dl, "1" swap]
            \arrow[dr, "f"]
            \arrow[d, "e"]
            \&
        \\
        \pob
            \&
        \pob
            \arrow[l, "f_1"]
            \arrow[r, "f_2" swap]
            \&
        \pob
    \end{tikzcd}
\]
is commutative. Then
\(f_1 = f_1 ee^\dagger = e^\dagger\).
As \(e\) is both isometric and coisometric, it is unitary, and so \((1, f)\) is another dilator of \(f\).
\end{proof}

\section{Epi-regular independence categories}
\label{s:epiregularindependencecats}

Independence categories and independent pullbacks are respectively introduced in \cref{s:independence,s:ind-pull}. \cref{s:factorisation} is about factorisations of spans. Epi-regularity is the topic of \cref{s:epi-regular}.

\subsection{Independence categories}
\label{s:independence}

The following definitions are inspired by \citeauthor{simpson:equivalence-conditional-independence}'s \textit{systems of independent pullbacks}~\cite[Def.~6.1]{simpson:equivalence-conditional-independence}.

\begin{definition}\label{d:independence}
An \textit{independence category} is a category equipped with a predicate~\(\perp\), pronounced ``is independent'', on its squares such that
\begin{enumerate}[label={(I\arabic*)}]
    \item \label{a:independence:1}
    \begin{tikzcd}[cramped]
        \pob
            \arrow[r, "g"]
            \arrow[d, "f" swap]
            \arrow[dr, independent]
            \&
        \pob
            \arrow[d, "v"]
        \\
        \pob
            \arrow[r, "u" swap]
            \&
        \pob
    \end{tikzcd} implies \(uf = vg\),

    \item\label{a:independence:2}
    \begin{tikzcd}[cramped]
        \pob
            \arrow[r, "1"]
            \arrow[d, "f" swap]
            \arrow[dr, independent]
            \&
        \pob
            \arrow[d, "f"]
        \\
        \pob
            \arrow[r, "1" swap]
            \&
        \pob
    \end{tikzcd},

    \item\label{a:independence:3}
    \begin{tikzcd}[cramped]
        \pob
            \arrow[r, "a"]
            \arrow[d, "f" swap]
            \arrow[dr, independent]
            \&
        \pob
            \arrow[d, "g"]
        \\
        \pob
            \arrow[r, "u" swap]
            \&
        \pob
    \end{tikzcd}
    and
    \begin{tikzcd}[cramped]
        \pob
            \arrow[r, "b"]
            \arrow[d, "g" swap]
            \arrow[dr, independent]
            \&
        \pob
            \arrow[d, "h"]
        \\
        \pob
            \arrow[r, "v" swap]
            \&
        \pob
    \end{tikzcd}
    imply
    \begin{tikzcd}[cramped]
        \pob
            \arrow[r, "a"]
            \arrow[d, "f" swap]
            \arrow[drr, independent]
            \&
        \pob
            \arrow[r, "b"]
            \&
        \pob
            \arrow[d, "h"]
        \\
        \pob
            \arrow[r, "u" swap]
            \&
        \pob
            \arrow[r, "v" swap]
            \&
        \pob
    \end{tikzcd},
    
    \item\label{a:independence:4}
    \begin{tikzcd}[cramped]
        \pob
            \arrow[r, "g"]
            \arrow[d, "f" swap]
            \arrow[dr, independent]
            \&
        \pob
            \arrow[d, "v"]
        \\
        \pob
            \arrow[r, "u" swap]
            \&
        \pob
    \end{tikzcd}
    implies
    \begin{tikzcd}[cramped]
        \pob
            \arrow[r, "f"]
            \arrow[d, "g" swap]
            \arrow[dr, independent]
            \&
        \pob
            \arrow[d, "u"]
        \\
        \pob
            \arrow[r, "v" swap]
            \&
        \pob
    \end{tikzcd},

    \item\label{a:independence:5}
    \begin{tikzcd}[cramped]
        \pob
            \arrow[r, "f"]
            \arrow[d, "f" swap]
            \arrow[dr, independent]
            \&
        \pob
            \arrow[d, "1"]
        \\
        \pob
            \arrow[r, "1" swap]
            \&
        \pob
    \end{tikzcd}.
\end{enumerate} 

A \textit{co-independence category}\todo{Maybe we could use the terms orthocategory, orthogonal square and orthopullback instead of co-independence category, etc.? If so, then we should swap the symbols for independence and orthogonality} is the opposite of an independence category, that is, it is a category equipped with a predicate \(\perp\), pronounced ``is co-independent'', on its squares that satisfies the self-dual axioms \cref*{a:independence:1} to \cref*{a:independence:4}, as well as the reflection
\[
    \begin{tikzcd}[cramped, sep=large]
	\pob
	\arrow[r, "1"]
	\arrow[d, "1" swap]
	\arrow[dr, coindependent]
	\&
	\pob
	\arrow[d, "f"]
	\\
	\pob
	\arrow[r, "f" swap]
	\&
	\pob
	\end{tikzcd}
\]
of axiom \cref*{a:independence:5}.

An \emph{independence functor} is a functor between independence categories that preserves independent squares. Write $\Ind$ for the category of independence categories and independence functors.
\end{definition}

Axiom \cref*{a:independence:3} and \cref*{a:independence:4} are the same as Simpson's axioms (IP3) and (IP2) respectively. Axiom \cref*{a:independence:2} and \cref*{a:independence:5} are, given the other axioms, equivalent to Simpson's axiom (IP1). 

For the reader familiar with double categories, the independence category axioms ensure that the independent squares of \(\C\) form a wide double subcategory \(\IndSq(\C)\) of the double category \(\Sq(\C)\) of commutative squares in \(\C\).

\begin{proposition}\label{p:dildag-indep}
	If \(\D\) is a dagger category, then \(\Coisom(\D)\) is an independence category when equipped with the predicate \(\perp\) on its squares defined by
    \begin{equation}\label{e:ind}
        \begin{tikzcd}[cramped, sep=large]
            \pob
            \arrow[r, "g"]
            \arrow[d, "f" swap]
            \arrow[dr, independent]
            \&
            \pob
            \arrow[d, "v"]
            \\
            \pob
            \arrow[r, "u" swap]
            \&
            \pob
        \end{tikzcd}
        \qquad
        \iff
        \qquad
        uf = vg \quad\text{and}\quad
        gf^\dagger = v^\dagger u.
    \end{equation}
	Also, if $G \colon \D \to \D'$ is a dagger functor, then $\Coisom(G) \colon \Coisom(\D) \to \Coisom(\D')$ is an independence functor.
	In fact, $\Coisom$ is a functor $\Dag \to \Ind$.
\end{proposition}


The proof is straightforward; simply unfold the relevant definitions.

For dagger categories \(\D\) that are dilatory (like all of our running examples), the definition of the independent squares in \(\Coisom(\D)\) can be simplified.

\begin{lemma}
\label{p:coisom-ind-commute}
In a dilatory dagger category, if a square
\[
    \begin{tikzcd}[cramped, sep=large]
        \pob
            \arrow[r, "g"]
            \arrow[d, "f" swap]
            \&
        \pob
            \arrow[d, "v"]
        \\
        \pob
            \arrow[r, "u" swap]
            \&
        \pob
    \end{tikzcd}
\]
of coisometries satisfies \(gf^\dagger = v^\dagger u\), then it is commutative (and thus independent).
\end{lemma}

\begin{proof}
The span \((uf, vg)\) is a dilation of \(1\) because
\[vg(uf)^\dagger = vgf^\dagger u^\dagger = vv^\dagger uu^\dagger = 1.\]
By \cref{p:dilator-coisometry}, the span \((1, 1)\) is a dilator of \(1\). Hence there is a unique morphism \(h\) such that \(uf = 1h\) and \(vg = 1h\). This means that \(uf = h = vg\).
\end{proof}

\begin{example}
\label{x:msurj-indep}
Recall from \cref{x:bitotal:coisom} that the coisometries in \(\MSurj\) are the surjective (single-valued) functions. A commutative square
\[
    \begin{tikzcd}
        X
            \arrow[r, "g"]
            \arrow[d, "f" swap]
            \&
        B
            \arrow[d, "v"]
        \\
        A
            \arrow[r, "u" swap]
            \&
        C
    \end{tikzcd}
\]
in the category \(\Surj \cong \Coisom(\MSurj)\) is independent exactly when
\[
    f^{-1}\set{a} \cap g^{-1}\set{b} \neq \emptyset
\]
for all \(a \in A\) and \(b \in B\) such that \(ua = vb\), or, equivalently, when the comparison function from the span \((f, g)\) to the pullback in \(\Set\) of the cospan \((u, v)\) is surjective.
\end{example}

\begin{example}\label{f:partial-injections:ind}
Recall from \cref{x:partial-injections:isom} that the isometries in \(\PInj\) are precisely the injective (total) functions. Consider a commutative diagram
\[
\begin{tikzcd}[sep=large]
	C
	\ar[r, "v"]
	\ar[d, "u" swap]
    \ar[dr, "h"]
	\&
	B
	\ar[d, "g"]
	\\
	A
	\ar[r, "f" swap]
	\&
	X
\end{tikzcd}
\]
in the category \(\Inj \cong \Isom(\PInj)\) of sets and injective functions. The outer square is co-independent exactly when $\Ran f \cap \Ran g = \Ran h$, or, equivalently, when it is a pullback. As \(\Inj\) has pullbacks, every cospan in \(\Inj\) is part of a co-independent square.
\end{example}

\begin{example}\label{x:finprob:ind}
	Recall from \cref{x:finprob:coisom} that the coisometries in \(\FinProb\) are the deterministic maps, and they can be regarded as random variables. Independence of squares in \(\FinProbDet\) is connected to independence of random variables.
    
    Consider two random variables \(X\) and \(Y\) on a finite probability space \((\Omega, \Pr)\) valued, respectively, in finite sets \(A\) and \(B\). The one-point probability space \(1\) is terminal in \(\FinProbDet\), so
    \begin{equation}
        \label{e:finprob:ind-rvs-sq}
        \begin{tikzcd}
            (\Omega, \Pr)
                \ar[r, "\delta_Y"]
                \ar[d, "\delta_X" swap]
                \&
            (B, Y_*\Pr)
                \ar[d, "!"]
            \\
            (A, X_*\Pr)
                \ar[r, "!" swap]
                \&
            1
        \end{tikzcd}
    \end{equation}
    is a commutative square in \(\FinProbDet\). The probability space \(1\) is also a zero object in \(\FinProb\); the zero morphism \((A, X_*\Pr) \to (B, Y_*\Pr)\) is defined by
    \[(a, b) \mapsto (Y_*\Pr)(b)\]
    for all \(a \in A\) and \(b \in B\). Hence the square \cref{e:finprob:ind-rvs-sq} is independent exactly when
    \[
        \delta_Y {\delta_X}^\dagger (b|a) \,=\, (Y_*\Pr)(b)
    \]
    for all \(a \in A\) and all \(b \in B\).
    By the discussion at the end of \cref{x:finprob:coisom}, this equation can be rewritten as
    \[
        \CPr[Y \mathbin{=} b|X \mathbin{=} a]\,=\,\CPr[Y \mathbin{=} b].
    \]
    This is one way to say that the random variables \(X\) and \(Y\) are independent.
            
    More generally, suppose that \(fX = gY\) for some surjective functions \(f \colon A \to C\) and \(g \colon B \to C\) to a finite set \(C\), and let \(Z = fX = gY\). Then the diagram
    \begin{equation}
        \label{x:fin-prob:ind-square}
        \begin{tikzcd}[sep=large]
            (\Omega, \Pr)
                \ar[r, "\delta_Y"]
                \ar[d, "\delta_X" swap]
                \ar[dr, "\delta_Z"]
                \&
            (B, Y_*\Pr)
                \ar[d, "\delta_g"]
            \\
            (A, X_*\Pr)
                \ar[r, "\delta_f" swap]
                \&
            (C, Z_*\Pr)
        \end{tikzcd}
    \end{equation}
    in \(\FinProbDet\) is commutative. Now, the outer square is independent exactly when
    \begin{equation}
        \label{e:fin-prob:ind}
        (\delta_Y {\delta_X}^\dagger)(b|a)
        \,=\, ({\delta_g}^\dagger \delta_f)(b|a)
    \end{equation}
    for all \(a \in A\) and \(b \in B\). However
    \[
        ({\delta_g}^\dagger \delta_f)(b|a)
        \,=\, \sum_{c \in C} {\delta_g}^\dagger(b|c)\,\delta_f(c|a)
        \,=\, {\delta_g}^\dagger(b|fa)
    \]
    and \({\delta_g}^\dagger = \delta_Y {\delta_Z}^\dagger\), so equation \cref{e:fin-prob:ind} can be rewritten as
    \begin{equation}
        \label{e:fin-prob:ind-1}
        \CPr[Y \mathbin{=} b|X \mathbin{=} a] \,=\, \CPr[Y \mathbin{=} b|Z \mathbin{=} fa],
    \end{equation}
    This is one way to say that \(X\) and~\(Y\) are conditionally independent given \(Z\). 
    
    If the outer square in \cref{x:fin-prob:ind-square} is independent, then \(\delta_Z\) is, by \cref{p:ind-pushout}, necessarily a pushout in \(\FinProbDet\) of \(\delta_X\) and \(\delta_Y\). However, every span in \(\FinProbDet\) has a pushout: take the pushout of the underlying random variables in \(\FinSet\) and equip its apex with the pushforward measure (see \cref{x:finprob:coisom}). We propose the name \textit{relatively independent} for two random variables \(X\) and \(Y\) that are conditionally independent given their pushout. 
    
    Note that equation \cref{e:fin-prob:ind-1} can also be simplified to
    \[
        \CPr[X \mathbin{=} a, Y \mathbin{=} b]
        \,=\,
        \begin{cases}
            \dfrac{\CPr[X \mathbin{=} a]\,\CPr[Y \mathbin{=} b]}{\CPr[Z \mathbin{=} fa]} &\text{if \(fa = gb\), and}\\
            0 &\text{otherwise,}
        \end{cases}
    \]
    highlighting the symmetry in \(X\) and \(Y\).
\end{example}

\begin{example}\label{x:mat:ind}
	Recall from \cref{x:mat:isom} that the isometries in $\Mat_{\leq 1}$, which are also the isometries in \(\Mat\), are the matrices whose columns are orthonormal. Co-independence of squares in \(\Isom(\Mat_{\leq 1}) = \Isom(\Mat)\) is related to orthogonality.
		
    First, consider a (necessarily commutative) square
    \[
        \begin{tikzcd}
            0
                \ar[r, "!"]
                \ar[d, "!" swap]
                \&
            n
                \ar[d, "B"]
            \\
            m
                \ar[r, "A" swap]
                \&
            p
        \end{tikzcd}
    \]
    of isometric matrices. Co-independence of the square means that \(B^\dagger A = 0\). This is equivalent to the columns of $A$ being orthogonal to the columns of $B$. It is also equivalent to the \textit{column space}
    \[
        \Col A = \setb[\big]{Au}{u \in \Reals^m}
    \]
    of \(A\) being orthogonal to \(\Col B\).

    More generally, consider a commutative diagram
    \begin{equation}
        \label{e:mat-ind}
        \begin{tikzcd}[sep=large]
            k
                \ar[r, "V"]
                \ar[d, "U" swap]
                \ar[dr, "C"]
                \&
            n
                \ar[d, "B"]
            \\
            m
                \ar[r, "A" swap]
                \&
            p
        \end{tikzcd}
    \end{equation}
    of isometric matrices. Co-independence of the outer square means that \(A^\dagger B = UV^\dagger\).
    
    Observe that \(\Col C\) is a subspace of both \(\Col A\) and \(\Col B\). Let \(P = CC^\dagger\); multiplying by \(P\) orthogonally projects vectors in \(\Reals^p\) onto \(\Col C\). The subspaces \(\Col A\) and \(\Col B\) are said to be \textit{relatively orthogonal}~\cite{jamneshanpan:syndeticity} with respect to the common subspace \(\Col C\) if \(a \cdot b = a \cdot Pb\) for all \(a \in \Col A\) and all \(b \in \Col B\). We will show that the outer square in \cref{e:mat-ind} is co-independent if and only if \(\Col A\) and \(\Col B\) are relatively orthogonal with respect to \(\Col C\). Both directions use the equation
    \[
        UV^\dagger = A^\dagger A UV^\dagger B^\dagger B
        = A^\dagger (AU)(BV)^\dagger B
        = A^\dagger CC^\dagger B
        = A^\dagger PB.
    \]
    For the \textit{only if} direction, if \(a \in \Col A\) and \(b \in \Col B\), then \(a = Ax\) for some \(x \in \Reals^m\) and \(b = By\) for some \(y \in \Reals^n\), so
    \[
        a \cdot b = Ax \cdot By = x \cdot A^\dagger B y = x \cdot UV^\dagger y = x \cdot A^\dagger PBy = Ax \cdot PBy = a \cdot Pb.
    \]
    For the \textit{if} direction, for all \(i \in \set{1, 2, \dots, m}\) and \(j \in \set{1, 2, \dots, n}\),
    \[
        e_i \cdot A^\dagger(1 - P)Be_j = Ae_i \cdot (1 - P)Be_j = A e_i \cdot Be_j - Ae_i \cdot PBe_j = 0,
    \]
    so \(A^\dagger(1 - P)B = 0\), and thus \(A^\dagger B = A^\dagger PB = UV^\dagger\).

    To give better intuition for relative orthogonality, we now prove an alternative characterisation\todo{Cite e.g. \href{https://terrytao.wordpress.com/tag/ergodicity/}{Tao} (better to cite a paper though)} of it in terms of orthogonal complements: the subspaces \(\Col A\) and \(\Col B\) are relatively orthogonal with respect to \(\Col C\) if and only if \(\Col A\) is orthogonal to the \textit{orthogonal complement}
    \[
        \Col B \ominus \Col C = \setb[\big]{b \in \Col B}{b \cdot c = 0 \text{ for all } c \in \Col C}
    \]
    of \(\Col C\) in \(\Col B\). Note that \(b \in \Col B \ominus \Col C\) if and only if \(Pb = 0\). For the \textit{only if} direction, if \(a \in \Col A\) and \(b \in \Col B \ominus \Col C\), then \(Pb = 0\), so
    \[
        a \cdot b = a \cdot Pb = a \cdot 0 = 0.
    \]
    For the \textit{if} direction, if \(a \in \Col A\) and \(b \in \Col B\), then
    \[
        a \cdot b = a \cdot Pb + a \cdot (1 - P)b = a \cdot Pb
    \]
    because \((1 - P)b \in \Col B \ominus \Col C\).

    By the dual of \cref{p:ind-pushout}, if the outer square in \cref{e:mat-ind} is co-independent, then the span \((U, V)\) is necessarily the pullback of the cospan \((A, B)\). This corresponds to the fact that if \(\Col A\) and \(\Col B\) are relatively orthogonal with respect to \(\Col C\), then \(\Col C = \Col A \cap \Col B\). Hence we may leave \(\Col C\) implicit and simply talk about relative orthogonality of the subspaces \(\Col A\) and \(\Col B\).
    
    For further intuition behind relative orthogonality, see \cref{f:rel-orth}.
    \begin{figure}
        \centering
        \begin{tikzpicture}[x=2em, y=1.5em]
        
        \filldraw[
            fill=black!25
        ] (0,0) -- (0,-3) -- (2,-1) -- (2,0) -- cycle;

        \begin{scope}[
            fill=black!5
        ]
            \filldraw (3,0) -- (5,2) -- (2,2) -- (0,0) -- cycle;
            \filldraw (0,0) -- (0,2) -- (-1,2) -- (-3,0) -- cycle;
        \end{scope}

        \filldraw[
            fill=black!25
        ] (0,0) -- (2,2) -- (2,5) -- (0,3) -- cycle;

        \draw[line width=1pt] (0,0) -- (2,2);

        \draw[line width=1pt] (1,1) -- (4,1);
        \draw[line width=1pt] (0,1) -- (-2,1);
        \draw[line width=1pt] (1,1) -- (1,4);
        \draw[line width=1pt] (1,0) -- (1,-2);

        \end{tikzpicture}
        \caption{Two planes in \(\Reals^3\) that are relatively orthogonal.}
        \label{f:rel-orth}
    \end{figure}
\end{example}

It is clear from \cref{x:finprob:ind,x:mat:ind} that the co-independent squares in \(\PInj\) and \(\Mat_{\leq 1}\) are always pullback squares. Less obviously, the independent squares in \(\MSurj\) and \(\FinProb\) are always pushout squares. This is no coincidence. 

\begin{proposition}
\label{p:ind-pushout}
If \(\D\) is a dagger category, then every independent square in \(\Coisom(\D)\) is a pushout square in \(\D\) and also in \(\Coisom(\D)\).
\end{proposition}

Of course, dually, every \textit{co-independent} square in \(\Isom(\D)\) is a \textit{pullback} square in both \(\D\) and \(\Isom(\D)\).

\begin{proof}
Consider an independent square
\[
    \begin{tikzcd}[sep=large, cramped]
        \pob
            \arrow[r, "g"]
            \arrow[d, "f" swap]
            \arrow[dr, independent]
            \&
        \pob
            \arrow[d, "v"]
        \\
        \pob
            \arrow[r, "u" swap]
            \&
        \pob
    \end{tikzcd}
\]
in \(\Coisom(\D)\). Let \(a\) and \(b\) be morphisms in \(\D\) making the square
\[
    \begin{tikzcd}[sep=large, cramped]
        \pob
            \arrow[r, "g"]
            \arrow[d, "f" swap]
            \&
        \pob
            \arrow[d, "b"]
        \\
        \pob
            \arrow[r, "a" swap]
            \&
        \pob
    \end{tikzcd}
\]
commutative. We must show that there is a unique morphism \(c\) in \(\D\) such that
\[
    \begin{tikzcd}[sep={0}, cramped]
        \pob
            \arrow[r, "g"]
            \arrow[d, "f" swap]
            \&[4.8em]
        \pob
            \arrow[d, "v" swap, shift right]
            \arrow[ddr, "b", out=-30, in=90, looseness=1, shorten <=-0.25ex, shorten >=0.25ex]
            \&[1.5em]
        \\[3.6em]
        \pob
            \arrow[r, "u", shift left]
            \arrow[rrd, "a" swap, out=-60, in=180, looseness=0.8, shorten <=-1ex, shorten >=0.5ex]
            \&
        \pob
            \arrow[dr, "c"{pos=0.1}, shorten >=0.5ex, shorten <=-1ex]
            \&
        \\[1em]
            \&
            \&
        \pob
    \end{tikzcd}
\]
is commutative. For convenience, let
\[
    h = uf = vg \qquad \text{and} \qquad d = af = bg.
\]
If such a morphism \(c\) exists, then it satisfies the equation
\[
    c = cuu^\dagger = au^\dagger = aff^\dagger u^\dagger = dh^\dagger.
\]
Conversely, if we define \(c\) by the equation \(c = dh^\dagger\),
then
\begin{align*}
    cu
    &= dh^\dagger u
    = aff^\dagger u^\dagger u
    = au^\dagger u
    = au^\dagger vv^\dagger u
    \\&= afg^\dagger g f^\dagger
    = bgg^\dagger g f^\dagger
    = bgf^\dagger
    = aff^\dagger
    = a,
\end{align*}
and similarly \(cv = b\). Hence the independent square is a pushout square in \(D\).

To see that it is also a pushout square in \(\Coisom(\D)\), observe that if \(a\) is coisometric, then \(cc^\dagger = cuu^\dagger c^\dagger = aa^\dagger = 1\), so \(c\) is also coisometric (a similar argument shows that \(b\) is also coisometric).
\end{proof}

In light of \cref{p:ind-pushout}, we may think of independence as really being a property of spans rather than squares. This perspective is particularly suited for independence categories like \(\MSurj\) and \(\FinProb\) in which \textit{every} span has a pushout.

\begin{proposition}\label{p:relative-independence}
In a dagger category, if a diagram
\[
    \begin{tikzcd}[sep=large, cramped]
        \pob
            \arrow[r, "g"]
            \arrow[d, "f" swap]
            \arrow[dr, "h"]
            \&
        \pob
            \arrow[d, "v"]
        \\
        \pob
            \arrow[r, "u" swap]
            \&
        \pob
    \end{tikzcd}
\]
of coisometries is commutative, then the outer square is independent if and only if
\[
    h^\dagger h = f^\dagger f g^\dagger g.
\]
\end{proposition}

\begin{proof}
If the square is independent, then
\[
    h^\dagger h = f^\dagger u^\dagger vg = f^\dagger f g^\dagger g.
\]
Conversely, if \(h^\dagger h = f^\dagger f g^\dagger g\), then
\[
    u^\dagger v = ff^\dagger u^\dagger vgg^\dagger = fh^\dagger h g^\dagger = ff^\dagger f g^\dagger g g^\dagger = fg^\dagger. \qedhere
\]
\end{proof}

\cref{p:ind-pushout,p:relative-independence} are connected to each other. To see this, we must first recall some facts~\cite{selinger:idempotents} about idempotents in dagger categories (which we will not need beyond the present subsection). An endomorphism \(p\) in a dagger category is called \textit{dagger idempotent} if \(pp = p\) and \(p = p^\dagger\), or, equivalently, if \(p^\dagger p = p\). If \(f\) is a coisometry, then \(f^\dagger f\) is dagger idempotent. A dagger idempotent \(p\) is said to \textit{dagger split} if there is a coisometry \(f\) such that \(p = f^\dagger f\). If \(p\) and \(q\) are commuting dagger idempotents on the same object \(A\), then \(pq\) is also dagger idempotent; in fact, it is the meet (greatest lower bound) of \(p\) and \(q\) with respect to the canonical partial ordering of the dagger idempotents on \(A\), which is defined by \(p \leq q\) if and only if \(pq = p\). By \cref{p:relative-independence}, in dagger categories in which every dagger idempotent dagger splits, a span \((f, g)\) of coisometries is part of an independent square
\[
    \begin{tikzcd}[sep=large, cramped]
        \pob
            \arrow[r, "g"]
            \arrow[d, "f" swap]
            \arrow[dr, independent]
            \&
        \pob
            \arrow[d, "v"]
        \\
        \pob
            \arrow[r, "u" swap]
            \&
        \pob
    \end{tikzcd}
\]
of coisometries exactly when the dagger idempotents \(f^\dagger f\) and \(g^\dagger g\) commute; just let \(u = hf^\dagger\) and \(v = hg^\dagger\) where \(h\) is a dagger splitting of the dagger idempotent \(f^\dagger f g^\dagger g\). The fact that \(h^\dagger h\) is the meet of \(f^\dagger f\) and \(g^\dagger g\) with respect to the canonical partial ordering of dagger idempotents corresponds to the fact that \((u, v)\) is the pushout in the category of coisometries of \((f, g)\), which we also know from \cref{p:ind-pushout}.

\cref{p:relative-independence} and the subsequent analysis in the previous paragraph together generalise known facts about the dagger category of relations of a regular category to arbitrary dagger categories (see, e.g., \cite[Thm.~5.2]{carboni:remarks-maltsev-goursat}).\footnote{The following facts will be helpful when translating the cited theorem into the language of the present article. A morphism \(e\) in a regular category \(\C\) is regular epic (called a \textit{surjection} in the cited article) if and only if the relation associated to \(e\) is coisometric in \(\Rel \C\). An equivalence relation in \(\C\) is precisely a dagger idempotent \(R\) in \(\Rel \C\) such that \(1 \leq R\). The embedding of the poset of equivalence relations on an object into the poset of dagger idempotents on the object is order reversing. Indeed, if \(R\) and \(S\) are equivalence relations and \(R \leq S\), then \(RS \leq SS = S\) and \(S = 1S \leq RS\), so \(RS = S\). Conversely, if \(RS = S\), then \(R = R1 \leq RS = S\). The kernel pair of a regular epi \(e\) is the relation \(e^\dagger e\) where \(e\) is regarded as a coisometry in \(\Rel \C\). An equivalence relation \(R\) in \(\C\) is effective (i.e., a kernel pair) precisely when it splits coisometrically, that is, when there is a coisometry \(e\) in \(\Rel \C\) such that \(R = e^\dagger e\).}

\begin{example}
    The dagger idempotents in \(\MSurj\) on a set \(A\) are precisely the equivalence relations on \(A\).
    Each dagger idempotent \({\sim}\) on \(A\) is dagger split, in particular, by the surjective function from \(A\) to the set of equivalence classes of \({\sim}\) that sends each element of \(A\) to its equivalence class. Conversely, for each surjective function \(f \colon A \to B\), the dagger idempotent \(f^\dagger f\) is the equivalence relation on \(A\) whose equivalence classes are the fibres of \(f\); this is the \textit{kernel pair} in \(\Set\) of \(f\). By the discussion before this example, a span \((f, g)\) of surjective functions is part of an independent square in \(\Surj\) exactly when the equivalence relations \(f^\dagger f\) and \(g^\dagger g\) commute. Combining this observation with the characterisation in \cref{x:msurj-indep} of the independent squares in \(\Surj\) yields a known characterisation of the pairs of equivalence relations that commute with each other~\cite{dubreil:relations-dequivilance} (see also \cite[Thm.~2]{britz:operations-family-equivalence}):  they are the pairs of equivalence relations that are \textit{independent} relative to their join.
\end{example}


\begin{lemma}
    \label{p:rev-descent}
    In an epi-regular independence category,
\[
    \begin{tikzcd}[cramped, sep=large]
        \pob
            \arrow[r, "g"]
            \arrow[d, "f" swap]
            \arrow[dr, independent]
            \&
        \pob
            \arrow[d, "v"]
        \\
        \pob
            \arrow[r, "u" swap]
            \&
        \pob
    \end{tikzcd}
    \qquad\text{implies}\qquad
    \begin{tikzcd}[cramped]
        \pob
            \arrow[ddrr, independent]
            \arrow[d, "h" swap]
            \arrow[r, "h"]
            \&
        \pob
            \arrow[r, "g"]
            \&
        \pob
            \arrow[dd, "v"]
        \\
        \pob
            \arrow[d, "f" swap]
            \&
            \&
        \\
        \pob
            \arrow[rr, "u" swap]
            \&
            \&
        \pob
    \end{tikzcd}
    .
\]
\end{lemma}

\begin{proof}
The pasting
\[
    \begin{tikzcd}[cramped, sep=large]
        \pob
            \arrow[r, "h"]
            \arrow[d, "h" swap]
            \arrow[dr, independent]
            \&
        \pob
            \arrow[dr, independent]
            \arrow[d, "1" swap]
            \arrow[r, "g"]
            \&
        \pob
            \arrow[d, "1"]
        \\
        \pob
            \arrow[dr, independent]
            \arrow[r, "1"]
            \arrow[d, "f" swap]
            \&
        \pob
            \arrow[dr, independent]
            \arrow[r, "g"]
            \arrow[d, "f" swap]
            \&
        \pob
            \arrow[d, "v"]
        \\
        \pob
            \arrow[r, "1" swap]
            \&
        \pob
            \arrow[r, "u" swap]
            \&
        \pob
    \end{tikzcd}
\]
is independent by the independence category axioms.
\end{proof}

\subsection{Independent pullbacks}
\label{s:ind-pull}

\begin{definition}
In an independence category, an \textit{independent pullback}
of a cospan \((u, v)\) is a span \((p_1, p_2)\) such that
\[
    \begin{tikzcd}[cramped, sep=large]
        \pob
            \arrow[r, "p_2"]
            \arrow[d, "p_1" swap]
            \arrow[dr, independent]
            \&
        \pob
            \arrow[d, "v"]
        \\
        \pob
            \arrow[r, "u" swap]
            \&
        \pob
    \end{tikzcd},
\]
and, if
\[
    \begin{tikzcd}[cramped, sep=large]
        \pob
            \arrow[r, "g"]
            \arrow[d, "f" swap]
            \arrow[dr, independent]
            \&
        \pob
            \arrow[d, "v"]
        \\
        \pob
            \arrow[r, "u" swap]
            \&
        \pob
    \end{tikzcd},
\]
then there is a unique morphism \(e\) such that the diagram
\[
    \begin{tikzcd}[cramped, sep=large]
            \&
        \pob
            \arrow[dl, "f" swap]
            \arrow[dr, "g"]
            \arrow[d, "e"]
            \&
        \\
        \pob
            \&
        \pob
            \arrow[l, "p_1"]
            \arrow[r, "p_2" swap]
            \&
        \pob
    \end{tikzcd}
\]
is commutative. 

An \textit{independent kernel pair} of a morphism \(u\) is an independent pullback of the cospan \((u, u)\). In an independence category with a terminal object~\(1\), an \textit{independent product} of objects \(X\) and \(Y\) is an independent pullback of the cospan formed by the unique maps \(X \to 1\) and \(Y \to 1\). \textit{Co-independent pushout}, \textit{co-independent cokernel pair} and \textit{co-independent coproduct} for co-independence categories are defined dually.
\end{definition}

Write \((A \boxtimes_C B, p_1, p_2)\) for a chosen independent pullback of a cospan \((C, u, v)\) from \(A\) to \(B\) in an independence category, and write \(A \boxtimes B\) instead of \(A \boxtimes_1 B\) for a chosen independent product of \(A\) and \(B\). Similarly, write \((A \boxplus_C B, i_1, i_2)\) for a chosen co-independent pushout of a span \((C, u, v)\) from \(A\) to \(B\) in a co-independence category, and write \(A \boxplus B\) instead of \(A \boxplus_0 B\) for a chosen co-independent pushout of \(A\) and \(B\).

Recall from \cref{p:dildag-indep} that, for each dagger category \(\D\), the category \(\Coisom(\D)\) is canonically an independence category.

\begin{proposition}\label{p:pullbacks-coisom}
	If $\D$ is a dilatory dagger category, then the independence category $\Coisom(\D)$ has independent pullbacks. In particular, for each cospan  \((u, v)\) in $\Coisom(\D)$, the dilator of \(v^\dagger u\) is an independent pullback of \((u, v)\). Also, if $G$ is a dilatory dagger functor, then $\Coisom(G)$ preserves independent pullbacks.
\end{proposition}

\begin{proof}
	By \cref{p:coisom-ind-commute}, the universal property of an independent pullback of \((u, v)\) is precisely the universal property of a dilator of \(v^\dagger u\).

	Let $G \colon \D \to \D'$ be a dilatory dagger functor. Let
	\[
		\begin{tikzcd}[cramped, sep=large]
			\pob
			\arrow[r, "p_2"]
			\arrow[d, "p_1" swap]
			\arrow[dr, independent]
			\&
			\pob
			\arrow[d, "v"]
			\\
			\pob
			\arrow[r, "u" swap]
			\&
			\pob
		\end{tikzcd}
	\]
	be an independent pullback in $\D$. Then $(p_1,p_2)$ is a dilator of $v^\dag u$. Since $G$ preserves dilators, $(Gp_1, Gp_2)$ is a dilator of $G(v^\dagger u)=(Gv)^\dagger (Gu)$ in $\D'$. Hence $G$ sends the square to an independent pullback square in $\Coisom(\D')$.
\end{proof}

\begin{example}\label{x:finprob:pb}
    Let us construct an independent pullback of a cospan
    \[
        \begin{tikzcd}[cramped]
            (A, \Pr_A)
                \&
            (C, \Pr_C)
                \arrow[from=l, "u"]
                \arrow[from=r, "v" swap]
                \&
            (B, \Pr_B)
        \end{tikzcd}
    \]
    in \(\FinProbDet\). From \cref{x:finprob:coisom}, we know that \(u = \delta_f\) and \(v = \delta_g\) for some surjective functions \(f \colon A \to C\) and \(g \colon B \to C\). Let
    \[
        A \times_C B = \setb[\big]{(a, b) \in A \times B}{fa = gb},
    \]
    and let \(p_1 \colon A \times_C B \to A\) and \(p_2 \colon A \times_C B \to B\) be the restrictions of the coordinate projection functions. The equation
    \[
        (\Pr_A \otimes_C \Pr_B)(a,b) \,=\, \frac{\Pr_A(a)\,\Pr_B(b)}{\Pr_C(c)}
    \]
    for all \((a, b) \in A \times_C B\) defines a probability measure \(\Pr_A \otimes_C \Pr_B\) on \(A \times_C B\) such that
    \[
        (p_1)_* (\Pr_A \otimes_C \Pr_B) = \Pr_A
        \qquad\text{and}\qquad
        (p_2)_* (\Pr_A \otimes_C \Pr_B) = \Pr_B.
    \]
    The span
    \[
        \begin{tikzcd}[cramped]
            (A, \Pr_A)
                \&
            (A \times_C B, \Pr_A \otimes_C \Pr_B)
                \arrow[l, "\delta_{p_1}" swap]
                \arrow[r, "\delta_{p_2}"]
                \&
            (B, \Pr_B)
        \end{tikzcd}
    \]
    in \(\FinProbDet\) is an independent pullback of the original cospan. To prove this, simplify the description in \cref{x:finprob:dilator} of the dilator of \(v^\dagger u\). The probability space \((A \times_C B, \Pr_A \otimes_C \Pr_B)\) is known as~\cite{dawid-studeny} the \textit{conditional product} of \(A\) and \(B\) given \(C\).

    Recall that the one-point probability space \(1\) is terminal in \(\Coisom(\FinProb)\). Letting \((C, \Pr_C) = 1\) in the construction above, we see that an independent product in \(\Coisom(\FinProb)\) of finite probability spaces \((A, \Pr_A)\) and \((B, \Pr_B)\) can be formed by equipping the Cartesian product \(A \times B\) of their underlying sets with the probability measure \(\Pr_A \otimes \Pr_B\) defined, for all \((a, b) \in A \times B\), by
    \[
        (\Pr_A \otimes \Pr_B)(a,b) \,=\, \Pr_A(a)\,\Pr_B(b).
    \]
    The probability space \((A \times B, \Pr_A \otimes \Pr_B)\) is known as the \textit{product} of the probability spaces \(A\) and \(B\), despite not actually being a categorical product in any category (see \cite[Ex.~3.2]{simpson:category-theoretic-structure-independence}).
\end{example}

\begin{example}\label{x:mat:pb}
    The co-independence category \(\Mat_1\) of isometric matrices has a zero object, namely \(0\), and it has dilators (see \cref{x:mat:dilator}), so it has co-independent coproducts. The canonical co-independent coproduct of objects \(m\) and \(n\) is the cospan
    \[
        \begin{tikzcd}[cramped]
            m
                \arrow[r, "\begin{bsmallmatrix}1\\0\end{bsmallmatrix}"]
                \&
            m + n
                \&
            n
                \arrow[l, "\begin{bsmallmatrix}0\\1\end{bsmallmatrix}" swap]
        \end{tikzcd},
    \]
    which is a coproduct of \(m\) and \(n\) in \(\Mat\).

    More generally, the co-independence category \(\Mat_1\) has co-independent pushouts. A co-independent square
	\[
        \begin{tikzcd}
            k
                \ar[r, "V"]
                \ar[d, "U" swap]
                \ar[dr, coindependent]
                \&
            n
                \ar[d, "J_2"]
            \\
            m
                \ar[r, "J_1" swap]
                \&
            p
        \end{tikzcd}
    \]
	in \(\Mat_1\) is a co-independent pushout exactly when the internal sum of \(\Col J_1\) and \(\Col J_2\) is all of $\mathbb{R}^p$, or, equivalently, when the cospan \((J_1, J_2)\) is a pushout in the category \(\Mat\) of the span \((U, V)\).
\end{example}

\begin{lemma}
\label{p:ind-pull-is-joint-monic}
In an independence category, if
\[
    \begin{tikzcd}[cramped, sep=large]
        \pob
            \arrow[r, "m_2"]
            \arrow[d, "m_1" swap]
            \arrow[dr, independent]
            \&
        \pob
            \arrow[d, "v"]
        \\
        \pob
            \arrow[r, "u" swap]
            \&
        \pob
    \end{tikzcd}
\]
is an independent pullback square then the span \((m_1, m_2)\) is jointly monic.
\end{lemma}

\begin{proof}
Let \(f\) and \(g\) be parallel morphisms such that
\[m_1f = m_1g \qquad\text{and}\qquad m_2f = m_2g.\]
Then
\[
    \begin{tikzcd}[cramped]
        \pob
            \arrow[r, "f"]
            \arrow[d, "f" swap]
            \arrow[ddrr, independent]
            \&
        \pob
            \arrow[r, "m_2"]
            \&
        \pob
            \arrow[dd, "v"]
        \\
        \pob
            \arrow[d, "m_1" swap]
            \&
            \&
        \\
        \pob
            \arrow[rr, "u" swap]
            \&
            \&
        \pob
    \end{tikzcd}
\]
by \cref{p:rev-descent}. As \((m_1, m_2)\) is an independent pullback of \((u, v)\), there is a unique morphism \(e\) such that the diagram
\[
    \begin{tikzcd}[cramped]
        \&
        \&
        \pob
            \arrow[dd, "e"]
            \arrow[dl, "f" swap]
            \arrow[dr, "f"]
        \&
        \&
    \\
        \&
        \pob
            \arrow[dl, "m_1" swap]
        \&
        \&
        \pob
            \arrow[dr, "m_2"]
        \&
    \\
        \pob
        \&
        \&
        \pob
            \arrow[ll, "m_1"]
            \arrow[rr, "m_2" swap]
        \&
        \&
        \pob
    \end{tikzcd}
\]
is commutative. It follows that \(f = e = g\).
\end{proof}

The pasting of two independent pullback squares is not, in general, an independent pullback square. This motivates the following definition.\todo{Maybe this belongs in a discussion section at the end once we add it back in?}

\begin{definition}
In an independence category, a \textit{strong independent pullback} of a cospan \((u, v)\) is a span \((p_1, p_2)\) such that
\[
    \begin{tikzcd}[cramped, sep=large]
        \pob
            \arrow[r, "p_2"]
            \arrow[d, "p_1" swap]
            \arrow[dr, independent]
            \&
        \pob
            \arrow[d, "v"]
        \\
        \pob
            \arrow[r, "u" swap]
            \&
        \pob
    \end{tikzcd},
\]
and, if
\[
    \begin{tikzcd}[cramped, sep=large]
        \pob
            \arrow[rr, "g"]
            \arrow[d, "f" swap]
            \arrow[drr, independent]
            \&
            \&
        \pob
            \arrow[d, "v"]
        \\
        \pob
            \arrow[r, "f'" swap]
            \&
        \pob
            \arrow[r, "u" swap]
            \&
        \pob
    \end{tikzcd},
\]
then there is a unique morphism \(e\) such that
\[
    \begin{tikzcd}[cramped, sep=large]
        \pob
            \arrow[from=r, "f" swap]
            \arrow[d, "f'" swap]
            \arrow[dr, independent]
            \&
        \pob
            \arrow[d, "e"]
        \\
        \pob
            \arrow[from=r, "p_1"]
            \&
        \pob
    \end{tikzcd}
    \qquad\text{and}\qquad
    \begin{tikzcd}[cramped, sep=large]
        \pob
            \arrow[dr, "g"]
            \arrow[d, "e" swap]
            \&
        \\
        \pob
            \arrow[r, "p_2" swap]
            \&
        \pob
    \end{tikzcd}
\]
is commutative.
\end{definition}

Clearly all strong independent pullbacks are independent pullbacks. An independence category satisfies the \textit{independent-pullback lemma}---the pasting of two independent pullbacks is again an independent pullback---if and only if every independent pullback is a strong independent pullback. This is a special case of the fact that a functor is a Grothendieck fibration if and only if it has weakly cartesian lifts and the composite of two weakly cartesian lifts is again weakly cartesian~\cite[Prop.~8.1.7]{borceux:1994:handbook-categorical-algebra-2}. The functor of interest is the codomain functor 
\(\IndSq(\C)_1 \to \IndSq(\C)_0\) of the double category \(\IndSq(\C)\). Its weakly cartesian lifts are precisely the independent pullback squares in \(\C\), while its (strongly) cartesian lifts are precisely the strong independent pullback squares in \(\C\).

In \cref{p:pull-pasting}, we will show that the independent-pullback lemma holds in all \textit{epi-regular} independence categories. From the discussion above, it follows that in epi-regular independence categories, all independent pullbacks are, in fact, strong independent pullbacks.

\begin{lemma}
\label{p:ind-pullback-identity}
For each morphism \(u\) in an independence category, the span \((1, u)\) is a strong independent pullback of the cospan \((u, 1)\).
\end{lemma}

\begin{proof}
By axiom \cref{a:independence:2}, it forms an independent square with the cospan. Let
\[
    \begin{tikzcd}[cramped, sep=large]
        \pob
            \arrow[rr, "g"]
            \arrow[d, "f" swap]
            \arrow[drr, independent]
            \&
            \&
        \pob
            \arrow[d, "1"]
        \\
        \pob
            \arrow[r, "f'" swap]
            \&
        \pob
            \arrow[r, "u" swap]
            \&
        \pob
    \end{tikzcd}
\]
be an independent square. Then \(uf'f = g\) by axiom \cref{a:independence:1}, and the pasting
\[
    \begin{tikzcd}[cramped, sep=large]
        \pob
            \arrow[r, "f"]
            \arrow[dr, independent]
            \arrow[d, "f" swap]
            \&
        \pob
            \arrow[r, "f'"]
            \arrow[dr, independent]
            \arrow[d, "1"]
            \&
        \pob
            \arrow[d, "1"]
        \\
        \pob
            \arrow[r, "1" swap]
            \&
        \pob
            \arrow[r, "f'" swap]
            \&
        \pob
    \end{tikzcd}
\]
is independent by axioms \cref{a:independence:2} to \cref{a:independence:5}. Conversely, if the square
\[
    \begin{tikzcd}[cramped, sep=large]
        \pob
            \arrow[r, "e"]
            \arrow[d, "f" swap]
            \arrow[dr, independent]
            \&
        \pob
            \arrow[d, "1"]
        \\
        \pob
            \arrow[r, "f'" swap]
            \&
        \pob
    \end{tikzcd},
\]
is independent, then \(e = f'f\) by axiom \cref{a:independence:1}.
\end{proof}

\subsection{Span factorisations}
\label{s:factorisation}

We use \citeauthor{carboni:bicategories-spans-relations}'s definition of ``strong epic'', see below, which is a slight generalisation of the usual one (and is equivalent to it in the presence of binary products).

\begin{definition}
\label{d:strong-epi}
A morphism \(e\) in a category is \textit{strong epic} if it is left orthogonal to all jointly monic spans, that is, for all commutative diagrams
\[
    \begin{tikzcd}[sep={0}, cramped]
        \pob
            \arrow[r, "h"]
            \arrow[d, "e" swap]
            \&[4.8em]
        \pob
            \arrow[d, "m_1" swap, shift right]
            \arrow[ddr, "m_2"]
            \&[0.3em]
        \\[3.6em]
        \pob
            \arrow[r, "f", shift left]
            \arrow[rrd, "g" swap]
            \arrow[ur, dashed, "d"]
            \&
        \pob
            \&
        \\[0.1em]
            \&
            \&
        \pob
    \end{tikzcd}
\]
where the span \((m_1, m_2)\) is jointly monic, there exists a morphism \(d\) such that \(m_1d = f\) and \(m_2d = g\). A \textit{strong epimorphism} is a morphism that is strong epic.
\end{definition}

As \((m_1, m_2)\) is jointly monic, there is only one such morphism \(d\), and also \(h = de\). The morphism \(d\) is called the \textit{diagonal filler}.

\begin{proposition}\label{p:strong-epic-props}\label{p:reg-is-strong}
	The following facts about strong epimorphisms hold:
	\begin{enumerate}
	    \item every strong epimorphism is epic,
	    \item the composite of two strong epimorphisms is again strong epic,
	    \item a morphism is invertible if and only if it is both monic and strong epic,
	    \item every regular epimorphism is strong epic.
	\end{enumerate}
\end{proposition}
\begin{proof}
For (i), let \(e\) be a strong epimorphism, and suppose that \(fe = ge\). Then the following diagram is commutative.
\[
    \begin{tikzcd}[sep={0}, cramped]
        \pob
            \arrow[r, "fe"]
            \arrow[d, "e" swap]
            \&[4.8em]
        \pob
            \arrow[d, "1" swap, shift right]
            \arrow[ddr, "1"]
            \&[0.3em]
        \\[3.6em]
        \pob
            \arrow[r, "f", shift left]
            \arrow[rrd, "g" swap]
            \&
        \pob
            \&
        \\[0.1em]
            \&
            \&
        \pob
    \end{tikzcd}
\]
Also, the span \((1, 1)\) is jointly monic. Hence there is a unique morphism \(d\) such that \(1d = f\) and \(1d = g\). In particular, \(f = d = g\).

For (ii), (iii), and (iv) adapt the proof of~\cite[Prop.~4.3.6]{borceux:1994:handbook-categorical-algebra-1}.
\end{proof}

\begin{definition}
	A (strong epic, jointly monic) factorisation \((e, m_1, m_2)\) of a span \((f, g)\) consists of a strong epimorphism \(e\) and a jointly monic span \((m_1, m_2)\) such that the following diagram is commutative.
	\[
	    \begin{tikzcd}[cramped, sep=large]
	        \&
	        \pob
	            \arrow[dl, "f" swap]
	            \arrow[dr, "g"]
	            \arrow[d, "e"]
	        \&
	        \\
	        \pob
	        \&
	        \pob
	            \arrow[l, "m_1"]
	            \arrow[r, "m_2" swap]
	        \&
	        \pob
	    \end{tikzcd}
	\]
\end{definition}

It is easy to check that (strong epic, jointly monic) factorisations of spans are unique up to isomorphism.

\begin{lemma}
\label{p:jointly-monic-is-dilator}
Let \(\D\) be a dilatory dagger category. A span \((m_1, m_2)\) of coisometries is jointly monic in \(\Coisom(\D)\) if and only if it is a dilator of \(m_2 {m_1}^\dagger\) in \(\D\).
\end{lemma}

\begin{proof}
For the \textit{if} direction, suppose that \((m_1, m_2)\) is a dilator of \(m_2 {m_1}^\dagger\). Let \(f\) and \(g\) be parallel coisometries such that
\(m_1f = m_1g\) and \(m_2f = m_2g\). Then
\[m_2f(m_1f)^\dagger = m_2 ff^\dagger {m_1}^\dagger = m_2 {m_1}^\dagger,\]
so \((m_1f, m_2f)\) is a dilation of \(m_2 {m_1}^\dagger\). As \((m_1, m_2)\) is a dilator of \(m_2 {m_1}^\dagger\), there is a unique coisometry \(h\) such that \(m_1f = m_1h\) and \(m_2f = m_2h\). Hence \(f = g\).

We now prove the \textit{only if} direction. Suppose that \((m_1, m_2)\) is jointly monic in \(\Coisom(\D)\). Let \((p_1, p_2)\) be a dilator of \(m_2 {m_1}^\dagger\). As \((m_1, m_2)\) is a dilation of \(m_2 {m_1}^\dagger\), there is a unique coisometry \(m\) such that the diagram
\[
    \begin{tikzcd}[cramped, sep=large]
            \&
        \pob
            \arrow[dl, "m_1" swap]
            \arrow[dr, "m_2"]
            \arrow[d, "m"]
            \&
        \\
        \pob
            \&
        \pob
            \arrow[l, "p_1"]
            \arrow[r, "p_2" swap]
            \&
        \pob
    \end{tikzcd}
\]
is commutative. Let \((q_1, q_2)\) be an independent kernel pair of \(m\). Then \(mq_1 = mq_2\) by \cref{p:coisom-ind-commute}. Hence
\[m_1 q_1 = p_1 m q_1 = p_1 mq_2 = m_1 q_2,\]
and similarly \(m_2q_1 = m_2q_2\). However \((m_1, m_2)\) is jointly monic, so \(q_1 = q_2\). Hence
\[m^\dagger m = q_2 {q_1}^\dagger = q_1 {q_1}^\dagger = 1,\]
and so \(m\) is isometric. However \(m\) is also coisometric, so it is actually unitary, and thus \((m_1, m_2)\) is another dilator of \(m_2 {m_1}^\dagger\).
\end{proof}

\begin{lemma}
\label{p:coisometry-strong-epic}
If \(\D\) is a dilatory dagger category, then every morphism in \(\Coisom(\D)\) is strong epic.
\end{lemma}

\begin{proof}
Consider a commutative diagram
\[
    \begin{tikzcd}[sep={0}, cramped]
        \pob
            \arrow[r, "h"]
            \arrow[d, "e" swap]
            \&[4.8em]
        \pob
            \arrow[d, "m_1" swap, shift right]
            \arrow[ddr, "m_2"]
            \&[0.3em]
        \\[3.6em]
        \pob
            \arrow[r, "f", shift left]
            \arrow[rrd, "g" swap]
            \&
        \pob
            \&
        \\[0.1em]
            \&
            \&
        \pob
    \end{tikzcd}
\]
in \(\Coisom(\D)\) where the span \((m_1, m_2)\) is jointly monic. Then
\[g f^\dagger = g ee^\dagger f^\dagger = g e (f e)^\dagger = m_2 h(m_1 h)^\dagger = m_2 hh^\dagger {m_1}^\dagger = m_2 {m_1}^\dagger,\]
so \((f, g)\) is a dilation of \(m_2{m_1}^\dagger\). However, by \cref{p:jointly-monic-is-dilator}, the span \((m_1, m_2)\) is a dilator of \(m_2{m_1}^\dagger\). Hence there is a unique morphism \(d\) such that the diagram
\[
    \begin{tikzcd}[cramped, sep=large]
            \&
        \pob
            \arrow[dl, "f" swap]
            \arrow[dr, "g"]
            \arrow[d, "d"]
            \&
        \\
        \pob
            \&
        \pob
            \arrow[l, "m_1"]
            \arrow[r, "m_2" swap]
            \&
        \pob
    \end{tikzcd}
\]
is commutative. This morphism \(d\) is the diagonal filler.
\end{proof}

\begin{proposition}\label{p:factorisation}
	If $\D$ is a dilatory dagger category, then every span in $\Coisom(\D)$	has a (strong epic, jointly monic) factorisation in $\Coisom(\D)$.
	If $G$ is a dilatory dagger functor, then $\Coisom(G)$ preserves jointly monic spans, and hence (strong epic, jointly monic) factorisations of spans.
\end{proposition}
\begin{proof}
	Let $\D$ be a dilatory dagger category. Let \((f, g)\) be a span in \(\Coisom(\D)\), let \((m_1, m_2)\) be the dilator of \(gf^\dagger\) and let \(e\) be the unique coisometry such that \(m_1e = f\) and \(m_2e = g\). Then \((m_1, m_2)\) is jointly monic by \cref{p:jointly-monic-is-dilator}, and \(e\) is strong epic by \cref{p:coisometry-strong-epic}. 	

	Let $G \colon \D \to \D'$ be a dilatory dagger functor.
	Let $(m_1,m_2)$ be a jointly monic span in $\Coisom(\D)$.
	By \Cref{p:jointly-monic-is-dilator}, it is a dilator of $m_2 {m_1}^\dagger$. 
	Since $G$ preserves dilators, the span $(Gm_1,Gm_2)$ is a dilator of $G(m_2 {m_1}^\dagger)=(Fm_2) {(Fm_1)}^\dagger$.
	Hence, by \Cref{p:jointly-monic-is-dilator}, the span $(Gm_1,Gm_2)$ is jointly monic in $\Coisom(\D')$.
\end{proof}

\subsection{Epi-regularity}
\label{s:epi-regular}

To contextualise the definition of epi-regular independence category, let us first recall the definition of regular category. We adopt the one proposed by \textcite[Sec.~3]{carboni:bicategories-spans-relations}.\footnote{There are in fact several non-equivalent definitions in the literature. The one proposed by \citeauthor{carboni:bicategories-spans-relations} is the most inclusive. The usual one additionally assumes finite completeness.}

\begin{definition}
   \label{d:regular-cat} 
    A category is \textit{regular} if
    \begin{enumerate}[label={(R\arabic*)}]
        \item 
        every cospan has a pullback,
        \item
        \label{a:reg:fact}
        every span has a (strong epic, jointly monic) factorisation, and
        \item
        strong epimorphisms are pullback stable.
    \end{enumerate}
\end{definition}

\begin{definition}\label{d:epiregular}
	An independence category is \textit{epi-regular} if
	\begin{enumerate}[label={(E\arabic*)}]
    \item 
    \label{a:pullback}
    every cospan has an independent pullback,
    \item
    \label{a:factorisation}
    every span has a (strong epic, jointly monic) factorisation, and
    \item
    \label{a:epic}
    every morphism is strong epic.
	\end{enumerate}
 	An independence functor between epi-regular independence categories is \emph{epi-regular} if it preserves independent pullbacks and jointly monic spans.
    
    Write $\RegInd$ for the category of epi-regular independence categories and epi-regular independence functors.
\end{definition}

The axioms for an epi-regular independence category are very similar to the assumptions on a small category \(\C\) that~\textcite[Sec.~7]{simpson:equivalence-conditional-independence} used to define his \textit{atomic conditional independence relation} for atomic sheaves on \(\C\). In particular, \citeauthor{simpson:equivalence-conditional-independence} assumes that (i)
\(\C\) comes equipped with a system of independent pullbacks, (ii) each span in \(\C\) has a so-called \textit{pairing}, and (iii) each morphism in \(\C\) is epic.\todo{When we write a conclusion, maybe this better belongs there?} The notion of epi-regular independence category is also reminiscent of \citeauthor{simpson:category-theoretic-structure-independence}'s notion of \textit{conditional independence structure}~\cite{simpson:category-theoretic-structure-independence}.

\begin{proposition}\label{p:coisom}
	The functor $\Coisom$ restricts to a functor $\DilDag\to\RegInd$.
\end{proposition}
\begin{proof}
	If $\D$ is a dilatory dagger category, then the independence category $\Coisom(\D)$ is epi-regular by \cref{p:pullbacks-coisom,p:coisometry-strong-epic,p:factorisation}.
	If $G$ is a dilatory dagger functor, then the independence functor $\Coisom(G)$ is epi-regular by \cref{p:dildag-indep,p:pullbacks-coisom,p:factorisation}. 
\end{proof}

The following partial converse of \cref{p:rev-descent} has previously~\cite[Def.~6.1]{simpson:equivalence-conditional-independence} been called the \textit{descent} property.

\begin{lemma}
    \label{p:descent}
    In an epi-regular independence category,
\[
    \begin{tikzcd}[cramped]
        \pob
            \arrow[ddrr, independent]
            \arrow[d, "h" swap]
            \arrow[r, "h"]
            \&
        \pob
            \arrow[r, "g"]
            \&
        \pob
            \arrow[dd, "v"]
        \\
        \pob
            \arrow[d, "f" swap]
            \&
            \&
        \\
        \pob
            \arrow[rr, "u" swap]
            \&
            \&
        \pob
    \end{tikzcd}
    \qquad\text{implies}\qquad
    \begin{tikzcd}[cramped, sep=large]
        \pob
            \arrow[r, "g"]
            \arrow[d, "f" swap]
            \arrow[dr, independent]
            \&
        \pob
            \arrow[d, "v"]
        \\
        \pob
            \arrow[r, "u" swap]
            \&
        \pob
    \end{tikzcd}
    .
\]
\end{lemma}

\begin{proof}
Let \((p_1, p_2)\) be an independent pullback of \((u, v)\). Let \(e\) be the unique morphism such that the diagram
\[
    \begin{tikzcd}[cramped]
        \&
        \&
        \pob
            \arrow[dd, "e"]
            \arrow[dl, "h" swap]
            \arrow[dr, "h"]
        \&
        \&
    \\
        \&
        \pob
            \arrow[dl, "f" swap]
        \&
        \&
        \pob
            \arrow[dr, "g"]
        \&
    \\
        \pob
        \&
        \&
        \pob
            \arrow[ll, "p_1"]
            \arrow[rr, "p_2" swap]
        \&
        \&
        \pob
    \end{tikzcd}
\]
is commutative. Now \(h\) is strong epic by axiom \cref{a:epic}, the span \((p_1, p_2)\) is jointly monic by \cref{p:ind-pull-is-joint-monic}, and the diagram
\[
    \begin{tikzcd}[sep={0}, cramped]
        \pob
            \arrow[r, "e"]
            \arrow[d, "h" swap]
            \&[4.8em]
        \pob
            \arrow[d, "p_1" swap, shift right]
            \arrow[ddr, "p_2"]
            \&[0.3em]
        \\[3.6em]
        \pob
            \arrow[r, "f", shift left]
            \arrow[rrd, "g" swap]
            \arrow[ur, dashed, "d"]
            \&
        \pob
            \&
        \\[0.1em]
            \&
            \&
        \pob
    \end{tikzcd}
\]
is commutative, so it has a diagonal filler \(d\). The result now follows from \cref{p:rev-descent}.
\end{proof}

The following partial converse to \cref{p:ind-pull-is-joint-monic} is an analogue of \cref{p:jointly-monic-is-dilator} for epi-regular independence categories.

\begin{lemma}
\label{p:joint-monic-is-ind-pull}
In an epi-regular independence category, an independent square
\[
    \begin{tikzcd}[cramped, sep=large]
        \pob
            \arrow[r, "m_2"]
            \arrow[d, "m_1" swap]
            \arrow[dr, independent]
            \&
        \pob
            \arrow[d, "v"]
        \\
        \pob
            \arrow[r, "u" swap]
            \&
        \pob
    \end{tikzcd}
\]
is an independent pullback if the span \((m_1, m_2)\) is jointly monic.
\end{lemma}

\begin{proof}
Let \((p_1, p_2)\) be an independent pullback of \((u, v)\), and let \(m\) be the unique morphism such that
\[
    \begin{tikzcd}[cramped, sep=large]
            \&
        \pob
            \arrow[dl, "m_1" swap]
            \arrow[dr, "m_2"]
            \arrow[d, "m"]
            \&
        \\
        \pob
            \&
        \pob
            \arrow[l, "p_1"]
            \arrow[r, "p_2" swap]
            \&
        \pob
    \end{tikzcd}
\]
is commutative. As \(m\) is strong epic and \((m_1, m_2)\) is jointly monic, the diagonal filler \(d\) in the diagram
\[
    \begin{tikzcd}[sep={0}, cramped]
        \pob
            \arrow[r, "1"]
            \arrow[d, "m" swap]
            \&[4.8em]
        \pob
            \arrow[d, "m_1" swap, shift right]
            \arrow[ddr, "m_2"]
            \&[0.3em]
        \\[3.6em]
        \pob
            \arrow[r, "p_1", shift left]
            \arrow[rrd, "p_2" swap]
            \arrow[ur, dashed, "d"]
            \&
        \pob
            \&
        \\[0.1em]
            \&
            \&
        \pob
    \end{tikzcd}
\]
exists. Also \((p_1, p_2)\) is jointly monic by \cref{p:ind-pullback-identity}, and \(1\) is strong epic, so the diagonal filler \(d'\) in the diagram
\[
    \begin{tikzcd}[sep={0}, cramped]
        \pob
            \arrow[r, "1"]
            \arrow[d, "m" swap]
            \&[4.8em]
        \pob
            \arrow[d, "m_1" swap, shift right]
            \arrow[ddr, "m_2"]
            \&[0.3em]
        \\[3.6em]
        \pob
            \arrow[r, "p_1", shift left]
            \arrow[rrd, "p_2" swap]
            \arrow[from=ur, dashed, "d'" swap]
            \&
        \pob
            \&
        \\[0.1em]
            \&
            \&
        \pob
    \end{tikzcd}
\]
also exists. Now \(d'd = 1\) because \((m_1, m_2)\) is jointly monic, and \(dd' = 1\) because \((p_1, p_2)\) is jointly monic, so \(d\) is actually invertible with \(d^{-1} = d'\). Hence \((m_1, m_2)\) is another independent pullback of \((u, v)\).
\end{proof}

We are now ready to prove that epi-regular independence categories satisfy an independent-pullback analogue of the well-known pullback lemma.

\begin{lemma}[Independent-pullback lemma]
\label{p:pull-pasting}
Consider a diagram
\[
    \begin{tikzcd}[cramped, sep=large]
        \pob
            \arrow[r, "a"]
            \arrow[d, "f" swap]
            \arrow[dr, independent]
            \&
        \pob
            \arrow[r, "b"]
            \arrow[d, "g"]
            \arrow[dr, independent]
            \&
        \pob
            \arrow[d, "h"]
        \\
        \pob
            \arrow[r, "s" swap]
            \&
        \pob
            \arrow[r, "t" swap]
            \&
        \pob
    \end{tikzcd}
\]
in an epi-regular independence category. If the right-hand square is an independent pullback, then the left-hand square is an independent pullback if and only if the outer rectangle is an independent pullback.
\end{lemma}

\begin{proof}
By \cref{p:ind-pull-is-joint-monic,p:joint-monic-is-ind-pull}, it suffices to show that \((f, a)\) is jointly monic if and only if \((f, ba)\) is jointly monic.

For the \textit{if} direction, suppose that \((f, ba)\) is jointly monic. Let \(u\) and \(v\) be parallel morphisms such that
\[fu = fv \qquad\text{and}\qquad au = av.\]
Then \(bau = bav\) and thus \(u = v\) because \((f, ba)\) is jointly monic.

For the \textit{only if} direction, suppose that \((f, a)\) is jointly monic. Let \(u\) and \(v\) be parallel morphisms such that
\[fu = fv \qquad\text{and}\qquad bau = bav.\]
As independent squares are commutative,
\[gau = sfu = sfv = gav.\]
By \cref{p:ind-pull-is-joint-monic}, the span \((g, b)\) is jointly monic, so it follows that \(au = av\). As such, \(u = v\) since the span \((f, a)\) is jointly monic.
\end{proof}

\begin{lemma}
\label{p:monos}
The following properties of a morphism \(m\) in an epi-regular independence category are equivalent:
\begin{enumerate}
    \item
    \(m\) is monic,
    \item there exists a morphism \(f\) such that
    \[
    \begin{tikzcd}[cramped, sep=large]
        \pob
            \arrow[r, "f"]\arrow[d, "f" swap]\arrow[dr, independent]
            \&
        \pob
            \arrow[d, "m"]
            \\
        \pob
            \arrow[r, "m" swap]
            \&
        \pob
        \end{tikzcd},
    \]
    \item the span \((1, 1)\) is an independent kernel pair of \(m\), and
    \item \(m\) is invertible.
\end{enumerate}
\end{lemma}
    
\begin{proof}
Let \((k_1, k_2)\) be an independent kernel pair of \(m\). Then \(mk_1 = mk_2\).

For (i) implies (ii), if \(m\) is monic, then \(k_1 = k_2\), so we may set \(f = k_1\).

For (ii) implies (iii), suppose that such a morphism \(f\) exists. The square 
\[
    \begin{tikzcd}[cramped, sep=large]
        \pob
            \arrow[r, "1"]\arrow[d, "1" swap]\arrow[dr, independent]
            \&
        \pob
            \arrow[d, "m"]
            \\
        \pob
            \arrow[r, "m" swap]
            \&
        \pob
    \end{tikzcd}
\]
is independent by \cref{p:descent}, and so is actually an independent pullback square by \cref{p:joint-monic-is-ind-pull} because \((1, 1)\) is jointly monic.

For (iii) implies (i), suppose that \((1, 1)\) is an independent kernel pair of \(m\). Let \(u\) and \(v\) be parallel morphisms such that \(mu = mv\). Then the pasting
\begin{equation}
    \label{e:pasting-kernel-eq}
    \begin{tikzcd}[cramped, sep=large]
        \pob
            \arrow[r, "u"]
            \arrow[d, "1" swap]
            \arrow[dr, independent]
            \&
        \pob
            \arrow[r, "1"]
            \arrow[d, "1"]
            \arrow[dr, independent]
            \&
        \pob
            \arrow[d, "m"]
        \\
        \pob
            \arrow[r, "u"]
            \arrow[d, "1" swap]
            \arrow[drr, independent]
            \&
        \pob
            \arrow[r, "m"]
            \&
        \pob
            \arrow[d, "1"]
        \\
        \pob
            \arrow[r, "v"]
            \arrow[d, "v" swap]
            \arrow[dr, independent]
            \&
        \pob
            \arrow[r, "m"]
            \arrow[d, "1"]
            \arrow[dr, independent]
            \&
        \pob
            \arrow[d, "1"]
        \\
        \pob
            \arrow[r, "1" swap]
            \&
        \pob
            \arrow[r, "m" swap]
            \&
        \pob
    \end{tikzcd}
\end{equation}
is independent. Hence there is a unique morphism \(w\) such that the diagram
\[
    \begin{tikzcd}[cramped, sep=large]
            \&
        \pob
            \arrow[dl, "u" swap]
            \arrow[dr, "v"]
            \arrow[d, "w"]
            \&
        \\
        \pob
            \&
        \pob
            \arrow[l, "1"]
            \arrow[r, "1" swap]
            \&
        \pob
    \end{tikzcd}
\]
is commutative. In particular, \(u = w = v\).

The equivalence of (i) and (iv) follows from \cref{p:strong-epic-props} (iii) because \(m\) is strong epic by \cref{a:epic}.
\end{proof}

In epi-regular independence categories, every morphism is a regular epimorphism.

\begin{proposition}
\label{p:strong-is-regular}
In an epi-regular independence category, every morphism is the coequaliser of its independent kernel pair.
\end{proposition}

The following proof generalises the \textit{if} direction of \cite[Prop.~2.2.2~(2)]{borceux:1994:handbook-categorical-algebra-2}.

\begin{proof}
Let \(f\) be a morphism, let \((k_1, k_2)\) be an independent kernel pair of \(f\), and let \(g\) be a morphism such that \(gk_1 = gk_2\). Let \((e, m_1, m_2)\) be a (strong epic, jointly monic) factorisation of the span \((f, g)\). We will show that \(m_1\) is invertible, so that \(m_2{m_1}^{-1}\) is the unique morphism \(h\) such that \(g = hf\). Let us construct the morphisms depicted in \cref{f:coeq-kern}.
\begin{diagram}
    \begin{tikzcd}[cramped, sep=large]
    \pob
        \arrow[r, "r_2" swap]
        \arrow[d, "r_1"]
        \arrow[rr, "k_2", bend left, shift left]
        \arrow[dd, "k_1" swap, bend right, shift right]
        \arrow[dr, independent]
        \&
    \pob
        \arrow[r, "q_2" swap]
        \arrow[d, "q_1"]
        \arrow[dr, independent]
        \&
    \pob
        \arrow[d, "e" swap]
        \arrow[dd, "f", bend left, shift left]
    \\
    \pob
        \arrow[r, "p_2"]
        \arrow[d, "p_1"]
        \arrow[dr, independent]
        \&
    \pob
        \arrow[r, "j_2"]
        \arrow[d, "j_1"]
        \arrow[dr, independent]
        \&
    \pob
        \arrow[d, "m_1" swap]
    \\
    \pob
        \arrow[r, "e"]
        \arrow[rr, "f" swap, bend right, shift right]
        \&
    \pob
        \arrow[r, "m_1"]
        \&
    \pob
    \end{tikzcd}
    \caption{}
    \label{f:coeq-kern}
\end{diagram}
Let \((j_1, j_2)\) be an independent kernel pair of \(m_1\), let \((p_1, p_2)\) be an independent pullback of \((e, j_1)\), and let \((q_1, q_2)\) be an independent pullback of \((j_2, e)\). By \cref{p:pull-pasting}, there is an independent pullback \((r_1, r_2)\) of \((p_2, q_1)\) such that \(k_1 = p_1r_1\) and \(k_2 = q_2r_2\). Then
\begin{align*}
    m_1j_1p_2r_1 &= m_1ep_1r_1 = fk_1 = fk_2 = m_1eq_2r_2 = m_1j_2q_1r_2 = m_1j_2p_2r_1,
    \\\shortintertext{and}
    m_2j_1p_2r_1 &= m_2ep_1r_1 = gk_1 = gk_2 = m_2eq_2r_2 = m_2j_2q_1r_2 = m_2j_2p_2r_1.
\end{align*}
However \((m_1, m_2)\) is jointly monic, so \(j_1p_2r_1 = j_2p_2r_1\). Also \(p_2\) and \(r_1\) are epic by axiom \cref{a:epic}, so actually \(j_1 = j_2\). By \cref{p:monos}, it follows that \(m_1\) is invertible.
\end{proof}

Combining \cref{p:strong-is-regular} and \cref{p:reg-is-strong}~(iv) yields the following alternative characterisation of epi-regular independence categories.

\begin{corollary}
\label{d:epi-regular-alt}
An independence category is epi-regular if and only if 
\begin{enumerate}[label={(E\arabic*')}]
    \item \label{a:pullback-2} every cospan has an independent pullback,
    \item \label{a:fact-2} every span has a (regular epic, jointly monic) factorisation, and
    \item \label{a:coeq-kern-pair} every morphism is the coequaliser of its independent kernel pair.
\end{enumerate}
\end{corollary}

\section{The equivalence}
\label{s:equivalence}

We are now ready to show that \textit{epi-regular independence category} and \textit{dilatory dagger category} are equivalent concepts. First, in \cref{s:one-equivalence}, we show that the categories \(\RegInd\) and \(\DilDag\) are equivalent. We upgrade this to a strict 2-equivalence of 2-categories in \cref{s:two-equivalence}.

\subsection{Equivalence of categories}\label{s:one-equivalence}

The goal of this section is to prove that the functor $\Coisom \colon \DilDag \to \RegInd$ is an equivalence of categories. To do this, we will explicitly construct its adjoint pseudo-inverse.

\begin{proposition}
\label{p:rel-exists}
If \(\C\) is an epi-regular independence category, then the following data defines a dagger category \(\Rel(\C)\).
\begin{itemize}
    \item The objects are the objects of \(\C\).
    \item The morphisms \(X \to Y\) are the \emph{relations} in \(\C\) from \(X\) to \(Y\), that is, the isomorphism classes of jointly monic spans in \(\C\) from \(X\) to \(Y\). Write \([m_1, m_2]\) for the relation represented by the jointly monic span \((m_1, m_2)\). 
    \item The identity morphism on an object \(X\) is \([1_X, 1_X]\).
    \item Composition is defined by \([s_1, s_2][r_1, r_2] = [m_1, m_2]\), where \((p_1, p_2)\) is an independent pullback of \((r_2, s_1)\), and \((e, m_1, m_2)\) is a (strong epic, jointly monic) factorisation of \((r_1p_1, s_2p_2)\), as depicted in \cref{f:relation-composition}.
    \begin{diagram}
    \centering
    \begin{minipage}[b]{0.5\textwidth}
        \centering
        \begin{tikzcd}[cramped, row sep={between origins,1.7em}, column sep={between origins,4em}, labels={inner sep=0.3ex}]
            \&
            \&
            \pob
                \arrow[ddl, "p_1" swap]
                \arrow[ddr, "p_2"]
                \arrow[ddd, "e"{inner sep=0.5ex,pos=0.6}, shorten >=-1ex]
            \&
            \&
        \\
            \&
            \&
            \&
            \&
        \\
            \&
            \pob
                \arrow[ddl, "r_1" swap]
                \arrow[ddr, "r_2"{pos=0.3}]
            \&
            \&
            \pob
                \arrow[ddl, "s_1"{swap, pos=0.3}]
                \arrow[ddr, "s_2"]
            \&
        \\[-0.2em]
            \&
            \&
            \pob
            \arrow[dll, "m_1"{pos=0.6}, shift left, crossing over]
            \&
            \&
        \\[0.2em]
            \pob
            \&
            \&
            \pob
            \&
            \&
            \pob
            \arrow[from=ull, "m_2"{swap, pos=0.6}, shift right, crossing over]
        \end{tikzcd}
        \caption{}
        \label{f:relation-composition}
    \end{minipage}%
    \begin{minipage}[b]{0.5\textwidth}
        \centering
        \begin{tikzcd}[cramped, row sep={between origins,1.7em}, column sep={between origins,4em}, labels={inner sep=0.3ex}]
            \&
            \&
            \pob
                \arrow[ddl, "q_1" swap, shorten >=-1ex]
                \arrow[ddr, "q_2", shorten >=-1ex]
                \arrow[d, "w"{inner sep=0.5ex, pos=0.7}, shorten >=-1ex]
            \&
            \&
        \\[0.8em]
            \&
            \&
            \pob
                \arrow[ddl, "p_1"]
                \arrow[ddr, "p_2" swap]
                \arrow[ddd, "e"{inner sep=0.5ex,pos=0.7}, shorten >=-1ex]
            \&
            \&
        \\[-0.8em]
            \&
            \pob
                \arrow[d, "u"{inner sep=0.5ex, swap}]
            \&
            \&
            \pob
                \arrow[d, "v"{inner sep=0.5ex}]
            \&
        \\[0.8em]
            \&
            \pob
                \arrow[ddl, "r_1" swap]
                \arrow[ddr, "r_2"{pos=0.3}]
            \&
            \&
            \pob
                \arrow[ddl, "s_1"{swap, pos=0.3}]
                \arrow[ddr, "s_2"]
            \&
        \\[-0.2em]
            \&
            \&
            \pob
            \arrow[dll, "m_1"{pos=0.6}, shift left, crossing over]
            \&
            \&
        \\[0.2em]
            \pob
            \&
            \&
            \pob
            \&
            \&
            \pob
            \arrow[from=ull, "m_2"{swap, pos=0.6}, shift right, crossing over]
        \end{tikzcd}
        \caption{}
        \label{f:composition-is-defined}
    \end{minipage}
    \end{diagram}
    \item The dagger is defined by \([r_1, r_2]^\dagger = [r_2, r_1]\).
\end{itemize}
Also, if \(F \colon \C \to \C'\) is an epi-regular independence functor, then the following data defines a dagger functor \(\Rel(F) \colon \Rel(\C) \to \Rel(\C')\).
	\begin{itemize}
		\item The object map is the object map of \(F\).
		\item The morphism map sends \([r_1, r_2]\) to \([Fr_1, Fr_2]\).
	\end{itemize} 
\end{proposition}

\begin{proof}
First, consider an epi-regular independence category \(\C\).

For composition in \(\Rel(\C)\) to be well-defined, it should be independent of choices of relation representatives, independent pullbacks, and span factorisations. Let \((r_1, r_2)\), \((s_1, s_2)\), \((p_1, p_2)\) and \((e, m_1, m_2)\) be defined as in the proposition statement. Also let \(u\) and \(v\) be isomorphisms into the apices of the spans \((r_1, r_2)\) and \((s_1, s_2)\) respectively, and let \((q_1, q_2)\) be an independent pullback of \((r_2u, s_1 v)\), as depicted in \cref{f:composition-is-defined}. Universality of independent pullbacks yields a unique morphism \(w\) such that \(p_1w = uq_1\) and \(p_2w = vq_2\). Observe that \((ew, m_1, m_2)\) is a (strong epic, jointly monic) factorisation of \((r_1uq_1, s_2vq_2)\). Because (strong epic, jointly monic) factorisations are unique up to isomorphism, composition in $\Rel(\C)$ is indeed well-defined.

Let us now check that composition in \(\Rel(\C)\) is associative. Let \([r_1, r_2]\), \([s_1, s_2]\) and \([t_1, t_2]\) be composable relations, as depicted in \cref{f:rel-assoc}.
\begin{diagram}
    \begin{tikzcd}[cramped, row sep={between origins,1.7em}, column sep={between origins,4em}, labels={inner sep=0.3ex}]
        \&
        \&
        \&
        \pob
            \arrow[ddl, "w_1" swap]
            \arrow[ddr, "w_2"{pos=0.4}]
            \arrow[ddd, "d"{inner sep=0.5ex, swap}, shorten >=-0.5ex]
        \&
        \&
        \&
    \\
        \&
        \&
        \&
        \&
        \&
        \&
    \\
        \&
        \&
        \pob
            \arrow[ddl, "u_1" swap]
            \arrow[ddr, "u_2"{pos=0.4}]
            \arrow[ddd, "e"{inner sep=0.5ex, swap}, shorten >=-0.5ex]
        \&
        \&
        \pob
            \arrow[ddl, "v_1"{pos=0.3, swap}]
            \arrow[ddr, "v_2"]
        \&
        \&
    \\[-0.5em]
        \&
        \&
        \&
            \pob
            \arrow[ddl, "p_1"{swap, pos=0.6}, crossing over, shift left]
            \arrow[drr, "p_2" swap, crossing over, shift right, shorten >=0.5ex]
        \&
        \&
        \&
    \\[0.5em]
        \&
        \pob
            \arrow[ddl, "r_1" swap]
            \arrow[ddr, "r_2"{pos=0.3}]
        \&
        \&
        \pob
            \arrow[ddl, "s_1"{pos=0.2}]
            \arrow[ddr, "s_2"]
        \&
        \&
        \pob
            \arrow[ddl, "t_1"{swap}, shorten <=0.5ex]
            \arrow[ddr, "t_2"]
        \&
    \\[-0.5em]
        \&
        \&
        \pob
        \arrow[dll, "m_1"{pos=0.6}, shift left, crossing over]
        \&
        \&
        \&
        \&
    \\[0.5em]
        \pob
        \&
        \&
        \pob
        \&
        \&
        \pob
        \arrow[from=ull, "m_2"{swap, pos=0.6}, shift right, crossing over]
        \&
        \&
        \pob
    \end{tikzcd}
    \caption{}
    \label{f:rel-assoc}
\end{diagram}
Let \((u_1, u_2)\) be an independent pullback of \((r_2, s_1)\), let \((v_1, v_2)\) be an independent pullback of \((s_2, t_1)\), and let \((w_1, w_2)\) be an independent pullback of \((u_2, v_1)\); these are also depicted in \cref{f:rel-assoc}. We will show that the composite relation \([t_1, t_2] ([s_1, s_2][r_1, r_2])\) is represented by the jointly monic part of a (strong epic, jointly monic) factorisation of \((r_1u_1w_1, t_2v_2w_2)\). By symmetry, so is the composite relation \(([t_1, t_2][s_1, s_2])[r_1, r_2]\). The result then follows by uniqueness of (strong epic, jointly monic) factorisations.

Let \((e, m_1, m_2)\) be a (strong epic, jointly monic) factorisation of \((r_1u_1, s_2u_2)\), and let \((p_1, p_2)\) be an independent pullback of \((m_2, t_1)\), as depicted in \cref{f:rel-assoc}. Also let \((c, n_1, n_2)\) be a (strong epic, jointly monic) factorisation of \((m_1p_1, t_2p_2)\). Then
\[
    [t_1, t_2] ([s_1, s_2][r_1, r_2]) = [t_1, t_2][m_1, m_2] = [n_1, n_2].
\]
Now, the pasting
\[
    \begin{tikzcd}[cramped, sep=large]
        \pob
            \arrow[r, "w_1"]
            \arrow[d, "w_1" swap]
            \arrow[dr, independent]
            \&
        \pob
            \arrow[r, "v_2"]
            \arrow[d, "v_1"]
            \arrow[dr, independent]
            \&
        \pob
            \arrow[d, "t_1"]
        \\
        \pob
            \arrow[r, "u_2"]
            \arrow[d, "1" swap]
            \arrow[drr, independent]
            \&
        \pob
            \arrow[r, "s_2"]
            \&
        \pob
            \arrow[d, "1"]
        \\
        \pob
            \arrow[r, "e"]
            \arrow[d, "e" swap]
            \arrow[dr, independent]
            \&
        \pob
            \arrow[r, "m_2"]
            \arrow[d, "1"]
            \arrow[dr, independent]
            \&
        \pob
            \arrow[d, "1"]
        \\
        \pob
            \arrow[r, "1" swap]
            \&
        \pob
            \arrow[r, "m_2" swap]
            \&
        \pob
    \end{tikzcd}
\]
is independent, so universality of independent pullbacks yields a unique morphism \(d\) (depicted in \cref{f:rel-assoc}) such that \(p_1d = ew_1\) and \(p_2d = v_2w_2\). Hence \((cd, n_1, n_2)\) is a (strong epic, jointly monic) factorisation of \((r_1u_1w_1, t_2v_2w_2)\), as required.

Unitality of composition in \(\Rel(\C)\) follows from \cref{p:ind-pullback-identity}. The fact that the dagger of \(\Rel(\C)\) is well-defined, functorial and involutive is trivial.

Next, consider an epi-regular independence functor \(F \colon \C \to \C'\). For \(\Rel(F)\) to be well-defined, \((r_1, r_2)\) being jointly monic must imply \((Fr_1, Fr_2)\) being jointly monic; also the isomorphism class \([Fr_1, Fr_2]\) must be independent of the chosen representative \((r_1, r_2)\) of the isomorphism class \([r_1, r_2]\). The former is true because \(F\) preserves (strong epic, jointly monic) factorisations of spans. The latter is true because \(F\) is functorial, and so preserves isomorphisms. Clearly \(\Rel(F)\) preserves identity morphisms. It preserves composition because it preserves both independent pullbacks and (strong epic, jointly monic) factorisations of spans.
\end{proof}

To show that \(\Rel(\C)\) is dilatory, we need a better understanding of its coisometries.

\begin{lemma}
    \label{p:coisom-rel}
Let \(\C\) be an epi-regular independence category.
\begin{enumerate}
    \item A morphism in \(\Rel(\C)\) is coisometric if and only if it is of the form \([1, f]\).
    \item For all composable morphisms \(f\) and \(g\) in \(\C\),
    \[[1, gf] = [1, g][1, f].\]
    \item For all parallel morphisms \(f\) and \(g\) in \(\C\),
    \[[1, f] = [1, g] \qquad\text{implies}\qquad f = g.\]
    \item For all morphisms \(f\), \(g\), \(s\) and \(t\) in \(\C\),
    \[
        \begin{tikzcd}[cramped, sep=large]
            \pob
                \arrow[r, "g"]
                \arrow[d, "f" swap]
                \arrow[dr, independent]
                \&
            \pob
                \arrow[d, "v"]
            \\
            \pob
                \arrow[r, "u" swap]
                \&
            \pob
        \end{tikzcd}
        \qquad\text{if and only if}\qquad
        \begin{tikzcd}[cramped, sep=large]
            \pob
                \arrow[r, "{[1, g]}"]
                \arrow[d, "{[1, f]}" swap]
                \arrow[dr, independent]
                \&
            \pob
                \arrow[d, "{[1, v]}"]
            \\
            \pob
                \arrow[r, "{[1, u]}" swap]
                \&
            \pob
        \end{tikzcd}
        .
    \]
\end{enumerate}
\end{lemma}

\begin{proof}
For the \textit{if} direction of (i), consider a morphism \(f\) in \(\C\). First note that the span \((1, f)\) is jointly monic, so \([1, f]\) is indeed a relation. Now refer to \cref{f:coisom}.
\begin{diagram}
\centering
\begin{minipage}[b]{0.5\textwidth}
    \centering
    \begin{tikzcd}[cramped, row sep={between origins,1.7em}, column sep={between origins,4em}, labels={inner sep=0.3ex}]
        \&
        \&
        \pob
            \arrow[ddl, "1" swap]
            \arrow[ddr, "1"]
            \arrow[ddd, "f"{inner sep=0.5ex,pos=0.6}, shorten >=-1ex]
        \&
        \&
    \\
        \&
        \&
        \&
        \&
    \\
        \&
        \pob
            \arrow[ddl, "f" swap]
            \arrow[ddr, "1"{pos=0.3}]
        \&
        \&
        \pob
            \arrow[ddl, "1"{swap, pos=0.3}]
            \arrow[ddr, "f"]
        \&
    \\[-0.2em]
        \&
        \&
        \pob
        \arrow[dll, "1"{pos=0.6}, shift left, crossing over]
        \&
        \&
    \\[0.2em]
        \pob
        \&
        \&
        \pob
        \&
        \&
        \pob
        \arrow[from=ull, "1"{swap, pos=0.6}, shift right, crossing over]
    \end{tikzcd}
    \caption{}
    \label{f:coisom}
\end{minipage}%
\begin{minipage}[b]{0.5\textwidth}
    \centering
    \begin{tikzcd}[cramped, row sep={between origins,1.7em}, column sep={between origins,4em}, labels={inner sep=0.3ex}]
        \&
        \&
        \pob
            \arrow[ddl, "k_1" swap]
            \arrow[ddr, "k_2"]
            \arrow[ddd, "e"{inner sep=0.5ex,pos=0.6}, shorten >=-1ex]
        \&
        \&
    \\
        \&
        \&
        \&
        \&
    \\
        \&
        \pob
            \arrow[ddl, "r_2" swap]
            \arrow[ddr, "r_1"{pos=0.3}]
        \&
        \&
        \pob
            \arrow[ddl, "r_1"{swap, pos=0.3}]
            \arrow[ddr, "r_2"]
        \&
    \\[-0.2em]
        \&
        \&
        \pob
        \arrow[dll, "1"{pos=0.6}, shift left, crossing over]
        \&
        \&
    \\[0.2em]
        \pob
        \&
        \&
        \pob
        \&
        \&
        \pob
        \arrow[from=ull, "1"{swap, pos=0.6}, shift right, crossing over]
    \end{tikzcd}
    \caption{}
    \label{f:coisom-rev}
\end{minipage}
\end{diagram}
The span \((1, 1)\) is an independent pullback of the cospan \((1, 1)\) by \cref{p:ind-pullback-identity}. Also \((f, 1, 1)\) is a (strong epic, jointly monic) factorisation of \((f1, f1)\). Hence \([1, f][1, f]^\dagger = [1, 1]\).

For the \textit{only if} direction of (i), consider a coisometry \([r_1, r_2]\) in \(\Rel(\C)\). Let \((k_1, k_2)\) be an independent kernel pair of \(r_1\). As \([r_1, r_2]\) is coisometric, the span \((r_2k_1, r_2k_2)\) factors through the span \((1, 1)\) via a strong epimorphism~\(e\), 
as depicted in \cref{f:coisom-rev}. Hence \(r_2k_1 = e = r_2k_2\). Thus \(r_1k_1 = r_1k_2\) and since \((r_1, r_2)\) is jointly monic, we get that \(k_1 = k_2\). By \cref{p:monos}, it follows that \(r_1\) is invertible. Hence \([r_1, r_2] = [1, r_2 {r_1}^{-1}]\).

For (ii), refer to \cref{f:rel-coisom-comp}.
\begin{diagram}
    \begin{tikzcd}[cramped, row sep={between origins,1.7em}, column sep={between origins,4em}, labels={inner sep=0.3ex}]
        \&
        \&
        \pob
            \arrow[ddl, "1" swap]
            \arrow[ddr, "f"]
            \arrow[ddd, "1"{inner sep=0.5ex,pos=0.6}, shorten >=-1ex]
        \&
        \&
    \\
        \&
        \&
        \&
        \&
    \\
        \&
        \pob
            \arrow[ddl, "1" swap]
            \arrow[ddr, "f"{pos=0.3}]
        \&
        \&
        \pob
            \arrow[ddl, "1"{swap, pos=0.3}]
            \arrow[ddr, "g"]
        \&
    \\[-0.2em]
        \&
        \&
        \pob
        \arrow[dll, "1"{pos=0.6}, shift left, crossing over]
        \&
        \&
    \\[0.2em]
        \pob
        \&
        \&
        \pob
        \&
        \&
        \pob
        \arrow[from=ull, "gf"{swap, pos=0.6}, shift right, crossing over]
    \end{tikzcd}
    \caption{}
    \label{f:rel-coisom-comp}
\end{diagram}
The span \((1, f)\) is an independent pullback of the cospan \((f, 1)\) by \cref{p:ind-pullback-identity}. Also \((1, 1, gf)\) is a (strong epic, jointly monic) factorisation of the span \((11, gf)\). Hence \([1, gf] = [1, g][1, f]\), as required.

For (iii), if \([1, f] = [1, g]\) then there is an isomorphism \(u\) such that \(1 = 1u\) and \(f = gu\). This means that \(f = gu = g1 = g\), as required.

For (iv), first note that \([1, v][1, g] = [1, u][1, f]\) if and only if \(vg = uf\); this follows from parts (ii) and (iii). Now let \((p_1, p_2)\) be an independent pullback of the cospan \((u, v)\). By \cref{p:ind-pull-is-joint-monic}, the span \((p_1, p_2)\) is jointly monic. Thus \([1, v]^\dagger[1, u] = [p_1, p_2]\). Hence \([1, g][1, f]^\dagger = [1, u]^\dagger[1, v]\) exactly when the span \((f, g)\) factors through the span \((p_1, p_2)\) via a strong epimorphism. By \cref{p:rev-descent} and the definition of independent pullback, this is true if and only if the square \((f, g, u, v)\) is independent.
\end{proof}

\begin{proposition}
\label{p:rel-dilatory}
If \(\C\) is an epi-regular independence category, then the dagger category \(\Rel(\C)\) is dilatory: the dilator of a relation \([r_1, r_2]\) is the span \(([1, r_1], [1, r_2])\). Also, if \(F \colon \C \to \C'\) is an epi-regular independence functor, then the dagger functor \(\Rel(F) \colon \Rel(\C) \to \Rel(\C')\) is dilatory. In fact \(\Rel\) is a functor \(\RegInd \to \DilDag\).
\end{proposition}

\begin{proof}
First, we show that \(([1, r_1], [1, r_2])\) is a dilation of \([r_1, r_2]\). By \cref{p:coisom-rel}~(i), the relations \([1, r_1]\) and \([1, r_2]\) are coisometric. As shown in \cref{f:dilator-1},
\begin{diagram}
\centering
\begin{minipage}[b]{0.5\textwidth}
    \centering
    \begin{tikzcd}[cramped, row sep={between origins,1.7em}, column sep={between origins,4em}, labels={inner sep=0.3ex}]
        \&
        \&
        \pob
            \arrow[ddl, "1" swap]
            \arrow[ddr, "1"]
            \arrow[ddd, "1"{inner sep=0.5ex,pos=0.6}, shorten >=-1ex]
        \&
        \&
    \\
        \&
        \&
        \&
        \&
    \\
        \&
        \pob
            \arrow[ddl, "r_1" swap]
            \arrow[ddr, "1"{pos=0.3}]
        \&
        \&
        \pob
            \arrow[ddl, "1"{swap, pos=0.3}]
            \arrow[ddr, "r_2"]
        \&
    \\[-0.2em]
        \&
        \&
        \pob
        \arrow[dll, "r_1"{pos=0.6}, shift left, crossing over]
        \&
        \&
    \\[0.2em]
        \pob
        \&
        \&
        \pob
        \&
        \&
        \pob
        \arrow[from=ull, "r_2"{swap, pos=0.6}, shift right, crossing over]
    \end{tikzcd}
    \caption{}
    \label{f:dilator-1}
\end{minipage}%
\begin{minipage}[b]{0.5\textwidth}
    \centering
    \begin{tikzcd}[cramped, row sep={between origins,1.7em}, column sep={between origins,4em}, labels={inner sep=0.3ex}]
        \&
        \&
        \pob
            \arrow[ddl, "1" swap]
            \arrow[ddr, "1"]
            \arrow[ddd, "e"{inner sep=0.5ex,pos=0.6}, shorten >=-1ex]
        \&
        \&
    \\
        \&
        \&
        \&
        \&
    \\
        \&
        \pob
            \arrow[ddl, "s_1" swap]
            \arrow[ddr, "1"{pos=0.3}]
        \&
        \&
        \pob
            \arrow[ddl, "1"{swap, pos=0.3}]
            \arrow[ddr, "s_2"]
        \&
    \\[-0.2em]
        \&
        \&
        \pob
        \arrow[dll, "r_1"{pos=0.6}, shift left, crossing over]
        \&
        \&
    \\[0.2em]
        \pob
        \&
        \&
        \pob
        \&
        \&
        \pob
        \arrow[from=ull, "r_2"{swap, pos=0.6}, shift right, crossing over]
    \end{tikzcd}
    \caption{}
    \label{f:dilator-2}
\end{minipage}
\end{diagram}
the span \((1, 1)\) is an independent pullback of the cospan \((1, 1)\) by \cref{p:ind-pullback-identity}, and \((1, r_1, r_2)\) is a (strong epic, jointly monic) factorisation of \((r_11, r_21)\). Hence \([1, r_2][1, r_1]^\dagger = [r_1, r_2]\).

To see that \(([1, r_1], [1, r_2])\) is terminal, consider another dilation of \([r_1, r_2]\). By \cref{p:coisom-rel}~(i), it is of the form \(([1, s_1], [1, s_2])\). If a comparison morphism exists, then it is of the form \([1, e]\), also by \cref{p:coisom-rel}~(i). Hence \([1, s_1] = [1, r_1][1, e] = [1, r_1e]\) by \cref{p:coisom-rel}~(ii). Thus \(s_1 = r_1 e\) by \cref{p:coisom-rel}~(iii). Similarly \(s_2 = r_2e\). Hence \((e, r_1, r_2)\) is a (strong epic, jointly monic) factorisation of \((s_1, s_2)\). However \((r_1, r_2)\) is jointly monic, so this uniquely determines \(e\). We now prove existence. Refer to \cref{f:dilator-2}. The span \((1, 1)\) is again an independent pullback of the cospan \((1, 1)\). However \([1, s_2][1, s_1]^\dagger = [r_1, r_2]\), so the span \((s_11, s_21)\) factors through \((r_1, r_2)\) via a strong epimorphism \(e\). 
Hence \([1, r_1][1, e] = [1, s_1]\) and \([1, r_2][1, e] = [1, s_2]\) by \cref{p:coisom-rel}~(ii).

The remaining claims are easy to check.
\end{proof}

To show that \(\Rel\) is an adjoint pseudo-inverse of \(\Coisom\), we will construct the unit and counit of the adjunction, and show that they are natural isomorphisms.

\begin{proposition}
\label{p:regularish-is-coisom}
For each epi-regular independence category \(\C\), the map \(f \mapsto [1, f]\) on morphisms defines an isomorphism \(\eta_\C \colon \C \to \Coisom(\Rel(\C))\) in~\(\RegInd\). In fact \(\eta \colon 1_{\RegInd} \Rightarrow \Coisom \circ \Rel\) is a natural isomorphism.
\end{proposition}

\begin{proof}
With reference to \cref{p:coisom-rel}, \(\eta_\C\) is well-defined by the \textit{if} direction of (i), and is functorial by (ii). It is also full by the \textit{only if} direction of (i), and faithful by (iii). As it is the identity on objects, it thus has an inverse functor. Actually \(\eta_\C\) and \({\eta_\C}^{-1}\) are independence functors by (iv). However every isomorphism in \(\Ind\) preserves independent pullbacks and jointly monic spans, so \(\eta_\C\) and \({\eta_\C}^{-1}\) are in fact epi-regular. It is easy to verify the naturality of \(\eta\).
\end{proof}

\begin{proposition}
	\label{p:dilatory-is-rel}
	For each dilatory dagger category $\D$, the map \([r_1, r_2] \mapsto r_2 {r_1}^\dagger\) on morphisms defines an isomorphism \(\varepsilon_\D \colon \Rel(\Coisom(\D)) \to \D\) in the category~\(\DilDag\). In fact \(\varepsilon \colon \Rel \circ \Coisom \Rightarrow 1_{\DilDag}\) is a natural isomorphism.
\end{proposition}

\begin{proof}
    First, we show that \(\varepsilon_\D\) is well-defined. Let \((r_1, r_2)\) and \((s_1, s_2)\) be parallel jointly monic spans in \(\Coisom(\D)\). If \([r_1, r_2] = [s_1, s_2]\), then there is an isomorphism \(u\) in \(\Coisom(\D)\) such that \(s_1 = r_1u\) and \(s_2 = r_2u\). Hence \(r_2{r_1}^\dagger = r_2uu^\dagger {r_1}^\dagger = s_2 {s_1}^\dagger\).

    We now show that \(\varepsilon_\D\) is functorial. Let \([r_1, r_2]\) and \([s_1, s_2]\) be composable relations in \(\Coisom(\D)\), let \((p_1, p_2)\) be an independent pullback of the cospan \((r_2, s_1)\), and let \((e, m_1, m_2)\) be a (strong epic, jointly monic) factorisation of \((r_1p_1, s_2p_2)\), as depicted in \cref{f:composition-is-defined}. Then \([s_1, s_2][r_1, r_2] = [m_1, m_2]\). Recalling \cref{p:dildag-indep},
    \begin{align*}
        \varepsilon_\D([s_1, s_2][r_1, r_2])
        &= \varepsilon_\D([m_1, m_2])
        = m_2{m_1}^\dagger
        = m_2 ee^\dagger {m_1}^\dagger
        \\&= s_2 p_2 {p_1}^\dagger {r_1}^\dagger
        = s_2{s_1}^\dagger r_2 {r_1}^\dagger
        = \varepsilon_\D([s_1, s_2])\,\varepsilon_\D([r_1, r_2]).
    \end{align*}

	Next, we prove that \(\varepsilon_\D\) is faithful. Let \([r_1, r_2]\) and \([s_1, s_2]\) be parallel morphisms in \(\Rel(\Coisom(\D))\). Suppose that \(\varepsilon_\D([r_1, r_2]) = \varepsilon_\D([s_1, s_2])\). By \Cref{p:jointly-monic-is-dilator}, the spans \((r_1,r_2)\) and \((s_1, s_2)\) are both dilators of \(r_2 {r_1}^\dagger = s_2 {s_1}^\dagger\). Hence \([r_1, r_2] = [s_1, s_2]\).

	The functor \(\varepsilon_\D\) is also full. Indeed, each morphism \(f\) in \(\D\) has a dilator \((r_1, r_2)\), and \((r_1, r_2)\) is jointly monic in \(\Coisom(\D)\) by \cref{p:jointly-monic-is-dilator}, so \(f = \varepsilon_\D([r_1, r_2])\).

    Next, we show that \(\varepsilon_\D\) is dilatory. Consider a relation \([r_1, r_2]\) in \(\Coisom(\D)\). By \cref{p:rel-dilatory}, its dilator in \(\Rel(\Coisom(\D))\) is the span \(([1, r_1], [1, r_2])\). The span \((\varepsilon_\D([1, r_1]), \varepsilon_\D([1, r_2])) = (r_1, r_2)\) in \(\Coisom(\D)\) is jointly monic, so, by \cref{p:jointly-monic-is-dilator}, it is the dilator in \(\D\) of the morphism \(r_2{r_1}^\dagger = \varepsilon_\D([r_1, r_2])\).

    Finally, we show that \({\varepsilon_\D}^{-1}\) is dilatory. By everything above, the functor \({\varepsilon_\D}^{-1}\) maps each morphism in \(\D\) to the relation in \(\Coisom(\D)\) represented by its dilator. Let \((r_1, r_2)\) be a dilator of a morphism \(f\) in \(\D\). Then \({\varepsilon_\D}^{-1}(f) = [r_1, r_2]\). Also, \(({\varepsilon_\D}^{-1}(r_1), {\varepsilon_\D}^{-1}(r_2)) = ([1, r_1], [1, r_2])\), by \cref{p:dilator-coisometry}. However \(([1, r_1], [1, r_2])\) is a dilator of \([r_1, r_2]\) by \cref{p:rel-dilatory}.

    It is easy to verify the naturality of \(\varepsilon\).
\end{proof}

Putting everything together, we obtain the following theorem.

\begin{theorem}\label{t:equivalence}
The functors
\[
    \begin{tikzcd}[column sep=large, cramped]
    \RegInd
        \ar[r, "\Rel", shift left=2]
        \ar[r, left adjoint]
        \&
    \DilDag
        \ar[l, "\Coisom", shift left=2]
    \end{tikzcd}
\]
form an adjoint equivalence.
\end{theorem}

\begin{proof}
The unit \(\eta\) and counit \(\varepsilon\) of the adjoint equivalence were constructed in \cref{p:regularish-is-coisom,p:dilatory-is-rel}. It remains to verify the triangle identities. As \(\eta\) and \(\varepsilon\) are isomorphisms, it suffices to verify only one of them. It is easy to check that \[\Coisom(\varepsilon_\D) \circ \eta_{\Coisom(\D)} = 1_{\Coisom(\D)}. \qedhere\]
\end{proof}

\begin{example}\label{x:finprob:couplings}
    The category \(\Rel(\Coisom(\FinProb))\) is the category of \textit{couplings} or \textit{transport plans} between finite probability spaces. The fact that it is isomorphic to \(\FinProb\), which follows from \cref{p:dilatory-is-rel}, is well known (see, e.g., \cite[Sec.~2]{dahlqvist2018borel} and \cite[Prop.~13.9]{fritz:synthetic}). (Note that restricting to strictly positive probability measures, in this finite context, is equivalent to quotienting the morphisms under almost sure equality. See also \Cref{s:probability-spaces}.)
\end{example}

\subsection{Strict 2-equivalence of 2-categories}\label{s:two-equivalence}

Let us now augment the categories \(\RegInd\) and \(\DilDag\) with 2-cells, and the functors \(\Coisom\) and \(\Rel\) with 2-cell maps, so that the adjoint equivalence in \cref{t:equivalence} becomes a strict 2-adjoint 2-equivalence (i.e., an adjoint equivalence of \(\Cat\)\nobreakdash-enriched categories).
The reader uninterested in 2-category theory may safely skip this subsection.

\begin{definition}
	A natural transformation \(\beta \colon G \Rightarrow G' \colon \D \to \D'\) into a dagger category \(\D'\) is \textit{coisometric} if all of its components are coisometric, that is, for each object \(X\) of \(\D\), the morphism \(\beta_X \colon GX \to G'X\) in \(\D'\) is coisometric. We will also refer to coisometric natural transformations as \textit{natural coisometries}.
\end{definition}

Given dagger categories \(\D\) and \(\D'\), the dagger functors \(\D \to \D'\) and the natural transformations between them form a dagger category (see, e.g., \cite[Ex.~2.7]{heunenkarvonen:limits}). The natural coisometries \(\D \to \D'\) are merely the coisometries in this dagger category.

We take as the 2-cells of $\DilDag$ the natural coisometries. It is easy to check that these natural transformations are closed under vertical and horizontal composition (for horizontal composition, observe that dagger functors preserve coisometries). We have thus extended \(\DilDag\) to a 2-category.

\begin{definition}
    A natural transformation \(\alpha \colon F \Rightarrow F' \colon \C \to \C'\)
    into an independence category \(\C'\) is \textit{independent} if all of its naturality squares are independent, that is, for all morphisms \(f \colon X \to Y\) in \(\C\), the square
	\[
	\begin{tikzcd}
		FX
            \ar[r, "Ff"]
            \ar[d, "\alpha_X" swap]
            \ar[dr, independent, start anchor={south east}, end anchor={north west}]
            \&
        FY
            \ar[d, "\alpha_Y"]
        \\
		F'X
            \ar[r, "F'f" swap]
            \&
        F'Y 
	\end{tikzcd}
	\]
    in \(\C'\) is independent.
\end{definition}

The independent squares in an independence category $\C$ are in bijection with the independent natural transformations to $\C$ from the \textit{walking morphism} (the category with two objects and a single nontrivial morphism between them).

We take as the 2-cells of $\RegInd$ the independent natural transformations. It is easy to check that these natural transformations are closed under vertical and horizontal composition. We have thus extended \(\RegInd\) to a 2-category.

\begin{lemma}
    For each 2-cell \(\beta \colon G \Rightarrow G' \colon \D \to \D'\) in \(\DilDag\), the components of \(\beta\) are the components of an independent natural transformation \[\Coisom(\beta) \colon \Coisom(G)\Rightarrow \Coisom(G').\]
\end{lemma}

\begin{proof}
The reason why we require that each 2-cell \(\beta \colon G \Rightarrow G' \colon \D \to \D'\) of \(\DilDag\) be coisometric should now be clear: for \(\Coisom(\beta)\) to be well-defined, its components, which are the components of \(\beta\), must be morphisms of \(\Coisom(\D')\).

Now, for each morphism $f \colon X \to Y$ in $\Coisom(\D)$, the square
\[
    \begin{tikzcd}
        GX
            \ar[r, "Gf"]
            \ar[d, "\beta_X" swap]
            \ar[dr, independent, start anchor={south east}, end anchor={north west}]
            \&
        GY
            \ar[d, "\beta_Y"]
        \\
        G'X
            \ar[r, "G'f" swap]
            \&
        G'Y 
    \end{tikzcd}
\]
in \(\Coisom(\D)\), which is the naturality square of \(\Coisom(\beta)\) for \(f\), is independent. Indeed, since $G$ and $G'$ are dagger functors, and \(\beta\) is natural,
\[
    \beta_X (Gf)^\dagger = \beta_X (Gf^\dagger) = (G'f^\dagger) \beta_Y = (G'f)^\dagger \beta_Y. \qedhere
\]
\end{proof}

The fact that \(\Coisom\) is horizontally and vertically functorial on 2-cells is easy to check. We have thus extended $\Coisom$ to a 2-functor $\DilDag \to \RegInd$.

\begin{lemma}
For each 2-cell \(\alpha \colon F \Rightarrow F' \colon \C \to \C'\) in \(\RegInd\), the morphisms
\[[1, \alpha_X] \colon FX \to F'X\]
for each object \(X\) of \(\C\) are the components of a natural coisometry
\[\Rel(\alpha) \colon \Rel(F) \Rightarrow \Rel(F').\]
\end{lemma}

\begin{proof}
    The components of $\Rel(\alpha)$ are coisometric by \cref{p:regularish-is-coisom}. It remains to show that \(\Rel(\alpha)\) is natural. Consider a morphism \([R, r_1, r_2] \colon X \to Y\) in \(\Rel(\C)\). Let \((p_1, p_2)\) be an independent pullback of \((\alpha_X, F'r_1)\), and let \((e, m_1, m_2)\) be a (strong epic, jointly monic) factorisation of \((1p_1, (F'r_2)p_2)\), as depicted in \cref{f:nat-coisom}. 
    \begin{diagram}
        \centering
        \begin{tikzcd}[cramped, row sep={between origins,1.7em}, column sep={between origins,4em}, labels={inner sep=0.3ex}]
            \&
            \&
            \pob
                \arrow[dddl, "Fr_1" swap, shift right=1.5, shorten <=-1.5ex]
                \arrow[dddr, "\alpha_R", shift left=1.5, shorten <=-1.5ex]
                \arrow[d, "f"{inner sep=0.5ex, pos=0.7}, shorten >=-1ex, shorten <=-1.5ex]
                \arrow[ddrr, "Fr_2", shift left=3, shorten <=-0.5ex, shorten >=1.5ex]
            \&
            \&
            \&
        \\[0.8em]
            \&
            \&
            \pob
                \arrow[ddl, "p_1"]
                \arrow[ddr, "p_2" swap]
                \arrow[ddd, "e"{inner sep=0.5ex,pos=0.6}, shorten >=-1ex]
            \&
            \&
            \&
        \\
            \&
            \&
            \&
            \&
            \pob
                \arrow[dddr, "\alpha_Y", shift left=1.2, shorten <=-3.5ex, shorten >=0.5ex]
            \&
        \\
            \&
            \pob
                \arrow[ddl, "1" swap]
                \arrow[ddr, "\alpha_X"{pos=0.3}]
            \&
            \&
            \pob
                \arrow[ddl, "F'r_1"{swap, pos=0.3}]
                \arrow[ddrr, "F'r_2"{pos=0.3}, shift left=0.6, shorten >=0.5ex]
            \&
            \&
        \\[-0.2em]
            \&
            \&
            \pob
            \arrow[dll, "m_1"{pos=0.6}, shift left, crossing over]
            \&
            \&
        \\[0.2em]
            \pob
            \&
            \&
            \pob
            \&
            \&
            \&
            \pob
            \arrow[from=ulll, "m_2"{swap, pos=0.6}, shift right, crossing over, shorten >=0.5ex]
        \end{tikzcd}
        \caption{}
        \label{f:nat-coisom}
    \end{diagram}
    Then \([F'r_1, F'r_2][1,\alpha_X] = [m_1, m_2]\). The naturality square of \(\alpha\) for \(r_1\) is independent, so universality of independent pullbacks yields a unique morphism \(f\) in~\(\C\), as depicted in \cref{f:nat-coisom}, such that \(p_1f = Fr_1\) and \(p_2f = \alpha_R\). Then \((ef, m_1, m_2)\) is a (strong epic, jointly monic) factorisation of \(((Fr_1)1, \alpha_Y (Fr_2))\), also using the naturality square of \(\alpha\) for \(r_2\). However \((1, Fr_2)\) is an independent pullback of \((Fr_2, 1)\) by \cref{p:ind-pullback-identity}. As such, it follows that
    \([1, \alpha_Y][Fr_1, Fr_2] = [m_1, m_2] = [F'r_1, F'r_2][1,\alpha_X]\).
\end{proof}

The proof that \(\Rel\) is vertically and horizontally functorial on 2-cells is routine (use \cref{p:coisom-rel}~(ii)). We have thus extended $\Rel$ to a 2-functor $\RegInd \to \DilDag$.

\begin{theorem}\label{p:2-equivalence}
The 2-functors
\[
    \begin{tikzcd}[column sep=large, cramped]
    \RegInd
        \ar[r, "\Rel", shift left=2]
        \ar[r, left adjoint]
        \&
    \DilDag
        \ar[l, "\Coisom", shift left=2]
    \end{tikzcd}
\]
form a strict 2-adjoint 2-equivalence.
\end{theorem}

\begin{proof}
    Building on \cref{t:equivalence}, it remains to check that \(\eta\) and \(\varepsilon\) are in fact strict 2-natural transformations. For each 2-cell \(\alpha \colon F \Rightarrow F' \colon \C \to \C'\) in \(\RegInd\) and each object \(X\) in \(\C\),
    \[
        \eta_{\C'}(\alpha_X) = [1, \alpha_X] = \Rel(\alpha)_X = \Coisom(\Rel(\alpha))_{\eta_{\C}(X)},
    \]
    so \(\eta\) is a strict 2-natural transformation. Also, for each 2-cell \(\beta \colon G \Rightarrow G' \colon \D \to \D'\) in \(\DilDag\) and each object \(X\) in \(\D\),
    \[
        \varepsilon_{\D'}(\Rel(\Coisom(\beta))_X) = \varepsilon_{\D'}([1, \Coisom(\beta)_X]) = \varepsilon_{\D'}([1, \beta_X]) = \beta_X = \beta_{\varepsilon_\D(X)},
    \]
    so \(\varepsilon\) is also a strict 2-natural transformation.\todo{Inconsistent level of detail here. We brush over the proofs that \(\Rel\) and \(\Coisom\) are 2-functorial, but we have included the proofs that \(\eta\) and \(\epsilon\) are 2-natural.}
\end{proof}

\section{Multivalued surjections in regular categories}
\label{s:bitotal-relations}

Our first running example was about the dilatory dagger category \(\MSurj\) of sets and surjective multivalued functions. We saw in \cref{x:bitotal:coisom} that its wide subcategory of coisometries, which is an epi-regular independence category, is isomorphic to the category \(\Surj\) of sets and surjective functions. Observing that surjective functions are precisely the regular epimorphisms in \(\Set\), 
it is natural to wonder what conditions on a category \(\C\) ensure that its wide subcategory of regular epimorphisms is an epi-regular independence category. The following proposition says that it is sufficient for \(\C\) to be a regular category (see \cref{d:regular-cat}).

\begin{proposition}\label{p:regular-epi-example}
	If \(\C\) is a regular category, then its wide subcategory \(\RegEpi(\C)\) of regular epimorphisms is an epi-regular independence category, where a square
    \[
        \begin{tikzcd}[cramped, sep=large]
            \pob
                \arrow[r, "g"]
                \arrow[d, "f" swap]
                \&
            \pob
                \arrow[d, "v"]
                \\
            \pob
                \arrow[r, "u" swap]
                \&
            \pob
        \end{tikzcd}
    \]
    is declared independent when it is a \emph{regular pushout}~\cite[Def.~1.2]{bourn:denormalized}, that is, when
    \begin{enumerate}
        \item it is commutative, and
        \item the comparison morphism from the span \((f, g)\) to the pullback in \(\C\) of the cospan \((u, v)\) is a regular epimorphism in \(\C\).
    \end{enumerate}
\end{proposition}

\begin{proof}
    A morphism in \(\C\) is regular epic if (similarly to \cref{p:strong-is-regular}) and only if (\cref{p:reg-is-strong}~(iv)) it is strong epic (see \cref{d:strong-epi}). Since strong epimorphisms (in any category) are closed under composition (\cref{p:strong-epic-props}), the regular epimorphisms in \(\C\) do indeed form a wide subcategory of \(\C\).

    To show that \(\RegEpi(\C)\) is an independence category, let us verify axiom~\cref{a:independence:3}; the other axioms are easy to check. Let
    \[
        \begin{tikzcd}[cramped, sep=large]
            \pob
                \arrow[dr, independent]
                \arrow[r, "a"]
                \arrow[d, "f" swap]
                \&
            \pob
                \arrow[d, "g"]
            \\
            \pob
                \arrow[r, "u" swap]
                \&
            \pob
        \end{tikzcd}
        \qquad\text{and}\qquad
        \begin{tikzcd}[cramped, sep=large]
            \pob
                \arrow[dr, independent]
                \arrow[r, "b"]
                \arrow[d, "g" swap]
                \&
            \pob
                \arrow[d, "h"]
            \\
            \pob
                \arrow[r, "v" swap]
                \&
            \pob
        \end{tikzcd}
    \]
    be independent squares in \(\RegEpi(\C)\). We must show that the composite rectangle $(f,h,ba,vu)$ is independent. Condition (i) is obvious. For condition (ii), let \((p_1, p_2)\), \((q_1, q_2)\) and \((r_1, r_2)\) be pullbacks in \(\C\), respectively, of \((v, h)\), \((u, p_1)\) and \((u, g)\). Universality of these pullbacks yields unique morphisms \(k\), \(\ell\) and \(m\) in \(\C\) such that \cref{f:reg-cat}
    \begin{diagram}
        \centering
        \begin{tikzcd}[cramped]
            \pob
                \arrow[drrr, "a", shift left=1.5]
                \arrow[ddddrr, "f"{inner sep=0.3ex}, swap, out=-90, in=150, looseness=1.1, end anchor={west}]
                \arrow[dr, "m"{pos=0.6, inner sep=0.3ex}]
            \\
            \&
            \pob 
                \arrow[rr, "r_2"{pos=0.4}]
                \arrow[dddr, "r_1"{pos=0.6, inner sep=0.3ex}, out=-90, in=135, end anchor={north west}]
                \arrow[dr, "\ell"{pos=0.6, inner sep=0.3ex}]
            \& \&
            \pob 
                \arrow[drrr, "b", shift left=1.5]
                \arrow[dddr, "g"{swap, pos=0.6, inner sep=0.3ex}, out=-90, in=135]
                \arrow[dr, "k"{pos=0.6, inner sep=0.3ex}]
            \\
            \& \&
            \pob 
                \arrow[rr, "q_2"{swap, pos=0.3}, crossing over]
                \arrow[dd, "q_1"]
            \& \&
            \pob
                \arrow[rr, "p_2" swap]
                \arrow[dd, "p_1"]
            \& \&
            \pob
                \arrow[dd, "h"]
            \\ \\
            \& \&
            \pob 
                \arrow[rr, "u" swap]
            \& \&
            \pob
                \arrow[rr, "v" swap]
            \& \&
            \pob
        \end{tikzcd}
        \caption{}
        \label{f:reg-cat}
    \end{diagram}
    is commutative. By the pullback lemma, the span \((q_1, p_2q_2)\) is a pullback of the cospan \((vu, h)\), and the morphism \(\ell m\) is the comparison morphism from the span \((f, ba)\) to this pullback. Our goal is thus to prove that \(\ell m\) is a regular epimorphism in \(\C\). First $m$ is a regular epimorphism in \(\C\) because the square $(f,g,a,u)$ is independent. Similarly $k$ is a regular epimorphism in \(\C\) because the square $(g,h,b,v)$ is independent. By the pullback lemma, \(\ell\) is a pullback of \(k\) along \(q_2\). However regular epimorphisms are pullback stable, so \(\ell\) is also a regular epimorphism in \(\C\). As \(\ell m\) is a composite of regular epimorphisms, it is itself a regular epimorphism, as required.

	To show that $\RegEpi(\C)$ is epi-regular, we will use the alternative characterisation in \cref{d:epi-regular-alt}. For axiom~\cref{a:pullback-2}, consider a cospan \((u, v)\) in \(\RegEpi(\C)\). It has a pullback \((p_1, p_2)\) in \(\C\). As regular epimorphisms in \(\C\) are pullback stable, the span \((p_1, p_2)\) actually comes from \(\RegEpi(\C)\). It is easy to check that the resulting square is an independent pullback.
    
    For axiom~\cref{a:coeq-kern-pair}, consider a morphism \(f\) in \(\RegEpi(\C)\). We need to show that \(f\) remains regular epic when viewed as a morphism in \(\RegEpi(\C)\). Now \(f\) has a kernel pair \((k_1, k_2)\) in \(\C\). As regular epimorphisms in \(\C\) are pullback stable, the span \((k_1, k_2)\) actually comes from \(\RegEpi(\C)\). As \(f\) is regular epic in~\(\C\), it is the coequaliser in \(\C\) of its kernel pair \((k_1, k_2)\). It suffices to show that \(f\) is also the coequaliser in \(\RegEpi(\C)\) of \((k_1, k_2)\). This follows from the fact that in any category, if \(uv\) is regular epic and \(v\) is epic then \(u\) is regular epic~\cite[Prop.~7.62~(3)]{adamek:joy-of-cats}.

	For axiom~\cref{a:fact-2}, consider a span \((f, g)\) in \(\RegEpi(\C)\). It has a (regular epic, jointly monic) factorisation \((e, m_1, m_2)\) in \(\C\). As \(m_1e = f\) is regular epic, it follows that \(m_1\) is regular epic. Similarly \(m_2\) is regular epic. The span \((m_1, m_2)\) clearly remains jointly monic in \(\RegEpi(\C)\). Hence \((e, m_1, m_2)\) is a (regular epic, jointly monic) factorisation of \((f, g)\) in \(\RegEpi(\C)\).
\end{proof}

There are many well-known examples of regular categories. The following one is of particular interest to computer scientists.

\begin{example}
Fix a countably infinite set \(\A\). Let \(\Perm(\A)\) denote the group of permutations of \(\A\). A \textit{support} of an element \(x\) of a left \(\Perm(\A)\)-set \(X\) is a subset \(S \subseteq \A\) such that every permutation of \(\A\) that fixes the elements of \(S\) also fixes \(x\). A \textit{nominal set}~\cite{gabbay1999new} (see also \cite{pitts2013nominal}) is a left \(\Perm(\A)\)-set whose elements all have a finite support. Each element \(x\) of a nominal set \(X\) has a least finite support, which is denoted \(\supp x\). Nominal sets and \(\Perm(\A)\)-equivariant maps form a category \(\Nom\). It is well known (see, e.g., \cite[Sec.~6.3]{pitts2013nominal}) that \(\Nom\) is equivalent to the \textit{Schanuel topos}. Since \(\Nom\) is a topos, it is, in particular, a regular category. A morphism in \(\Nom\) is regular epic if and only if it is surjective. By \cref{p:regular-epi-example}, the category \(\RegEpi(\Nom)\) of Nominal sets and surjective equivariant maps is canonically an epi-regular independence category. 
\end{example}

\subsection{Regular independence categories}
\label{s:nom}

Call an independence category \textit{regular} if
\begin{enumerate}[label=(RI\arabic*)]
    \item 
     \label{a:strong-ind-pull}
    every cospan has a strong independent pullback,
    \item every span has a (strong epic, jointly monic) factorisation, and
    \item 
    \label{a:epic-ind-pull}
    strong epimorphisms are stable under independent pullback.
\end{enumerate}
Clearly every epi-regular independence category is a regular independence category. Also every regular category is a regular independence category where the independent squares are merely the commutative ones.

\begin{example}
The category \(\Nom\) is of course a regular independence category where the independent squares are the commutative ones. Curiously, there is a second system of independent squares in \(\Nom\) that it a regular independence category in a different way.

A square 
\[
\begin{tikzcd}[cramped, sep=large]
    X
    \arrow[r, "g"]
    \arrow[d, "f" swap]
    \&
    A
    \arrow[d, "v"]
    \\
    B
    \arrow[r, "u" swap]
    \&
    C
\end{tikzcd}
\]
in \(\Nom\) is called \emph{relatively separated} if it is commutative and
\[\supp(fx) \cap \supp(gx) = \supp(hx)\]
for all \(x \in X\), where \(h = uf = vg\).
Declaring a square independent if it is relatively separated makes \(\Nom\) into an independence category (see~\cite[Ex.~2.4]{simpson:category-theoretic-structure-independence}). 

Every cospan
\[
    \begin{tikzcd}[cramped]
        A
            \arrow[r, "u"]
            \&
        C
            \&
        B
            \arrow[l, "v" swap]
    \end{tikzcd}
\]
in \(\Nom\) has a canonical independent pullback (see~\cite[Ex.~6.3]{simpson:category-theoretic-structure-independence}), namely, 
the \textit{relatively separated product} 
\[
    A \otimes_C B = \setb[\big]{(a, b) \in A \times_C B}{\supp(a) \cap \supp(b) = \supp(c) \;\;\text{where}\;\; c = ua = vb}
\]
together with the restrictions of the coordinate projection functions. It is easy to check that these are in fact \textit{strong} independent pullbacks, so axiom (RI1) holds.

Axiom (RI2) is the same as axiom \cref{a:reg:fact}, which we already know holds.

For axiom (RI3), consider a cospan 
\[
    \begin{tikzcd}[cramped]
        A
            \arrow[r, "u"]
            \&
        C
            \&
        B
            \arrow[l, "v" swap]
    \end{tikzcd}
\]
in \(\Nom\) where $u$ is surjective. Let \(b \in B\), and let $c = vb$. As \(u\) is surjective, there exists $a \in A$ such that $ua = c$. Build a permutation $\pi$ of \(\A\) that fixes $\supp c$ and maps $\supp a \setminus \supp c$ to some set $S$ disjoint from $\supp b$. Then \((\pi \cdot a, b) \in A \times_C B\) because
\[
    u(\pi \cdot a) = \pi \cdot ua = \pi \cdot c = c = vb.
\]
Now \(\supp c \subseteq \supp a\)~\cite[Lem.~2.12~(i)]{pitts2013nominal}, so~\cite[Prop.~2.10]{pitts2013nominal}
\begin{align*}
    \supp (\pi \cdot a)
    = \pi \cdot \supp a
    = \pi \cdot \paren[\big]{(\supp a \setminus \supp c) \sqcup \supp c}
    = S \sqcup \supp c.
\end{align*}
But \(\supp c \subseteq \supp b\)~\cite[Lem.~2.12~(i)]{pitts2013nominal}, so
\[
    \supp (\pi \cdot a) \cap \supp b = (S \sqcup \supp c) \cap \supp b = \emptyset \sqcup \supp c = \supp c.
\]
Hence \((\pi \cdot a, b) \in A \otimes_B C\).
\end{example}

The following proposition is a generalisation of \cref{p:regular-epi-example} for regular independence categories. 

\begin{proposition}
\label{p:prop-reg-reg-epi-ind}
If \(\C\) is a regular independence category, then \(\RegEpi(\C)\) is an epi-regular independence category where a square in \(\RegEpi(\C)\) is declared independent exactly when it is independent in \(\C\) and the comparison morphism to the independent pullback is a regular epimorphism.
\end{proposition}

First, we generalise \cref{p:descent}.

\begin{lemma}
\label{p:gen-descent}
In a regular independence category, if \(h\) is strong epic then
\[
    \begin{tikzcd}[cramped]
        \pob
            \arrow[ddrr, independent]
            \arrow[d, "h" swap]
            \arrow[r, "h"]
            \&
        \pob
            \arrow[r, "g"]
            \&
        \pob
            \arrow[dd, "v"]
        \\
        \pob
            \arrow[d, "f" swap]
            \&
            \&
        \\
        \pob
            \arrow[rr, "u" swap]
            \&
            \&
        \pob
    \end{tikzcd}
    \qquad\text{implies}\qquad
    \begin{tikzcd}[cramped, sep=large]
        \pob
            \arrow[r, "g"]
            \arrow[d, "f" swap]
            \arrow[dr, independent]
            \&
        \pob
            \arrow[d, "v"]
        \\
        \pob
            \arrow[r, "u" swap]
            \&
        \pob
    \end{tikzcd}
    .
\]
\end{lemma}

\begin{proof}
The proof of \cref{p:descent} works here essentially unchanged. 
The use of axiom \cref{a:epic} is replaced by the outright assumption that \(h\) is strong epic.
\end{proof}

Now we generalise the part of \cref{p:monos} that we need.

\begin{lemma}
\label{l:strong-epi-invertible}
In a regular independence category, if a strong epimorphism \(m\) has an independent kernel pair \((j, j)\), then \(m\) is invertible.
\end{lemma}

\begin{proof}
By axiom \cref{a:epic-ind-pull}, \(j\) is strong epic. Hence, by \cref{p:gen-descent}, the span \((1, 1)\) forms an independent square with the cospan \((m_1, m_1)\). Consider two morphisms \(u\) and \(v\) satisfying \(m_1u = m_1v\). The pasting \cref{e:pasting-kernel-eq} is independent, so there is a unique morphism \(w\) such that \(jw = u\) and \(jw = v\). In particular, \(u = jw = v\). Hence \(m\) is monic. But it is also strong epic, so it is invertible.
\end{proof}

Next, we generalise \cref{p:strong-is-regular}.

\begin{lemma}
\label{p:new-strong-is-reg}
In a regular independence category, every strong epimorphism is the coequaliser of its independent kernel pair.
\end{lemma}

\begin{proof}
The proof of \cref{p:strong-is-regular} works here \textit{mutatis mutandis}. By axiom~\cref{a:strong-ind-pull}, all independent pullbacks are strong. Use this fact directly (instead of \cref{p:pull-pasting}) to obtain the span \((r_1, r_2)\). Next, \(p_2\) and \(r_1\) are independent pullbacks of \(e\), which is strong epic. By axiom \cref{a:epic-ind-pull} (instead of axiom \cref{a:epic}), they are themselves also strong epic. Finally, \(m_1\) is invertible by \cref{l:strong-epi-invertible} (instead of \cref{p:monos}).
\end{proof}

Finally, we prove \cref{p:prop-reg-reg-epi-ind}.

\begin{proof}[Proof of \cref{p:prop-reg-reg-epi-ind}]
The proof of \cref{p:regular-epi-example} works here \textit{mutatis mutandis}. In particular, replace all pullbacks and kernel pairs with independent pullbacks and independent kernel pairs. Also replace the uses of the pullback lemma by appeals to the strength of the independent pullbacks. The coincidence of the class of regular epimorphisms with the class of strong epimorphisms is now by \cref{p:reg-is-strong,p:new-strong-is-reg}. When verifying axiom \cref{a:coeq-kern-pair}, the fact that \(f\) is the coequaliser in \(\C\) of its independent kernel pair \((k_1, k_2)\) follows from \cref{p:new-strong-is-reg}.
\end{proof}

\begin{example}
Since \(\Nom\) is an independence category in two different ways, the relations in \(\RegEpi(\Nom)\) have two different compositions: one uses pullbacks while the other uses relatively separated products.
\end{example}

Having seen that many of the results in the present article hold not only for epi-regular independence categories, but, more generally (with minor modifications), for regular independence categories, the reader may wonder why we did not focus on the more general notion of regular independence category from the beginning. Unfortunately, we (the authors) do not know how to generalise \cref{p:2-equivalence}, which is our main theorem. In particular, given a regular independence category \(\C\), we do not know how to characterise the morphisms in the poset-enriched dagger category \(\Rel(\C)\) that come from~\(\C\). They are, in general, neither the coisometries nor the maps. It may turn out in the future that a different notion is more deserving of the name \textit{regular independence category}.

\section{Partial injections in rm-adhesive and restriction categories}

Our second running example was about the dilatory dagger category \(\PInj\) of sets and injective partial functions.

\subsection{Rm-adhesive categories}
\label{s:adhesive}

We saw in \cref{x:partial-injections:isom} that its wide subcategory of isometries, which is a mono-coregular co-independence category, is isomorphic to the category \(\Inj\) of sets and injective functions. Observing that injective functions are precisely the regular monomorphisms in \(\Set\), it is natural to wonder what conditions on a category \(\C\) ensure that its wide subcategory of regular monomorphisms is a mono-coregular co-independence category. 
\cref{p:rm-adhesive} below says that it is sufficient for \(\C\) to be \textit{rm-adhesive}. A category \(\C\) is \textit{rm-adhesive} (or \textit{quasiadhesive}) if and only if it has pushouts along regular monomorphisms, these pushouts are pullback stable and are pullbacks, and regular subobjects are closed under binary union~\cite[Thm.~B]{garner:adhesive}.

\begin{proposition}
\label{p:rm-adhesive}
If \(\C\) is an rm-adhesive category, then its wide subcategory \(\RegMono(\C)\) of regular monomorphisms is a mono-coregular co-independence category, where a square is declared co-independent if it is a pullback square.
\end{proposition}

\begin{proof}
First we must show that \(\RegMono(\C)\) is a co-independence category. The self-dual axioms \cref{a:independence:1,a:independence:4} hold trivially, while the self-dual axioms \cref{a:independence:2,a:independence:3} and the dual of axiom \cref{a:independence:5} are well-known properties of pullbacks. It remains to show that \(\RegMono(\C)\) is mono-coregular. 

For the dual of axiom \cref{a:pullback-2}, consider a span \((u, v)\) in \(\RegMono(\C)\). As \(\C\) is rm-adhesive and \(u\) and \(v\) are regular monic in \(\C\), there is a pushout \((i_1, i_2)\) in \(\C\) of \((u, v)\), the morphisms \(i_1\) and \(i_2\) are regular monic, and the pushout square is also a pullback square. The square is thus a co-independent square in \(\RegMono(\C)\). To see that it is a co-independent pushout, use the fact that regular subobjects are closed under binary union.

For the dual of axiom \cref{a:coeq-kern-pair}, let \(f\) be a morphism in \(\RegMono(\C)\). Then \(f\) has a cokernel pair \((q_1, q_2)\). But pushout squares in rm-adhesive categories are pullback squares, so \((f,f)\) is a pullback in \(\C\) of \((q_1, q_2)\). As \(f\) is monic, it follows that \(f\) is actually an equaliser in \(\C\) of \((q_1, q_2)\). To see that it remains an equaliser of \((q_1, q_2)\) in \(\RegMono(\C)\), use the fact that if \(mn\) is regular monic and \(m\) is monic, then \(n\) is regular monic~\cite[Prop.~7.62~(3)]{adamek:joy-of-cats}.

For the dual of axiom \cref{a:fact-2}, consider a cospan \((f, g)\) in \(\RegMono(\C)\). Let \((p_1, p_2)\) be its pullback in \(\C\). As regular monomorphisms (in any category) are pullback stable, the morphisms \(p_1\) and \(p_2\) are regular monic, and so the resulting square in \(\RegMono(\C)\) is actually a co-independent square. Let \((i_1, i_2)\) be a co-independent pushout of \((p_1, p_2)\). By universality of co-independent pushouts, there is a unique morphism \(e\) in \(\RegMono(\C)\) such that \(ei_1 = f\) and \(ei_2 = g\). As \((i_1, i_2)\) is a co-independent pushout, it is jointly epic by the dual of \cref{p:ind-pull-is-joint-monic}. Hence \((e, i_1, i_2)\) is a (regular monic, jointly epic) factorisation of the cospan \((f, g)\).
\end{proof}

\begin{example}
Let us consider an example from computer science. Fix a set~\(V\). A \emph{$V$\nobreakdash-valued heap}~\cite[Ex.~2.5]{simpson:category-theoretic-structure-independence} is a pair \((L,v)\) consisting of a finite set \(L\) and a function \(v \colon L \to V\). Think of \(V\) as the set of \textit{values} that could be stored in the memory of a computer, and of \(L\) as a set of \textit{locations} in the memory where values are stored. The function \(v\) encodes the values stored in each memory location. A \emph{heap embedding} \(f \colon (L,v) \to (L',v')\) is an injective function \(f \colon L \to L'\) satisfying \(v'f = v\). Heap embeddings model allocation of additional memory. Obviously $V$\nobreakdash-valued heaps and embeddings form a category $\Heap_V$. In particular, this category is a mono-coregular co-independence category (see~\cite[Prop.~7.5]{simpson:category-theoretic-structure-independence}), where a square in \(\Heap_V\) is declared co-independent if the underlying square in \(\Inj\) is co-independent (see \cref{f:partial-injections:ind}). This fact can instead be deduced from the previous proposition. Indeed, consider the inclusion functor \(\cat{F} \colon \FinSet \to \Set\). The regular monomorphisms in the comma category \(\cat{F} \mathbin{\downarrow} V\) are precisely the morphisms in \(\cat{F} \mathbin{\downarrow} V\) that lie over regular monomorphisms in \(\FinSet\), which are merely injective functions. Hence \(\Heap_V\) is merely another name for the category \(\RegMono(\cat{F} \mathbin{\downarrow} V)\). As \(\FinSet\) is rm-adhesive and \(\cat{F}\) preserves pushouts, the category \(\cat{F} \mathbin{\downarrow} V\) is also rm-adhesive (similarly to~\cite[Prop.~6.6~(ii)]{lack:adhesive}). It follows that \(\Heap_V = \RegMono(\cat{F} \mathbin{\downarrow} V)\) is canonically a mono-coregular co-independence category by \cref{p:rm-adhesive}.\todo{Summarise the following into a single sentence capturing the significance of this example: As a slight generalisation of the injections example, we briefly consider an example that was important for the development of independence categories~\cite{simpson:category-theoretic-structure-independence}.
It comes from the computer science notion of \emph{heap}, which describes dynamically allocated memory: in program verification, it is important to make sure that separate processes do not access the same memory. This is studied using special logics called \textit{separation logics} (e.g. \cite{o2001local}). Recently, the co-independent pushout structure of heaps has been used in \cite{kammar2017monad} and there is interest in generalising separation logics to other notions of independence and separation such as conditional independence~\cite{simpson:equivalence-conditional-independence,li2023lilac,li2024nominal}.}
\end{example}

\subsection{Restriction categories}
\label{s:partial-isomorphisms}

There is another perspective on \(\PInj\) worth exploring: it is the wide subcategory of restriction isomorphisms in the restriction category \(\Par\) of sets and partial functions (see \cref{x:partial-injections:dagger}). To give meaning to this observation, let us briefly review the basics of restriction categories.

Restriction categories are a categorical foundation for understanding partiality. A \textit{restriction category} \cite[Sec.~2.1.1]{cockett2002restriction} is a category equipped with a choice of morphism \(\overbar{f} \colon A \to A\) for each morphism \(f \colon A \to B\) such that
\begin{align*}
    f \overbar{f} = f && \overbar{f} \overbar{g} = \overbar{g} \overbar{f} && \overbar{g\overbar{f}} = \overbar{g} \overbar{f} && \overbar{g}f = f \overbar{gf}
\end{align*}
whenever these equations make sense. The morphism $\overbar{f}$ is called the \textit{restriction} of~$f$. A morphism \(f\) in a restriction category is called \textit{total} if \(\overbar{f} = 1\).

One can think of the morphisms in a restriction category as being partially defined: the morphism $\overbar{f}$ encodes the domain of definition of $f$. For example, the category \(\Par\) is a restriction category where the restriction \(\overbar{f}\) of a partial function \(f \colon A \to B\) is the partial function \(A \to A\) represented by the inclusion function \(\supp f \hookrightarrow A\), that is, it is a partial variant of the identity function on \(A\) that is only defined on \(\supp f\).

A \textit{restriction inverse} of a morphism \(f \colon A \to B\) is a morphism of \(g \colon B \to A\) such that \(gf = \overbar{f}\) and \(fg = \overbar{g}\). A morphism \(f\) is called a \textit{restriction isomorphism} if it has a restriction inverse \cite[Sec.~2.3.1]{cockett2002restriction}, in which case it has a unique restriction inverse~\cite[Lem.~2.18~(vii)]{cockett2002restriction}. If \(\C\) is a restriction category, then its restriction isomorphisms form a wide subcategory \(\ResIso(\C)\) of \(\C\). The category \(\ResIso(\C)\) is canonically a dagger category: for each morphism \(f\), the morphism \(f^\dagger\) is defined to be the unique restriction inverse of \(f\). For all morphisms \(f\) and \(g\) in \(\ResIso(\C)\), the equations $ff^\dagger f = f$ and $ff^\dagger gg^\dagger = g g^\dagger ff^\dagger$ hold~\cite[Sec.~2.3.2]{cockett2002restriction}; dagger categories like this are called \textit{inverse categories}.

The restriction isomorphisms in \(\Par\) are precisely the \textit{injective} partial functions; in other words, the dilatory dagger category \(\PInj\) is merely \(\ResIso(\Par)\). Given this observation, it is natural to wonder more generally what conditions on a restriction category \(\C\) guarantee that \(\ResIso(\C)\) is dilatory. \cref{p:res-cat} below says that it is sufficient for \(\C\) to have \textit{restriction pushouts} and \textit{binary joins}.

\textit{Restriction pushouts} are the restriction analogue of (ordinary) pushouts: they are simply pushouts whose injections are total~\cite[Sec.~2]{cockett2007restriction}. 

\textit{Join} is just another name for \textit{supremum} or \textit{least upper bound}. In a restriction category, every hom-set is canonically partially ordered~\cite[Sec.~2.1.4]{cockett2002restriction}, with \(f \leq g\) if \(f = g\overbar{f}\). Intuitively, think of $f \leq g$ as saying that wherever $f$ is defined, $g$ is also defined and is equal to $f$. Parallel morphisms \(f\) and \(g\) are called \textit{compatible} if \(f \overbar{g} = g\overbar{f}\) \cite[Prop.~6.3]{cockett2009boolean}. Intuitively, this says that wherever \(f\) and \(g\) are both defined, they are equal. We say that a restriction category has \textit{binary joins} \cite[Defs.~6.7 and~10.1]{cockett2009boolean} if every compatible pair of morphisms \(f\) and \(g\) has a join \(f \lor g\), and joins are preserved by composition, that is, $v(f \lor g)u = vfu \lor vgu$ for all morphisms \(u\) and \(v\) for which the composites are defined.

Since codilators are defined in terms of isometries, it will be helpful to better understand the isometric restriction isomorphisms. A morphism \(s\colon A \to B\) in a restriction category is called a \textit{restriction section} \cite[Sec.~3.1]{cockett2003restriction} if there exists a morphism \(r \colon B \to A\) such that \(rs = 1\) and \(sr = \overbar{r}\). In this case, the morphism \(s\) is monic, and thus total~\cite[Lem.~2.1~(vi)]{cockett2002restriction}. Hence $\overbar{s} = 1 = rs$. So \(s\) is a restriction isomorphism with restriction inverse \(r\). It follows that a morphism in a restriction category is a restriction section if and only if it is an isometric restriction isomorphism.

\begin{proposition}
    \label{p:res-cat}
	If \(\C\) is a restriction category with restriction pushouts and binary joins, then the dagger category \(\ResIso(\C)\) of restriction isomorphisms in \(\C\) is dilatory.
    
    In particular, the codilator of a restriction isomorphism \(f\) is the restriction pushout \((i_1, i_2)\) in \(\C\) of the span \((\overbar{f}, f)\): every codilation \((s_1, s_2)\) of \(f\) forms a commutative square with the span \((\overbar{f}, f)\), and the unique morphism \(s\) such that \(si_1 = s_1\) and \(si_2 = s_2\) is a restriction section with restriction inverse \(s^\dagger = i_1 {s_1}^\dagger \lor i_2 {s_2}^\dagger\), as depicted in \cref{f:res-dilator}.
    \begin{diagram}
        \centering
        \begin{tikzcd}[sep={0}, cramped]
            \pob
                \arrow[r, "f"]
                \arrow[d, "\overbar{f}" swap]
                \&[4.8em]
            \pob
                \arrow[d, "i_2" swap, shift right]
                \arrow[ddr, "s_2", out=-30, in=90, looseness=1, shorten <=-0.25ex, shorten >=0.25ex]
                \&[1.5em]
            \\[3.6em]
            \pob
                \arrow[r, "i_1", shift left]
                \arrow[rrd, "s_1" swap, out=-60, in=180, looseness=0.8, shorten <=-1ex, shorten >=0.5ex]
                \&
            \pob
                \arrow[dr, "s"{pos=0.1}, shorten >=0.5ex, shorten <=-1ex]
                \&
            \\[1em]
                \&
                \&
            \pob
        \end{tikzcd}
        \caption{}
        \label{f:res-dilator}
    \end{diagram}
\end{proposition}
\begin{proof}
	First, we show that the cospan \((i_1, i_2)\) is a codilation of \(f\), that is, that \(i_1\) and \(i_2\) are restriction sections and \({i_2}^\dagger i_1 = f\). The cospan \((1, f^\dagger)\) forms a commutative square with the span \((\overbar{f}, f)\), so there is a unique morphism \(p_1\) such that \(p_1i_1 = 1\) and \(p_1i_2 = f^\dagger\). Additionally, note that
    \begin{align*}
        \overbar{p_1} i_1
	    &= i_1 \overbar{p_1i_1}
	    = i_1
	    = i_1p_1i_1
    \\ \shortintertext{and}
        \overbar{p_1} i_2
	    &= i_2 \overbar{p_1i_2}
	    = i_2 \overbar{f^\dagger}
	    = i_2 f f^\dagger
	    = i_1 f^\dagger f f^\dagger
	    = i_1 f^\dagger
	    = i_1 p_1 i_2\text.
    \end{align*}
	Since the cospan \((i_1, i_2)\) is jointly epic, \(i_1p_1 = \overbar{p_1}\). Therefore \(i_1\) is a restriction section with ${i_1}^\dagger = p_1$. On the other hand, the cospan \((f, 1)\) also forms a commutative square with the span \((\overbar{f}, f)\), so there is a unique morphism \(p_2\) such that \(p_2i_1 = f\) and \(p_2 i_2 = 1\). Then 
    \begin{align*}
        \overbar{p_2} i_1
	    &= i_1 \overbar{p_2i_1}
	    = i_1 \overbar{f}
	    = i_2 f
	    = i_2p_2i_1\text,
        \\ \shortintertext{and}
        \overbar{p_2} i_2
	    &= i_2 \overbar{p_2i_2}
	    = i_2 \overbar{1}
	    = i_2
	    = i_2p_2i_2\text.
    \end{align*}
	Since the cospan \((i_1, i_2)\) is jointly epic, we get that \(i_2 p_2 = \overbar{p_2}\). So \(i_2\) is also a restriction section with ${i_2}^\dagger = p_2$. As such, also \(f = p_2 i_1 = {i_2}^\dagger i_1\). So \((i_1, i_2)\) is indeed a codilation of \(f\). 

	Next, let \((s_1, s_2)\) be a codilation of \(f\), meaning that
	\begin{align*}
	    {s_1}^\dagger s_1 &= 1, &
        {s_2}^\dagger s_1 &= f,&
	    s_1{s_1}^\dagger &= \overbar{{t_1}^\dagger},
	    \\
	    {s_2}^\dagger s_2 &= 1,&
        &\text{and}&
	    s_2 {s_2}^\dagger &= \overbar{{s_2}^\dagger}\text.
	\end{align*}
	Now the cospan \((s_1, s_2)\) forms a commutative square with the span \((\overbar{f}, f)\) because
	\begin{align*}
	    s_1 \overbar{f}
	    = s_1 f^\dagger f
	    &= s_1 f^\dagger ff^\dagger ff^\dagger f
	    = (s_1 f^\dagger f {s_1}^\dagger)(s_2 ff^\dagger {s_2}^\dagger) s_2f
	    \\ &= (s_2 ff^\dagger {s_2}^\dagger)(s_1 f^\dagger f {s_1}^\dagger)s_2f
	    = s_2 ff^\dagger ff^\dagger ff^\dagger f
	    = s_2f\text.
	\end{align*}
	Hence there is a unique morphism \(s\) such that \(si_1 = s_1\) and \(si_2 = s_2\). It remains to show that \(s\) is restriction section. By the definition of a restriction pushout, $i_1$ and $i_2$ are total. Recall that if $u$ is total then $\overbar{uv} = \overbar{v}$ \cite[Lem.~2.1~(iii)~and~(vi)]{cockett2002restriction}. Hence $\overbar{i_1 {s_1}^\dagger} =  \overbar{{s_1}^\dagger} = s_1 {s_1}^\dagger$ and $\overbar{i_2 {s_2}^\dagger} = \overbar{{s_2}^\dagger} = s_2 {s_2}^\dagger$. It follows that \(i_1 {s_1}^\dagger\) and \(i_2 {s_2}^\dagger\) are compatible. Indeed
	\begin{align*}
	    i_2 {s_2}^\dagger \overbar{i_1 {s_1}^\dagger}
	    &= i_2 {s_2}^\dagger \overbar{{s_1}^\dagger}
	    = i_2 {s_2}^\dagger s_1 {s_1}^\dagger
	    = i_2 f {s_1}^\dagger
	    = i_1 \overbar{f} {s_1}^\dagger
	    = i_1 (s_1 \overbar{f})^\dagger
	    \\ &= i_1 (s_2 f)^\dagger
	    = i_1 f^\dagger {s_2}^\dagger
	    = i_1 {s_1}^\dagger s_2 {s_2}^\dagger
	    = i_1 {s_1}^\dagger \overbar{{s_2}^\dagger}
	    = i_1 {s_1}^\dagger \overbar{i_2 {s_2}^\dagger}\text.
	\end{align*}
	Hence the join \(i_1 {s_1}^\dagger \lor i_2 {s_2}^\dagger\) exists. Our goal is now to show that $(i_1 {s_1}^\dagger \lor i_2 {s_2}^\dagger)s = 1$. Observe that if $u \leq v$, then $u$ and $v$ are compatible and $u \lor v = v$. As every restriction $\overbar{u}$ satisfies $\overbar{u} \leq 1$, it follows that $1 \lor \overbar{u} = 1$. Using this fact, and the fact that joins preserve composition, we compute that 
	\begin{align*}
	    (i_1 {s_1}^\dagger \lor i_2 {s_2}^\dagger)si_1
	    &= (i_1 {s_1}^\dagger \lor i_2 {s_2}^\dagger)s_1
	    = i_1 {s_1}^\dagger s_1 \lor i_2 {s_2}^\dagger s_1
	    \\ &= i_1 \lor i_2 f
	    = i_1 \lor i_2 {i_2}^\dagger i_1
	    = (1 \lor \overbar{{i_2}^\dagger})i_1
	    = i_1
    \\ \shortintertext{and}
	    (i_1 {s_1}^\dagger \lor i_2 {s_2}^\dagger)si_2
	    &= (i_1 {s_1}^\dagger \lor i_2 {s_2}^\dagger)s_2
	    = i_1 {s_1}^\dagger s_2 \lor i_2 {s_2}^\dagger s_2
	    \\ &= i_1 f^\dagger \lor i_2
	    = i_1 {i_1}^\dagger i_2 \lor i_2
	    = (\overbar{{i_1}^\dagger} \lor 1) i_2
	    = i_2\text.
	\end{align*}
	Thus $i_1$ and $i_2$ are jointly epic, and $(i_1 {s_1}^\dagger \lor i_2 {s_2}^\dagger)s = 1$. Finally, since the restriction of the join is the join of the restriction, that is, $\overbar{u \lor v} = \overbar{u} \lor \overbar{v}$ \cite[Lem.~6.6]{cockett2009boolean},  
	\begin{align*}
	    \overbar{i_1 {s_1}^\dagger \lor i_2 {s_2}^\dagger}
	    &= \overbar{i_1 {s_1}^\dagger} \lor \overbar{i_2 {s_2}^\dagger}
	    = \overbar{{s_1}^\dagger} \lor \overbar{{s_2}^\dagger}
	    \\ &= s_1 {s_1}^\dagger \lor s_2 {s_2}^\dagger 
	    = si_1 {s_1}^\dagger \lor s i_2 {s_2}^\dagger
	    = s(i_1 {s_1}^\dagger \lor i_2 {s_2}^\dagger)\text.
	\end{align*}
	Thus $s(i_1 {s_1}^\dagger \lor i_2 {s_2}^\dagger) = \overbar{i_1 {s_1}^\dagger \lor i_2 {s_2}^\dagger}$. So \(s\) is a restriction section, as desired, and moreover \(s^\dagger = i_1 {s_1}^\dagger \lor i_2 {s_2}^\dagger\).
\end{proof}

\begin{example} 
	Let $\CRing_\bullet$ be the category of commutative rings and \textit{non-unital} ring homomorphisms, that is, functions $f \colon R \to S$ such that \(f(r + s) = f(r) + f(s)\) and \(f(rs) = f(r) f(s)\), but possibly not \(f(1) = 1\). Then $\CRing_\bullet$ is a corestriction category, where the corestriction $\overbar{f} \colon S \to S$ of a non-unital ring homomorphism $f \colon R \to S$ is defined as $\overbar{f}(s) = f(1) s$. A corestriction isomorphism in $\CRing_\bullet$ is a non-unital ring homomorphism $f \colon R \to S$ for which there is another non-unital ring homomorphism $f^\dagger \colon S \to R$ such that $f(f^\dagger(s))= f(1) s$ and $f^\dagger(f(r))= f^\dagger(1) r$. The corestriction category $\CRing_\bullet$ has corestriction pullbacks and binary meets (the order does not reverse when dualising), so its wide subcategory of corestriction isomorphisms is canonically a dilatory dagger category. In particular, the dilator of a corestriction isomorphism $f \colon R \to S$ is
    \[R \boxplus_f S = \setb[\big]{(r,s)}{r \in R, s\in S, f(r) = f(1) s}\text.\]
\end{example}

\section{Probability spaces and Markov categories}
\label{s:probability-spaces}

Our third running example was about the dilatory dagger category \(\FinProb\) of finite probability spaces and measure-preserving stochastic matrices. It generalises, almost verbatim, to standard Borel probability spaces and measure-preserving Markov kernels, as we shall see in \cref{s:standard-borel}. In fact, it generalises further to the setting of Markov categories with conditionals, as we will see in \cref{s:general-probability-spaces}.

\subsection{Standard Borel probability spaces}
\label{s:standard-borel}

A \textit{measurable space} is a set \(A\) equipped with a \(\sigma\)-algebra~\(\Sigma_A\). Given measurable spaces \(A\) and \(B\), a \textit{Markov kernel} \(r \colon A \to B\) is a function \(r(-|-) \colon \Sigma_B \times A \to \Reals\) such that
\begin{enumerate}
    \item for all \(V \in \Sigma_B\), the function \(r(V|-) \colon A \to \Reals\) is measurable; and,
    \item for all \(a \in A\), the function \(r(-|a) \colon \Sigma_B \to \Reals\) is a probability measure on \(B\).
\end{enumerate}
The composite of Markov kernels \(r \colon A \to B\) and \(s \colon B \to C\) is defined, for all \(a \in A\) and \(W \in \Sigma_C\), by the \textit{Chapman--Kolmogorov formula}
\[
    (sr)(W|a) \,=\, \int_B s(W|y)\,r(dy|a).
\]

A \textit{probability space} is a pair \((A, \Pr_A)\) where \(A\) is a measurable space and \(\Pr_A\) is a probability measure. Given probability spaces \((A, \Pr_A)\) and \((B, \Pr_B)\), a \textit{measure-preserving Markov kernel} \(r \colon (A, \Pr_A) \to (B, \Pr_B)\) is a Markov kernel \(r \colon A \to B\) such that for all \(V \in \Sigma_B\),
\[
    \int_A r(V|x)\,\Pr_A(dx) \,=\, \Pr_B(V).
\]

A \textit{Bayesian inverse} of a measure-preserving Markov kernel \(r \colon (A, \Pr_A) \to (B, \Pr_B)\) is a measure-preserving Markov kernel \(s \colon (B, \Pr_B) \to (A, \Pr_A)\) such that
\[
    \int_U r(V|x) \, \Pr_A(dx) \,=\, \int_V s(U|y)\, \Pr_B(dy)
\]
for all \(U \in \Sigma_A\) and \(V \in \Sigma_B\). A measurable space \(A\) is called \textit{standard Borel} if \(\Sigma_A\) is the Borel \(\sigma\)-algebra of a completely metrisable topology on \(A\), and a probability space \((A, \Pr_A)\) is called \textit{standard Borel} if \(A\) is standard Borel. Every measure-preserving Markov kernel between standard Borel probability spaces has a Bayesian inverse~\cite[Prop.~3.13]{perrone-vanbelle:martingales}. 

Two measure-preserving Markov kernels \(h, h' \colon (B, \Pr_B) \to (A, \Pr_A)\) are called \textit{almost surely equal} if for all \(U \in \Sigma_A\), there exists \(V \in \Sigma_B\) such that \(\Pr_B(V) = 1\) and \(h(U|b) = h'(U|b)\) for all \(b \in V\). If \(h\) and \(h'\) are both Bayesian inverses of a given measure-preserving Markov kernel, then \(h\) is almost surely equal to \(h'\). The quotient \(\SBProb\) of the category of standard Borel probability spaces and measure-preserving Markov kernels by almost sure equality is a dagger category where the dagger of a morphism is its Bayesian inverse~\cite[Thm.~2.10]{dahlqvist2018borel}. The category \(\FinProb\), which was defined in \cref{x:finprob:dagger}, is equivalent to the full subcategory of  \(\SBProb\) spanned by the finite probability spaces.

The characterisation in \cref{x:finprob:coisom} of the coisometries in \(\FinProb\) as the deterministic maps generalises as follows. A measure-preserving Markov kernel \(f \colon (A, \Pr_A) \to (B, \Pr_B)\) is called \textit{almost surely deterministic} if
\[
    \Pr_A\paren[\big]{\setb[\big]{a \in A}{f(V|a) \in \set{0, 1}}} = 1
\]
for all \(V \in \Sigma_B\). The almost surely deterministic morphisms in \(\SBProb\) form a wide subcategory \(\SBProbDet\). A morphism in \(\SBProb\) is coisometric if and only if it is almost surely deterministic~\cite[Prop.~2.5]{ensarguet2023ergodic}. In other words, \(\Coisom(\SBProb) = \SBProbDet\).

There is another characterisation of the coisometries in \(\SBProb\) in terms of random variables. A \textit{random variable} on a probability space \((\Omega, \Pr)\) valued in a measurable space \(A\) is a measurable function \(X \colon \Omega \to A\). The \textit{pushforward} of \(\Pr\) along \(X\) is the probability measure \(X_* \Pr\) on \(A\) defined by
\[
    (X_* \Pr)(U) = \Pr(X^{-1}U)
\]
for all \(U \in \Sigma_A\). Also, the \textit{delta kernel} of \(X\) is the almost surely deterministic measure-preserving Markov kernel \(\delta_X \colon (\Omega, \Pr) \to (A, X_* \Pr)\) defined by
\[
    \delta_X(U|\omega)
    =
    \begin{cases}
        1 &\text{if \(X(\omega) \in U\), and}\\
        0 &\text{otherwise,}
    \end{cases}
\]
for all \(\omega \in \Omega\) and \(U \in \Sigma_A\). The mapping of a random variable to its delta kernel is \textit{retrofunctorial}; this means that \(\delta_1 = 1\), and that \(\delta_{fX} = \delta_f \delta_X\) for all random variables \(X\) and measurable functions \(f\) for which the composite \(fX\) is defined. 

Every almost surely deterministic morphism between \textit{standard Borel} probability spaces is a delta kernel of some random variable. Indeed, consider a representative \(f \colon (\Omega, P) \to (A, P_A)\) of an almost surely deterministic morphism in \(\Prob\). As \(A\) is standard Borel, \(\Sigma_A\) is countably generated. Hence, since \(f\) is almost surely deterministic, there exists \(E \in \Sigma_\Omega\) with \(P(E) = 1\) such that \(f(U|\omega) \in \set{0, 1}\) for all \(\omega \in \Omega\) and all \(U \in \Sigma_A\). The restriction of \(f\) to \(E\) is a delta kernel~\cite[Ex.~10.5]{fritz:synthetic}. Redefining \(f\) to be some arbitrary delta kernel on the complement of \(E\) transforms \(f\) into a delta kernel that is almost surely equal to \(f\).

In \cref{x:finprob:dilator}, we saw that the dagger category \(\FinProb\) is dilatory. We will now show that the larger dagger category \(\SBProb\) is similarly dilatory.

The \textit{product} of two measurable spaces \(A\) and \(B\) is their Cartesian product \(A \times B\) equipped with the \(\sigma\)-algebra generated by all subsets of \(A \times B\) of the form \(U \times V\) where \(U \in \Sigma_A\) and \(V \in \Sigma_V\). Together with the coordinate projection functions \(p_1 \colon A \times B \to A\) and \(p_2 \colon A \times B \to B\), this is a product of \(A\) and \(B\) in the category of measurable spaces and measurable functions.

Consider a measure-preserving Markov kernel \(r \colon (A, \Pr_A) \to (B, \Pr_B)\). The \textit{product} (or \textit{semidirect product}) of \(\Pr_A\) and \(r\) is the probability measure \(\Pr_A \otimes r\) on \(A \times B\) defined, for all \(E \in \Sigma_{A \times B}\), by
\[
    (\Pr_A \otimes r)(E) \,=\, \int_A r(E_x|x)\,\Pr_A(dx),
\]
where \(E_a = \setb[\big]{b \in B}{(a, b) \in E}\) is in \(\Sigma_B\) for each \(a \in A\). It is~\cite[Cor.~14.26]{klenke:probability-theory} the unique probability measure \(\mu\) on \(A \times B\) such that
\[
    \mu(U \times V) \,=\, \int_U r(V|x)\, \Pr_A(dx)
\]
for all \(U \in \Sigma_A\) and \(V \in \Sigma_B\). Observe that \((p_1)_*(\Pr_A \otimes r) = \Pr_A\) because
\[
    (p_1)_*(\Pr_A \otimes r)(U)
    = (\Pr_A \otimes r)(U \times B)
    = \int_{U} r(B|x)\,\Pr_A(dx)
    = \Pr_A(U)
\]
for all \(U \in \Sigma_A\), and that \((p_2)_*(\Pr_A \otimes r) = \Pr_B\) because
\[
    (p_2)_*(\Pr_A \otimes r)(V)
    = (\Pr_A \otimes r)(A \times V)
    = \int_A r(V|x) \, \Pr_A(dx)
    = \Pr_B(V)
\]
for all \(V \in \Sigma_B\).

\begin{proposition}
\label{p:krn:dilator}
The dagger category \(\SBProb\) is dilatory.

In particular, for each morphism \(r \colon (A, \Pr_A) \to (B, \Pr_B)\), the span
\[\begin{tikzcd}[cramped]
    (A, \Pr_A) \&
    (A \times B, \Pr_A \otimes r)
        \arrow[l, "\delta_{p_1}" swap]
        \arrow[r, "\delta_{p_2}"]\&
    (B, \Pr_B)
\end{tikzcd}\]
of coisometries is a dilator of \(r\).
\end{proposition}

\begin{lemma}
\label{l:bloom}
In the notation of \cref{p:krn:dilator}, for all \(a \in A\) and \(E \in \Sigma_{A \times B}\),
\[
{\delta_{p_1}}^\dagger (E|a)
= r(E_a|a).
\]
\end{lemma}

\begin{proof}
By \cite[Thm.~14.25]{klenke:probability-theory}, the equation
\[
    k(E|a) = \int_A \delta_1(dx|a) \int_B r(dy|x) \, \ind_E(x, y)
\]
defines a Markov kernel \(k \colon A \to A \times B\), where \(\ind_E\) is the \textit{indicator function} of \(E\). For all \(a \in A\) and \(E \in \Sigma_{A \times B}\),
\[
    k(E|a) = \int_B r(dy|a)\, \ind_E(a, y) = \int_{E_a} r(dy|a) = r(E_a|a),
\]
so \(k\) is a measure preserving Markov kernel from \((A, \Pr_A)\) to  \((A \times B, \Pr_A \otimes r)\). Finally,
\begin{align*}
    \int_E \delta_{p_1}(U|z)\,(\Pr_A \otimes r)(dz)
    &= (\Pr_A \otimes r)\paren[\big]{E \cap (U \times B)}\\
    &= \int_A r\paren[\big]{\paren[\big]{E \cap (U \times B)}_x|x}\, \Pr_A(dx)\\
    &= \int_U r(E_x|x)\, \Pr_A(dx)\\
    &= \int_U k(E|x)\, \Pr_A(dx)
\end{align*}
for all \(E \in \Sigma_{A \times B}\) and all \(U \in \Sigma_A\), so \(k^\dagger = \delta_{p_1}\).
\end{proof}

\begin{proof}[Proof of \cref{p:krn:dilator}]
First, for all \(a \in A\) and \(V \in \Sigma_B\),
\[
    (\delta_{p_2} {\delta_{p_1}}^\dagger)(V|a)
    = \int_{A \times B} \delta_{p_2}(V|z)\, {\delta_{p_1}}^\dagger(dz|a)
    = {\delta_{p_1}}^\dagger(A \times V|a)
    = r(V|a)
\]
by \cref{l:bloom}, so \((\delta_{p_1}, \delta_{p_2})\) is a dilation of \(r\).

Consider another dilation of \(r\). It is necessarily of the form
\[
    \begin{tikzcd}[cramped]
        (A, \Pr_A) \&
        (\Omega, \Pr)
            \arrow[l, "\delta_{X}" swap]
            \arrow[r, "\delta_{Y}"]\&
        (B, \Pr_B)
    \end{tikzcd}
\]
where \(X\) and \(Y\) are random variables on \((\Omega, \Pr)\) valued, respectively, in \(A\) and \(B\) such that \(\Pr_A = X_*\Pr\) and \(\Pr_B = Y_* \Pr\). Let \(\ang{X, Y} \colon \Omega \to A \times B\) denote the product pairing of \(X\) and \(Y\) in the category of measurable spaces. Then
\begin{align*}
    \int_{U} (\delta_Y {\delta_X}^\dagger)(V|x)\, \Pr_A(dx)
    &=\, \int_{x \in U} \int_{\omega \in \Omega} \delta_Y(V|\omega)\, {\delta_X}^\dagger (d\omega |x)\, \Pr_A(dx)\\
    &=\, \int_U {\delta_X}^\dagger (Y^{-1} V|x)\, \Pr_A(dx)\\
    &=\, \int_{Y^{-1} V} \delta_X(U|\omega) \,\Pr(d\omega)\\
    &=\, \Pr(X^{-1} U \mathrel{\cap} Y^{-1} V)\\
    &=\, (\ang{X, Y}_* \Pr)(U \times V)\\
    \intertext{for all \(U \in \Sigma_A\) and all \(V \in \Sigma_B\). But \(\delta_Y {\delta_X}^\dagger = r\), so}
    \int_{U} (\delta_Y {\delta_X}^\dagger)(V|x)\, \Pr_A(dx) 
    \,&=\, \int_{U} r(V|x)\, \Pr_A(dx)
    \,=\, (\Pr_A \otimes r)(U \times V)
\end{align*}
for all \(U \in \Sigma_A\) and all \(V \in \Sigma_B\). It follows that \(\Pr_A \otimes r = \ang{X, Y}_* \Pr\).

By retrofunctoriality of delta kernels, the diagram
\[
    \begin{tikzcd}[row sep=large]
        \&
        (\Omega, \Pr)
            \arrow[dl, "\delta_{X}" swap]
            \arrow[dr, "\delta_{Y}"]
            \arrow[d, "\delta_{\ang{X, Y}}"]
        \&
    \\
        (A, \Pr_A) \&
        (A \times B, \Pr_A \otimes r)
            \arrow[l, "\delta_{p_1}"]
            \arrow[r, "\delta_{p_2}" swap]
            \&
        (B, \Pr_B)
    \end{tikzcd}
\]
is commutative. For uniqueness, consider a coisometry \(f \colon (\Omega, \Pr) \to (A \times B, \Pr_A \otimes r)\) such that \(\delta_{p_1}f = \delta_X\) and \(\delta_{p_2}f = \delta_Y\). As \(\Omega\) and \(A \times B\) are standard Borel, we know that \(f = \delta_Z\) for some random variable \(Z\) on \((\Omega, \Pr)\) valued in \(A \times B\) such that \(Z_*\Pr = \Pr_A \otimes r\). Now \(p_1 Z\) is almost surely equal (in the standard sense for measurable functions) to \(X\) because \(\delta_{p_1} \delta_Z = \delta_X\)~\cite[Ex.~13.3, countably generated case]{fritz:synthetic}. Similarly \(p_2 Z\) is almost surely equal to \(Y\). Hence \(Z\) is almost surely equal to \(\ang{X, Y}\). Hence \(f = \delta_Z = \delta_{\ang{X,Y}}\).
\end{proof}

Consider two random variables \(X\) and \(Y\) on a probability space \((\Omega, \Pr)\) valued, respectively, in measurable spaces \(A\) and \(B\). For each \(V \in \Sigma_B\), there exists, by the Radon--Nikodým theorem, a measurable function \(\CPr[Y \mathbin{\in} V|X \mathbin{=} {-}] \colon A \to \Reals\) such that
\[
    \int_U \CPr[Y \mathbin{\in} V|X \mathbin{=} x]\, (X_*\Pr)(dx)
    \,=\, \Pr(X^{-1}U \mathbin{\cap} Y^{-1}V)
\]
for all \(U \in \Sigma_A\). Also, every two such measurable functions \(A \to \Reals\) are equal to each other \((X_*P)\)-almost everywhere. Unfortunately, for fixed \(a \in A\), the function \(\CPr[Y \mathbin{\in} {-}|X \mathbin{=} a] \colon \Sigma_B \to \Reals\) is not necessarily a probability distribution.

A \textit{regular conditional distribution} of \(Y\) given \(X\) is (see also \cite[Def.~8.28]{klenke:probability-theory}) a Markov kernel \(r \colon A \to B\) such that
\[
    \int_U r(V|x)\, (X_*\Pr)(dx)
    \,=\,
    \Pr(X^{-1}U \mathbin{\cap} Y^{-1}V)
\]
for all \(V \in \Sigma_B\), that is, such that
\[
    r(V|a)\,=\,\CPr[Y \mathbin{\in} V|X \mathbin{=} a]
\]
for all \(V \in \Sigma_B\) and \((X_*\Pr)\)-almost all \(a \in A\). Assuming that \(\delta_X\) has a Bayesian inverse, the Markov kernel \(\delta_Y {\delta_X}^\dagger\) is a regular conditional distribution of \(Y\) given \(X\). Indeed,
\[
    (\delta_Y {\delta_X}^\dagger)(V|a)
    \,=\, \int_\Omega \delta_Y(V|\omega) \, {\delta_X}^\dagger(d\omega|a)
    \,=\, {\delta_X}^\dagger(Y^{-1}V|a)
\]
for all \(a \in A\) and \(V \in \Sigma_B\), so
\begin{align*}
    \int_U (\delta_Y {\delta_X}^\dagger)(V|x)\, (X_*\Pr)(dx)
    \,&=\, \int_U {\delta_X}^\dagger(Y^{-1}V|x)\, (X_*\Pr)(dx)
    \\&=\, \int_{Y^{-1}V} \delta_X(U|\omega)\, \Pr(d\omega)
    \,=\, \Pr(X^{-1}U \mathbin{\cap} Y^{-1}V).
\end{align*}

The following proposition generalises the characterisation in~\Cref{x:finprob:ind} of independent squares in \(\FinProbDet\).

\begin{proposition}\label{p:krn:ind}
    Consider a commutative diagram
	\[
	\begin{tikzcd}[row sep=large]
		(\Omega, \Pr)
            \ar[d, "\delta_X" swap]
            \ar[r, "\delta_Y"]
            \ar[dr, "\delta_Z"]
            \&
        (B, Y_* \Pr)
            \ar[d, "\delta_g"]
        \\
		(A, X_* \Pr)
            \ar[r, "\delta_f" swap]
            \&
        (C, Z_* \Pr)
	\end{tikzcd}
	\]
    in \(\SBProbDet\). The outer square is independent if and only if the random variables \(X\) and \(Y\) are conditionally independent given the random variable \(Z\).
\end{proposition}

\begin{proof}
The outer square is independent exactly when  \(\delta_Y {\delta_X}^\dagger = {\delta_g}^\dagger \delta_f\) almost surely. Now, for all \(V \in \Sigma_B\) and all \(a \in A\),
\[
    ({\delta_g}^\dagger \delta_f)(V|a)
    \,=\, \int_C {\delta_g}^\dagger(V|z)\, \delta_f(dz|a)
    \,=\, {\delta_g}^\dagger(V|fa)
\]
Also, \({\delta_g}^\dagger = \delta_Y {\delta_Y}^\dagger {\delta_g}^\dagger = \delta_Y {\delta_Z}^\dagger\) almost surely. Hence the outer square is independent if and only if, for all \(V \in \Sigma_B\) and \((X_*P)\)-almost all \(a \in A\),
\[
    \CPr[Y \mathbin{\in} V|X \mathbin{=} a]
    \,=\, \CPr[Y \mathbin{\in} V|Z \mathbin{=} fa].
\]
This is one way~\cite[Def.~12.20]{klenke:probability-theory} to say that \(X\) and \(Y\) are conditionally independent given \(Z\).
\end{proof}

Finally, a \emph{coupling} of (or \emph{transport plan} between) two probability spaces \((A,\Pr_A)\) and \((B,\Pr_B)\) is usually defined~\cite[Def.~1.1]{villani} in one of two ways:
\begin{enumerate}
    \item as a pair of random variables \(X\) and \(Y\) on some probability space \((\Omega, \Pr)\) valued, respectively, in the measurable spaces \(A\) and \(B\), such that \(\Pr_A = X_*\Pr\) and \(\Pr_B = Y_* \Pr\); or,
    \item as a probability measure \(\Pr\) on the measurable space \(A \times B\) such that \(\Pr_A = (p_1)_* \Pr\) and \(\Pr_B = (p_2)_* \Pr\).
\end{enumerate}
A coupling of the first kind is essentially the same as a \textit{span} in \(\SBProbDet\), while one of the second kind is essentially equivalent to a \textit{relation} in \(\SBProbDet\). For a probability theorist, the two definitions above are more-or-less interchangeable; from the perspective of category theory, this is because every every span in \(\SBProbDet\) has a (strong epic, jointly monic) factorisation. The usual composition of couplings of the first kind~\cite[Ch.~1, Gluing Lemma]{villani} corresponds to span composition via independent pullback, while the usual composition of couplings of the second kind~\cite[Lem.~7.6]{villani-older} corresponds to composition of relations as defined in \cref{p:rel-exists}. Our main theorem, \cref{p:2-equivalence}, says that the dilatory dagger category \(\Rel(\SBProbDet)\) of standard Borel probability spaces and couplings is canonically isomorphic to the dilatory dagger category \(\SBProb\) of standard Borel probability spaces and measure-preserving Markov kernels.

\subsection{Markov categories with conditionals}
\label{s:general-probability-spaces}

\emph{Markov categories}~\cite[Def.~2.1]{fritz:synthetic} (see also~\cite[Def.~2.3]{chojacobs2019strings}) are a framework for synthetic probability theory modelled on the category of measurable spaces and Markov kernels. The present subsection outlines the Markov-categorical generalisation of the story in the previous subsection.

If a Markov category \(\C\) has \textit{conditionals}~\cite[Def.~11.5]{fritz:synthetic}, then the \textit{probability spaces} in \(\C\) and their \textit{measure-preserving morphisms}~\cite[Def.~13.8]{fritz:synthetic}, quotiented by \textit{almost sure equality}~\cite[Def.~13.1]{fritz:synthetic}, form~\cite[Prop.~13.9]{fritz:synthetic} a category \(\Prob(\C)\). Also every morphism in \(\Prob(\C)\) has a \textit{Bayesian inverse}~\cite[Rem.~13.10]{fritz:synthetic}, so \(\Prob(\C)\) is canonically a dagger category. This construction is a generalisation of the construction of \(\SBProb\) in \cref{s:standard-borel}. Indeed, the standard Borel measurable spaces and their Markov kernels form a Markov category \(\SBMeas\) with conditionals, and \(\SBProb \cong \Prob(\SBMeas)\).

A morphism in \(\Prob(\C)\) is coisometric exactly when~\cite[Prop.~2.5]{ensarguet2023ergodic} it is \textit{almost surely deterministic}~\cite[Def.~13.11]{fritz:synthetic} as a morphism in \(\C\).

\begin{proposition}
    \label{p:prob-dilatory}
	Let \(\C\) be a Markov category with conditionals. The dagger category \(\Prob(\C)\) is dilatory.
\end{proposition}

\begin{proof}[Proof sketch]
    Let \(\otimes\) denote the monoidal product of \(\C\). For all objects \(A\) and \(B\) of \(\C\), let \(\pi_1\) and \(\pi_2\) denote, respectively, the canonical projections \(A \otimes B \to A\) and \(A \otimes B \to B\). For all morphisms \(r \colon C \to A\) and \(s \colon C \to B\) in \(\C\), let \(\ang{r,s}\) denote the composite of the copy map \(C \to C \otimes C\) with \(r \otimes s \colon C \otimes C \to A \otimes B\).

    Fix a morphism \(r \colon (A, \Pr_A) \to (B, \Pr_B)\) in \(\Prob(\C)\). Since
    \[
        \pi_1 \ang{1,r}\Pr_A = \Pr_A
        \qquad\text{and}\qquad
        \pi_2 \ang{1,r}\Pr_A = r\Pr_A = \Pr_B,
    \]
    the morphisms \(\pi_1\) and \(\pi_2\) form a span
    \[
        \begin{tikzcd}[cramped]
            (A, \Pr_A)
            \&
            \paren[\big]{A \otimes B, \ang{1, r}\Pr_A}
                \arrow[l, "\pi_1" swap]
                \arrow[r, "\pi_2"]
            \&
            (B, \Pr_B)
        \end{tikzcd}
    \]
    in \(\Prob(\C)\). Its legs are coisometric because \(\pi_1\) and \(\pi_2\) are deterministic. Also \(\ang{1, r}\) is a Bayesian inverse of \(\pi_1\), so \(\pi_2{\pi_1}^\dagger\) is almost surely equal to \(\pi_2 \ang{1, r} = r\).\footnote{This factorisation of \(r\) is known as the \emph{bloom-shriek factorisation}~\cite[Def.~7.3]{fullwood:information-stochastic} of \(r\).}
    Hence the span is a dilation\footnote{Here, and throughout this document, the term \textit{dilation} refers to the concept in \Cref{d:dilation}. The notion of \textit{dilation} in a Markov category~\cite{fritz2022dilations} is related, but distinct.} of \(r\).

    To see that it is a \textit{terminal} dilation of \(r\), let
    \begin{equation}
        \label{e:Markov:dilation}
        \begin{tikzcd}[cramped]
            (A, \Pr_A)
            \&
            (\Omega, \Pr)
                \arrow[l, "f" swap]
                \arrow[r, "g"]
            \&
            (B, \Pr_B)
        \end{tikzcd}
    \end{equation}
    be another dilation of \(r\). Suppose that there is a coisometry \(h\) that makes the diagram
    \[
        \begin{tikzcd}[row sep=large]
            \&
            (\Omega, \Pr)
                \arrow[dl, "f" swap]
                \arrow[dr, "g"]
                \arrow[d, "h"]
            \&
        \\
            (A, \Pr_A) \&
            (A \times B, \ang{1,r}\Pr_A)
                \arrow[l, "\pi_1"]
                \arrow[r, "\pi_2" swap]
                \&
            (B, \Pr_B)
        \end{tikzcd}
    \]
    almost surely commutative. Then \(h\) is almost surely deterministic, so it almost surely equals \(\ang{\pi_1h,\pi_2h}\)~\cite[Prop~2.54]{fritz2023involutive}, which almost surely equals \(\ang{f,g}\). It remains to prove existence. Clearly \(\pi_1\ang{f,g} = f\) and \(\pi_2\ang{f,g} = g\). By the definition of \(f^\dagger\) and of almost sure equality of \(r\) and \(gf^\dagger\),
    \[
        \ang{f,g}\Pr = \ang{1, gf^\dag}\Pr_A = \ang{1,r}\Pr_A,
    \]
    so \(\ang{f,g}\) is a morphism \((\Omega, \Pr) \to (A \times B, \ang{1, r}\Pr_A)\) in \(\Prob(\C)\). This morphism is coisometric because it is a composite of the copy map, which is deterministic, with \(f \otimes g\), which is coisometric because \(f\) and \(g\) are coisometric.
\end{proof}

The characterisation (\Cref{p:krn:ind}) of independent squares in \(\SBProbDet\) also generalises to \(\Prob(\C)\). Coisometries in \(\Prob(\C)\) can be regarded~\cite[Def.~2.1]{fritz:synthetic} as random variables in \(\C\), and there is a suitable notion of conditional independence~\cite[Def.~12.1]{fritz:synthetic}.

\begin{proposition}
    \label{p:prob-independence}
    Let \(\C\) be a Markov category with conditionals. Consider a commutative diagram
	\[
	\begin{tikzcd}
		(\Omega,\Pr) \ar{r}{g} \ar[swap]{d}{f} \ar{dr}{h} \& (B,\Pr_B) \ar{d}{v} \\
		(A,\Pr_A) \ar[swap]{r}{u} \& (C,\Pr_C)
	\end{tikzcd}
	\]
    of coisometries in \(\Prob(\C)\). The outer square is independent if and only if \(f\) and \(g\) are conditionally independent given \(h\).
\end{proposition}

\begin{proof}
	See \cite[Prop.~9]{stein:RV} (there the term \textit{independent} refers to its Definition~6).
\end{proof}

Finally, the canonical isomorphism \(\Rel(\Coisom(\Prob(\C))) \cong \Prob(\C)\) from \cref{t:equivalence} is a variant of the known isomorphism~\cite[Prop.~13.9~(b)]{fritz:synthetic} of the category of \textit{couplings}~\cite[Rem.~12.10]{fritz:synthetic} in \(\C\) with \(\Prob(\C)\).

\section{Hilbert spaces and contractions}
\label{s:contractions}

A function \(f \colon X \to Y\) between Hilbert spaces \(X\) and \(Y\) is called \textit{adjointable} if there is a function \(g \colon Y \to X\) such that \(\innerProd{y}{fx} = \innerProd{gy}{x}\) for all \(x \in X\) and all \(y \in Y\). There is at most one such function \(g\); it is called the \textit{adjoint} of \(f\) and denoted by \(f^\dagger\). Hilbert spaces and adjointable maps form a dagger category~\(\Hilb\). A morphism \(f \colon X \to Y\) in \(\Hilb\) is isometric if and only if \(\norm{fx} = \norm{x}\) for all \(x \in X\).

A \textit{contraction} \(f \colon X \to Y\) between Hilbert spaces \(X\) and \(Y\) is a linear map \(f \colon X \to Y\) such that \(\norm{fx} \leq \norm{x}\) for all \(x \in X\). All contractions are adjointable, and the adjoint of a contraction is again a contraction. Hence contractions form a wide dagger subcategory \(\Hilb_{\leq 1}\) of \(\Hilb\). The isometries in \(\Hilb_{\leq 1}\) are precisely the isometries in \(\Hilb\), so a cospan in \(\Hilb\) is a codilation in \(\Hilb\) if and only if it is a codilation in \(\Hilb_{\leq 1}\).

\begin{proposition}
The dagger category \(\Hilb_{\leq 1}\) is dilatory.
\end{proposition}

The proof is inspired by the proof that every contraction in a Douglian pre-Hilbert \(*\)-category has a dilator~\cite[Cor.~7.17]{dimeglio:rcategories}.

\begin{proof}
Consider a morphism \(r \colon A \to B\) in \(\Hilb_{\leq 1}\). Since \(r\) is a contraction, the morphism \(1 - r^\dagger r\) is positive. Let \(e \colon A \to D\) be the corestriction of \(\sqrt{1 - r^\dagger r}\) to the closure in \(A\) of its range. Then \(e\) is epic in \(\Hilb\) and \(e^\dagger e = 1 - r^\dagger r\). The cospan
\[
    \begin{tikzcd}[cramped, sep=huge]
        A
            \arrow[r, "{\begin{bsmallmatrix}e\vphantom{0}\\r\vphantom{1}\end{bsmallmatrix}}"]
            \&
        D \oplus B
            \&
        B
            \arrow[l, "{\begin{bsmallmatrix}0\vphantom{e}\\1\vphantom{r}\end{bsmallmatrix}}" swap]
    \end{tikzcd}
\]
is a codilation of \(r\) because
\[
    \begin{bmatrix}
        e & 0\\
        r & 1
    \end{bmatrix}^\dagger
    \begin{bmatrix}
        e & 0\\
        r & 1
    \end{bmatrix}
    = \begin{bmatrix}
        1 & r^\dagger\\
        r & 1
    \end{bmatrix}.
\]
To show that it is a codilator of \(r\), consider another codilation \((X, f, g)\) of \(r\). Then
\[
    \begin{bmatrix}
        f&
        g
    \end{bmatrix}^\dagger
    \begin{bmatrix}
        f &
        g
    \end{bmatrix}
    =
    \begin{bmatrix}
        1 & r^\dagger\\
        r & 1
    \end{bmatrix}
    =
    \begin{bmatrix}
        e & 0\\
        r & 1
    \end{bmatrix}^\dagger
    \begin{bmatrix}
        e & 0\\
        r & 1
    \end{bmatrix}.
\]
By Douglas' lemma~\cite[Thm.~1]{douglas:majorization-factorization-range}, there is an isometry \(h \colon D \oplus B \to X\) such that
\[
    \begin{bmatrix}
        f &
        g
    \end{bmatrix}
    = h
    \begin{bmatrix}
        e & 0\\
        r & 1
    \end{bmatrix},
\]
that is, such that the following diagram is commutative.
\[
    \begin{tikzcd}[sep=large]
        \&
    X
        \&
    \\
    A
        \arrow[r, "{\begin{bsmallmatrix}e\vphantom{0}\\r\vphantom{1}\end{bsmallmatrix}}" swap]
        \arrow[ur, "f"]
        \&
    D \oplus B
        \arrow[u, "h"{pos=0.4}]
        \&
    B
        \arrow[l, "{\begin{bsmallmatrix}0\vphantom{e}\\1\vphantom{r}\end{bsmallmatrix}}"]
        \arrow[ul, "g" swap]
    \end{tikzcd}
\]
Uniqueness of \(h\) follows from the fact that \(\begin{bsmallmatrix}
    e & 0\\
    r & 1
\end{bsmallmatrix}\) is epic in \(\Hilb\).
\end{proof}

It follows that the category \(\Hilb_1\) of Hilbert spaces and isometries is a mono-coregular co-independence category. It is interesting to note that the category \(\Hilb\) is itself a coregular category. Every regular subobject in \(\Hilb\) has an isometric representative, so the category \(\Hilb_1\) is \textit{almost} the same as the wide subcategory of regular monomorphisms in \(\Hilb\). From this perspective, this example is somewhat similar to the dual of the ones discussed in \cref{s:bitotal-relations}.

\end{document}